\documentclass[12pt]{amsart}
\usepackage[latin1]{inputenc}
\usepackage[portuguese,english]{babel}
\usepackage{amscd}
\usepackage{eucal}
\usepackage{enumerate}
\usepackage{amsthm}
\usepackage{amssymb} 
\usepackage{amsmath}
\usepackage[dvips]{graphicx}
\usepackage{psfrag}
\usepackage[all]{xy}
\usepackage{epsfig}
\DeclareMathOperator{\sgn}{sgn}

\newtheorem{Theorem}{Theorem}
\numberwithin{Theorem}{section}
\newtheorem*{theorem*}{Theorem 1}

\newtheorem*{proposition*}{Proposition}
\newtheorem{proposition}[equation]{Proposition}

\newtheorem*{corollary*}{Corollary}
\theoremstyle{definition}
\newtheorem*{definition}{Definition}
\newtheorem*{comments*}{Comments}
\newtheorem{example}{Example}
\newtheorem*{example*}{Example}
\newtheorem{claim}{Claim}

\newtheorem{Remark}[equation]{Remark}
\newtheorem*{remarks*}{Remarks}

\numberwithin{equation}{section}

\gdef\myletter{}

\let\savetheequation\theequation
\def\theequation{\savetheequation\myletter}

\newcommand{\Ker}{{\rm Ker}}

\newcommand{\Q}{{\mathbb Q}}

\newcommand{\Z}{{\mathbb Z}}
\newcommand{\C}{{\mathbb C}}

\renewcommand{\P}{{\mathbb P}}
\newcommand{\R}{{\mathbb R}}

\newcommand{\dint}{\displaystyle\int}  % Integral in displaystyle

\def    \to     {{\longrightarrow}}

\begin{document}
\title[Hyperpolygon spaces and moduli spaces of PHBs]{Hyperpolygon spaces and moduli spaces of parabolic Higgs bundles}
\author{Leonor Godinho}
\address{Departamento de Matem\'{a}tica, Instituto Superior T\'{e}cnico, Av. Rovisco Pais, 1049-001 Lisboa, Portugal}
\email{lgodin@math.ist.utl.pt}
\author{Alessia Mandini}
\address{Centro de An\'{a}lise Matem\'{a}tica, Geometria e Sistemas Din\^{a}micos, Instituto Superior T\'{e}cnico, Av. Rovisco Pais, 1049-001 Lisboa, Portugal}
\email{amandini@math.ist.utl.pt}

\thanks{Both authors were partially supported  by Fundação para a Ciência e a Tecnologia (FCT/Portugal) through program POCI 2010/FEDER and project PTDC/MAT/098936/2008;  
the first author was partially supported by Fundação Calouste Gulbenkian; 
the second author was partially supported by FCT grant SFRH/BPD/44041/2008.}

\maketitle
\begin{abstract}
Given an $n$-tuple of positive real numbers $\alpha$ we consider the hyperpolygon space $X(\alpha)$, the hyperk\"{a}hler quotient analogue to the K\"ahler moduli space of polygons in $\R^3$. We prove the existence of an isomorphism between hyperpolygon spaces and moduli spaces of stable, rank-$2$, holomorphically trivial parabolic Higgs bundles over $\C \P^1$ with fixed determinant and trace-free Higgs field. This isomorphism allows us to prove that  hyperpolygon spaces  $X(\alpha)$ undergo an elementary transformation in the sense of Mukai as $\alpha$ crosses a wall in the space of its admissible values. We describe the changes in the core of   $X(\alpha)$ as a result of this transformation as well as the changes in the nilpotent cone of the corresponding moduli spaces of parabolic Higgs bundles. Moreover, we study the intersection rings of the core components of  $X(\alpha)$. In particular, we find generators of these rings, prove a recursion relation in $n$ for their intersection numbers and use it to obtain explicit formulas for the computation of these numbers. Using our isomorphism, we obtain similar formulas for each connected component of the nilpotent cone of the corresponding moduli spaces of parabolic Higgs bundles thus determining their intersection rings. As a final application of our isomorphism we describe the cohomology ring structure of these moduli spaces of parabolic Higgs bundles and of the  components of their nilpotent cone.
\end{abstract}

\section{Introduction}

In this work we study two families of manifolds: hyperpolygon spaces and  moduli spaces of stable, rank-$2$, holomorphically trivial parabolic Higgs bundles over $\C \P^1$ with fixed determinant and trace free Higgs field, proving the existence of an isomorphism between them. This relationship connecting two different fields of study allows us to benefit from techniques and ideas from each of these areas to obtain new results and insights. In particular, using the study of variation of moduli spaces of parabolic Higgs bundles over a curve, we describe the dependence of hyperpolygon spaces $X(\alpha)$ and their cores on the choice of the parameter $\alpha$. We study the chamber structure on the space of admissible values of $\alpha$ and show that, when a wall is crossed, the hyperpolygon space suffers an elementary transformation in the sense of Mukai. Working on the side of hyperpolygons, we take advantage of the geometric description of the core components of a hyperpolygon space to study their intersection rings. We find homology cycles dual to generators of these rings and prove a recursion relation that allows us to decrease the dimension of the spaces involved. Based on this relation we obtain explicit expressions for the computation of the intersection numbers of the core components of hyperpolygon spaces. Using our isomorphism we can obtain similar formulas for the nilpotent cone components of the moduli space of rank-$2$, holomorphically trivial parabolic Higgs bundles over $\C \P^1$ with fixed determinant and trace-free Higgs field. To better understand these results we begin with a brief overview of the two families of spaces involved.
   
Let $K$ be a compact Lie group acting on a symplectic manifold $(V, \omega)$ with a moment map $\mu : V \to \frak{k}^*$. Then, for an appropriate central value $\alpha$ of the moment map, one has a smooth symplectic quotient 
$$
M(\alpha) := \mu^{-1}(\alpha) / K.
$$
Suppose that the cotangent bundle $T^*V $ has a hyperk\"ahler structure and that the action of $K$ extends naturally to an action on $T^*V $ with a hyperk\"ahler moment map
$ \mu_{HK}: T^*V \to  \frak{k}^* \oplus (\frak k \otimes \C)^* $.
Then one defines the hyperk\"ahler quotient as
$$
X(\alpha, \beta) :=  \mu_{HK}^{-1} (\alpha, \beta) /K
$$
for appropriate values of $(\alpha, \beta)$.
When $V = S^2 \times \cdots \times S^2$ is a product of $n$ spheres and $K=SO(3)$, the space $X(\alpha, \beta)$ for generic $(\alpha, \beta)$ is a  smooth non-compact hyperk\"ahler quotient of a product of cotangent bundles $T^*S^2$ by $SO(3)$. 
When $\beta=0$, 
$$X(\alpha):= X(\alpha,0)$$ 
contains the  so-called polygon space $M(\alpha)$ of all configurations of closed piecewise linear paths in $\R^3$ with $n$ steps of lengths $\alpha_1, \ldots, \alpha_n$ modulo rotations and translations (a symplectic quotient of a product of $S^2$s by $SO(3)$). For this reason, $X(\alpha)$ is usually called a \emph{hyperpolygon space}.
This family of hyperk\"ahler quotients was first studied by Konno in \cite{konno} where he shows that these spaces, when smooth, are all diffeomorphic. 

It is known that a polygon space $M(\alpha)$ can be viewed as the moduli space of stable representations of a star-shaped quiver, as in Figure~\ref{quiver}.
\begin{figure}
\begin{center}
\includegraphics[width=3cm]{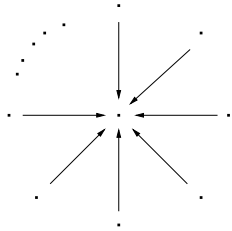}
\caption{Star-shaped quiver}
\label{quiver}
\end{center}
\end{figure}
More precisely, a star-shaped quiver $\mathcal{Q}$ with dimension vector 
$$v=(2,1, \ldots, 1)\in \R^{n+1}$$ 
is a directed graph with vertex set $J=\{ 0\} \cup \{ 1, \ldots,n\}$ and edge set $E = \big\{ (i,0) \mid i \in \{1, \ldots, n\} \big\}$. A representation of $\mathcal{Q}$, associated to a choice of finite dimensional vector spaces $V_i$, for  $i \in J$, such that $\dim V_i = v_i$,
is the space of homomorphisms from $V_i$ to $V_j$ for every pair of vertices $i$ and $j$ connected by an edge in $E$.
Therefore, the representation space of the star-shaped quiver $\mathcal{Q}$ described above is 
$$ E(\mathcal Q, V)= \bigoplus_{i =1}^n \textrm{Hom} (V_i, V_0) \cong \C^{2n}. $$ 
The group $ \prod GL(V_i) / GL(1)_{\Delta}$ acts in a Hamiltonian way on  $E(\mathcal Q, V)$ and the polygon space $M(\alpha)$ is obtained by symplectic reduction of $E(\mathcal Q, V)$ by this group, at the value $\alpha$. Similarly, one can obtain the hyperpolygon space $X(\alpha)$ as the hyperk\"ahler reduction of $T^* E(\mathcal Q, V)$ by the group $ \prod GL(V_i) / GL(1)_{\Delta}$ at $(\alpha,0)$.
Consequently, polygon and hyperpolygon spaces are examples of K\"ahler and hyperk\"ahler quiver  varieties in the sense of Nakajima \cite{n1,n2}. 

Any hyperk\"ahler quiver variety $X$ admits a natural $\C^*$-action and the \emph{core} $\mathfrak L$ of $X$ is defined as the set of points $x \in X$ for which the limit 
$$
\lim_{\lambda \to \infty} \lambda \cdot x
$$
exists. It clearly contains all the fixed-point set components and their flow-downs. Moreover, the core $\mathfrak L$ is a Lagrangian subvariety with respect to the holomorphic symplectic form and is a deformation retraction of $X$.
The circle $S^1 \subset \C^*$ acts on $X$ in  a Hamiltonian way with respect to the real symplectic form. This action has been studied by  Konno \cite{konno} for  hyperpolygon spaces. He shows that the fixed-point set of this action contains the polygon space $M(\alpha)$ (where the moment map attains its minimum) and that all the other components of $X(\alpha)^{S^1}$ are in bijection with the collection of index sets $S \subset \{ 1, \ldots, n\}$ of cardinality at least $2$ which satisfy 
\begin{equation}\label{eq:SHORT}
\sum_{i \in S} \alpha_{i} - \sum_{i \in S^c} \alpha_{i} < 0
\end{equation}
(see Theorem~\ref{thm:fixedpoints}).
Sets satisfying \eqref{eq:SHORT} are called \emph{short} sets following Walker \cite{W} and play a very important role in the study of polygon and hyperpolygon spaces.
The core of the hyperpolygon space $X(\alpha)$ is then 
$$ 
\mathfrak L_{\alpha} := M(\alpha) \cup \bigcup_{S \in \mathcal S'(\alpha)} U_S,
$$
where $U_S$ is the closure of the flow-down set of the fixed-point set component $X_S$ determined by the set $S$, and $\mathcal S'(\alpha)$ is the collection of short sets of cardinality at least $2$. Note that, even though the hyperpolygon spaces $X(\alpha)$ are all diffeomorphic for any generic choice of $\alpha$, they are not isomorphic as complex manifolds, nor as real symplectic manifolds nor as hyperk\"ahler manifolds. In particular, they are not $S^1$-equivariantly isomorphic and the dependence of $X(\alpha)$ and of its core $\mathfrak L_{\alpha} $  will be seen in Section~\ref{sec:wchyper}. The study of these changes is important since, for instance, the connected components of the core of a quiver variety  give a basis for the middle homology of the variety.

Let us now focus on the other family of spaces studied in this work.
Higgs bundles over a compact connected Riemann surface $\Sigma$ have been introduced by Hitchin \cite{H,H2} and are an important object of study in geometry with several relations with physics and representation theory. \emph{Parabolic Higgs bundles}, as first introduced by Simpson  \cite{S} (and hereafter referred to as simply PHBs), are
holomorphic vector bundles over $\Sigma$ endowed with a parabolic structure, that is, choices of weighted flags in the fibers over certain distinct marked points $x_1, \ldots, x_n $ in $\Sigma$, together with a Higgs field that respects the parabolic structure.

More precisely, if $D$ is the divisor $D= \sum_{i=1}^n x_i$  and $K_{\Sigma}$ is the canonical bundle over $\Sigma$, a parabolic Higgs bundle is a pair ${\bf E}:= (E, \Phi)$ where $E$ is a parabolic bundle over $\Sigma$ and 
$$\Phi : E \to E \otimes K_{\Sigma} (D)$$ 
(called the \emph{Higgs field}) is a strongly parabolic homomorphism. This means that $\Phi$ is a meromorphic endomorphism-valued one-form with simple poles along $D$ whose residues are nilpotent with respect to the flags.

As in the non-parabolic case, there exists a stability criterion (depending on the parabolic weights) that leads to the construction of moduli spaces of semistable parabolic Higgs bundles \cite{Y2}.
These spaces are smooth quasiprojective algebraic manifolds when the parabolic weights are chosen so that stability and semistability coincide. Such parabolic weights are called \emph{generic}. 

The original work of Hitchin in the non-parabolic setting extends to this context. In particular, the moduli space of parabolic Higgs bundles can be identified (as smooth manifolds) with the moduli space of solutions of the parabolic version of Hitchin's equations
$$
F(A)^{\bot} + [\Phi, \Phi^*] = 0 , \quad \overline{\delta}_A \Phi =0,
$$ 
where $A$ is a singular connection, unitary with respect to a singular hermitian metric on the bundle $E$ adapted to the parabolic structure (see \cite{K} for details).

The moduli spaces of parabolic Higgs bundles have a rich geometric structure. In particular, they contain the total space of the cotangent bundle of the moduli space of parabolic bundles whose holomorphic symplectic form can be extended to the entire moduli space.
Let $\mathcal{N}_{\beta,r,d}$ be the moduli space of rank-$r$,  degree-$d$ parabolic Higgs bundles that are stable for a choice of parabolic weights $\beta$, and let $\mathcal{N}^{0,\Lambda}_{\beta,r,d} \subset \mathcal{N}_{\beta,r,d}$ be the subspace of elements $(E, \Phi)$ that have fixed determinant and trace-free Higgs field.
Konno provides a gauge-theoretic interpretation of the moduli spaces $\mathcal{N}^{0,\Lambda}_{\beta,r,d} $ endowing them with a real symplectic form that, combined with the holomorphic one, gives a hyperk\"ahler structure on  $\mathcal{N}^{0,\Lambda}_{\beta,r,d}$, \cite{K}.

On the moduli space $\mathcal{N}_{\beta,r,d}$ there is a natural $\C^*$-action by scalar multiplication of the Higgs field. Restricting to $S^1 \subset \C^*$ one obtains a Hamiltonian circle action whose moment map $f: \mathcal{N}_{\beta,r,d} \to \R$ is a perfect Morse-Bott function on $\mathcal{N}_{\beta,r,d}$. Its downward Morse flow coincides with the so-called \emph{nilpotent cone} of $\mathcal{N}_{\beta,r,d}$ (see \cite{GGM} where  the work of Hausel \cite{Ha} is generalized to the parabolic case).

In this paper we show that hyperpolygon spaces are $S^1$-isomorphic to certain subspaces of $\mathcal{N}^{0,\Lambda}_{\beta,2,0}$ for $\Sigma=\C \P^1$ and for a generic choice of the parabolic weights $\beta_2(x_i) , \beta_1(x_i)$ with $x_i \in D$. 
Let $\alpha$ be the vector  
$$
\alpha:= \big(\beta_2(x_1) - \beta_1(x_1), \ldots, \beta_2(x_n) - \beta_1(x_n) \big) \in \R_+^n.
$$
Then the hyperpolygon space $X(\alpha)$ is $S^1$-isomorphic to the moduli space $\mathcal H(\beta) \subset \mathcal{N}^{0,\Lambda}_{\beta,2,0}$ of stable rank-$2$, holomorphically trivial PHBs over $\C\P^1$ with fixed determinant and trace free Higgs field. 
The isomorphism 
\begin{equation}\label{eq:1isom}
\mathcal I :X(\alpha) \to \mathcal H(\beta)
\end{equation}
constructed  in \eqref{eq:isom} restricts to an isomorphism between the polygon space $M (\alpha)$ and the moduli space of stable, rank-$2$, holomorphically trivial parabolic bundles over $\C\P^1$ with fixed determinant. (Viewing a polygon as a representation of a star-shaped quiver $\mathcal Q$ naturally yields a flag structure on $n$ fibers of a  rank-$2$, trivial bundle over $\C\P^1$.)
The fact that these two spaces are isomorphic has already been noted by Agnihotri and Woodward in \cite{AW} for small values of $\beta$. There, a different approach is taken to show that the symplectic quotient of a product of $SU(m)$-coadjoint orbits is isomorphic to the space of rank-$m$  parabolic degree-$0$  bundles over $\C\P^1$ for sufficiently small parabolic weights.

Generalizing the Morse-theoretic techniques introduced by Hitchin \cite{H} for the non-parabolic case, Boden and Yokogawa \cite{BY} and Garcia-Prada, Gothen and Muñoz \cite{GGM} use the restriction of the moment map $f$ to $\mathcal{N}^{0,\Lambda}_{\beta,r,d}$ to compute the Betti numbers in the rank-$2$ and rank-$3$ situation. These turn out to be independent of the parabolic weights. This fact is explained by Nakajima \cite{n3} who shows that the moduli spaces  $\mathcal{N}^{0,\Lambda}_{\beta,r,d}$ are actually diffeomorphic for any generic choice of the parabolic weights $\beta$.

The space $Q$ of admissible values of the parabolic weights $\beta$ contains a finite number of hyperplanes, called \emph{walls}, formed by non-generic values of $\beta$, which divide $Q$ into a finite number of chambers of generic values.
Thaddeus in \cite{T} shows that as $\beta$ crosses one of these walls the moduli space of parabolic Higgs bundles undergoes an elementary transformation in the sense of Mukai \cite{Mu} (see also \cite{Huy} for a detailed construction of these elementary transformations).

We adapt the work of Thaddeus to the moduli space $\mathcal H(\beta)$. In particular, we conclude that if $\mathcal{H}^\pm:=\mathcal H(\beta^\pm)$ are moduli spaces of PHBs for parabolic weights $\beta^+$  and $\beta^-$ on either side of a wall $W$, then $\mathcal{H}^+$ and $\mathcal{H}^-$ are related by a Mukai transformation where $\mathcal H^+$ and  $\mathcal H^-$ have a common blow-up. The locus in  $\mathcal{H}^-$ which is blown up is isomorphic to a complex projective space $\P U^-$ parameterizing all non-split extensions 
$$
0 \to {\bf L}^+ \to {\bf E} \to {\bf L} ^- \to 0
$$
of a trivial parabolic Higgs line bundle ${\bf L}^-$ that are $\beta^-$-stable but $\beta^+$-unstable.  
Using the isomorphism in \eqref{eq:1isom} 
we conclude that the corresponding hyperpolygon spaces $X^\pm := X(\alpha^{\pm})$  
are related by a Mukai transformation (see Theorem~\ref{thm:wch}). Moreover, the blown up locus  $\P U^-$ corresponds, via the isomorphism above, 
to a core component $U_S^-$ in $X^-$ for some short set $S \subset \{1, \ldots, n \}$ uniquely determined by the wall 
$W$. Taking advantage of the geometric description of the core components in $X(\alpha)$ we study
the changes in the other components $U_B^\pm$ of the cores $\mathfrak L_\pm$ when crossing a wall, which naturally 
depend on the intersections $U_B^- \cap U_S^- $ and $U_B^+ \cap U_{S^c}^- $ (see Section \ref{sec:wchyper}). Moreover, we recover the description of the birational map relating the polygon spaces $M(\alpha^\pm)$ given in \cite{m}.
These changes in the core translate, via our isomorphism, to changes in the nilpotent cone of $\mathcal{H}(\beta)$. In particular, one recovers the dependence on the parabolic weight $\beta$ of the moduli spaces of rank-$2$, degree-$0$ parabolic bundles over $\C \P^1$  studied in \cite{BH}. The study of the dependence of the whole nilpotent cone on the weights $\beta$ is new in the literature. 

Going back to the study of hyperpolygon spaces and their cores we consider $n$ circle bundles $\widetilde{V}_i$ over
$X(\alpha)$ and their first Chern classes $c_i := c_1(\widetilde{V}_i)$ as defined by Konno \cite{konno}.
These classes generate the cohomology ring of the hyperpolygon space $X(\alpha)$ (see \cite{konno, hp, hauselp}),
as well as  the cohomology of all the core components.
In particular, the restrictions $c_i \vert_{M(\alpha)}$ to the polygon space $M(\alpha)$ 
are the cohomology classes considered in \cite{AG} to determine the intersection ring of $M(\alpha)$. In this work
we give explicit formulas for the computation of the intersection numbers of the restrictions of the classes $c_i$ to the other 
core components. 

For that we first prove a recursion formula in $n$ which allows us to decrease the dimension of the spaces involved (see Theorem~\ref{thm:recursion}). Analog recursion formulas have already appeared for other moduli spaces in the work of Witten and Kontsevich (on moduli spaces of punctured curves) \cite{Ko, W1, W2}, of Weitsman (on moduli spaces of flat connections on $2$-manifolds of genus $g$ with $n$ marked points) \cite{Wi} and of Agapito and Godinho  (on moduli spaces of polygons in $\R^3$) \cite{AG}.
Based on our recursion relation we obtain explicit formulas for the intersection numbers of the core components $U_S$ (see Theorems~\ref{thm:1} and \ref{thm:2}).

Finally, the isomorphism $\mathcal{H}(\beta) \leftrightarrow X(\alpha)$ allows us to consider circle bundles over   $\mathcal{H}(\beta)$ (the pullbacks of those constructed over $X(\alpha)$) and their Chern classes. We can then obtain explicit formulas for the intersection numbers of the restrictions of these Chern classes to the different components of the nilpotent cone of $\mathcal{H}(\beta)$, which allow us to determine their intersection rings.

For completion, we use the isomorphism $\mathcal{I}$ together with the work of Harada-Proudfoot \cite{hp} and Hausel-Proudfoot \cite{hauselp} for hyperpolygon spaces to present the cohomology rings of $\mathcal{H}(\beta)$ and of its nilpotent cone components (see Theorems~\ref{thm:H} \and \ref{thm:US}). 

Here is an outline of the contents of the paper. In Section~\ref{sec:prel}, we review the basic definitions and facts about hyperpolygon spaces and moduli spaces of PHBs. In Section~\ref{sec:isomorphism}, we prove the existence of an isomorphism between hyperpolygon spaces and moduli spaces $\mathcal{H}(\beta)$ of stable rank-$2$, holomorphically trivial PHBs over $\C \P^1$ with fixed determinant and trace-free Higgs field, which is $S^1$-equivariant with respect to naturally defined circle actions on these two spaces. In Section~\ref{wall crossing}, we adapt Thaddeus' work \cite{T} on the variation of moduli spaces of PHBs to $\mathcal{H}(\beta)$ and, in Section~\ref{sec:wchyper}, we prove, via our isomorphism, that the corresponding hyperpolygon spaces $X(\alpha)$ undergo a Mukai transformation when the parameter $\alpha$ crosses a wall in the space of its admissible values. Moreover, in this section, we describe the changes suffered by the different core components as a result of this transformation. These changes easily translate to changes in the different components of the nilpotent cone of $\mathcal{H}(\beta)$. In Section~\ref{sec:intnum}, we construct circle bundles over $X(\alpha)$ and study the intersection numbers of their restrictions to each core component, giving examples of applications. In Section~\ref{sec:intersecPHB}, we see that the formulas obtained for the core components of $X(\alpha)$ also apply to the nilpotent core components of the corresponding moduli space of PHBs $\mathcal{H}(\beta)$, thus determining their intersection ring. Finally, for completion, we give presentations of the cohomology rings of $\mathcal{H}(\beta)$ and of each of its nilpotent cone components.

{\bf Acknowledgments:} The authors are grateful to Tamás Hausel for bringing PHBs to their attention and suggesting a possible connection with hyperpolygon spaces. Moreover, they would like to thank Ignasi Mundet Riera for useful discussions on the explicit construction  of the isomorphism ~\eqref{eq:1isom}. Finally, they would like to thank Gustavo Granja for helpful explanations regarding vector bundle extensions and Jean-Claude Hausmann  and Luca Migliorini for many useful conversations.  
\section{Preliminaries}\label{sec:prel}

\subsection{Polygons and Hyperpolygons}
\label{sec:pol}
Hyperpolygon spaces have been introduced by Konno \cite{konno} from a symplectic point of view, 
as the hyperk\"ahler quotient analogue of polygon spaces, and from an algebro-geometric 
point of view, as GIT quotients.
 
Hyperpolygon and polygon spaces are respectively the hyperk\"ahler and K\"ahler 
quiver varieties associated to 
star-shaped quivers $ \mathcal Q$ (Figure \ref{quiver}), that is,
those with vertex set $I \cup \{0\},$ for $I:= \{1 , \ldots, n \},$
and edge set $E = \{ (i,0) \mid i \in I\}$. 

Consider the representation of a star-shaped quiver  $\mathcal Q$ obtained by taking vector spaces $V_i = \C $ for $i \in I ,$ and 
$V_0= \C^2$. Then one gets the hyperk\"ahler quiver variety associated with 
$ \mathcal Q$ by performing hyperk\"ahler reduction on the cotangent bundle 
of the representation space $$ E(\mathcal Q, V)= \bigoplus_{i \in I} \textrm{Hom} (V_i, V_0) \cong \C^{2n} $$ 
with respect to the action of the group 
$ U(2) \times U(1)^n$  by conjugation.
Since the diagonal circle in $ U(2) \times U(1)^n$ acts trivially on the cotangent bundle of $E(\mathcal Q, V)$, 
one can consider the action of the quotient group
$$K := \Big( U(2) \times U(1)^n\Big)/U(1)= \Big(SU(2) \times U(1)^n\Big)/ \Z_2,$$
where $\Z_2$ acts by multiplication of each factor by $-1.$
As $T^* \C^2 \cong (\C^{2})^* \times \C^2$ can be identified with 
the space of quaternions, the cotangent bundle $T^*E(\mathcal Q, V) \cong T^* \C^{2n}$ has 
a natural hyperk\"ahler structure (see for example \cite{konno, hitchin}).  
Moreover, the hyperk\"ahler quotient of $T^* \C^{2n}$ 
by $K $ can be explicitly described as follows. Let
 $(p,q)$ be coordinates on $T^* \C^{2n}$, where $p=(p_1, \ldots, p_n)$ 
is the $n$-tuple of row vectors  $p_i =(a_i, b_i) \in \C^2$ and 
$q= (q_1, \ldots, q_n)$ is the $n$-tuple of column vectors 
$q_i =\Big( \begin{array}{c}
c_i \\ d_i 
\end{array} \Big) \in \C^2.$  
The action of $K$ on $ T^*\C^{2n}$ is given by 
$$ (p,q) \cdot  [A; e_1, \ldots, e_n]= 
\Big( (e_1^{-1}p_1 A, \ldots, e_n^{-1}p_n A), ( A^{-1} q_1 e_1, \ldots, A^{-1} q_n e_n ) \Big).$$
This action is hyperhamiltonian with hyperk\"ahler moment map \cite{konno}
$$ \mu_{HK}:= \mu_{\R} \oplus \mu_{\C} : T^* C^{2n} \rightarrow 
\big(\mathfrak{su}(2)^* \oplus (\R^n)^*\big) \oplus \big(\mathfrak{sl}(2, \C)^* \oplus (\C^n)^*\big),$$ 
where the real moment map $\mu_{\R}$ is given by 
\begin{equation} \label{real} 
\mu_{\R} (p,q) = \frac{ \bold{i}}{2} \sum_{i=1}^n (q_i q_i^* -p_i^* p_i )_0 
\oplus \Big(\frac{1}{2} (|q_1|^2 -|p_i|^2), \ldots, 
\frac{1}{2} (|q_n|^2 -|p_n|^2) \Big)
\end{equation} 
for $\bold{i}:= \sqrt{-1},$ and the complex moment map $\mu_{\C}$ is given by
\begin{equation} \label{complex}
 \mu_{\C} (p,q) = -  \sum_{i=1}^n (q_i p_i)_0 \oplus ( \bold{i} \, p_1 q_1, 
\ldots,\bold{i} \, p_n q_n).
\end{equation}
The hyperpolygon space $X(\alpha)$ is then defined to be  the hyperk\"ahler quotient
$$ X(\alpha) = T^*\C^{2n} / \!\! / \!\!/ \!\!/_{\alpha} K := 
\Big( \mu_{\R}^{-1} (0, \alpha) \cap \mu_{\C}^{-1} (0,0) \Big) / K$$
for $\alpha=(\alpha_1, \ldots, \alpha_n ) \in \R^n_+$.

\begin{Remark} An element $(p,q) \in T^*\C^{2n}$ is in  $\mu_{\C}^{-1} (0,0)$ if and only if 
$$
p_i q_i=0 \quad \text{and} \quad \sum_{i=1}^n (q_i p_i)_0=0,
$$
i.e. if and only if
\begin{equation}\label{complex1}
 a_i c_i + b_i d_i = 0
\end{equation}
and 
\begin{equation}\label{complex2}
\sum_{i=1}^n a_i c_i - b_i d_i =0, \quad
\sum_{i=1}^n a_i d_i =0, \quad
\sum_{i=1}^n b_i c_i =0.
\end{equation}

Similarly, $(p,q)$ is in $\mu_{\R}^{-1} (0, \alpha) $ if and only if 
$$
\frac{1}{2} \big(|q_i|^2 -|p_i|^2\big) = \alpha_i	\quad \text{and} \quad \sum_{i=1}^n \big(q_i q_i^* -p_i^* p_i \big)_0 =0,
$$
i.e. if and only if
\begin{equation}\label{real1}
|c_i|^2 +|d_i|^2 - |a_i|^2 - |b_i|^2 = 2 \alpha_i
\end{equation}
and 
\begin{equation}\label{real2}
\sum_{i=1}^n |c_i|^2 - |a_i|^2 + |b_i|^2 - |d_i|^2 =0, \quad
\sum_{i=1}^n a_i \bar{b_i} - \bar{c_i} d_i =0.
\end{equation}

\end{Remark}

An element $\alpha=(\alpha_1, \ldots, \alpha_n ) \in \R^n_+$ is said to be 
\emph{generic} if and only if 
\begin{equation}
\label{eq:epsilon}
\varepsilon_S(\alpha) := \sum_{i \in S} \alpha_{i} - \sum_{i \in S^c} \alpha_{i} \neq 0
\end{equation}
for every index set $S \subset \{1, \ldots, n\}$. For a generic $\alpha$, the hyperpolygon space $X(\alpha)$ is a 
non-empty smooth manifold of complex dimension $2(n-3)$. 

On the other hand, one defines polygon spaces $M(\alpha)$ using  the quiver $\mathcal Q$ of Figure \ref{quiver}
and the collection of vector spaces $V_i= \C$ and $V_0= \C^2$ by performing symplectic reduction on $E(\mathcal Q, V)=\C^{2n}$ by the action of $K$.
More precisely, one considers the Hamiltonian action of $K$ on $\C^{2n}$  given by 
$$q \cdot  [A; e_1, \ldots, e_n]= (A^{-1} q_1 e_1, \ldots, A^{-1} q_n e_n ),$$
with moment map
\begin{equation} \label{eq:mommap}
 \begin{array}{rcl}
\mu: \C^{2n} & \rightarrow & \mathfrak{su}(2)^* \oplus \big( \mathfrak{u}(1)^n \big)^*\\ 
q & \mapsto & \displaystyle\sum_{i=1}^n (q_i q_i^*)_0 
\oplus \Big(\frac{1}{2} |q_1|^2, \ldots, \frac{1}{2} |q_n|^2 \Big).
\end{array}
\end{equation}   
Then for $\alpha\in \R^n_{+}$, 
$$M(\alpha) := \C^{2n}/\!\!/_{(0, \alpha)} K= \mu^{-1}(\alpha)/K. $$
Note that $M(\alpha)$ lies inside the hyperpolygon space $X(\alpha)$ as the locus of points $[p, q]$ with $p=0$.

Performing reduction in stages one obtains the polygonal 
description of $M(\alpha).$ In fact, the symplectic reduction of $\C^{2n}$ 
by $U(1)^n$ (or, more precisely, by the maximal subtorus 
$T^n:= (Id \oplus U(1)^n)/ \Z_2$ in $K$) at the $\alpha$-level set is the 
product  of $n$ spheres of radii $\alpha_1, \ldots, \alpha_n$ and  
the residual action of $K/T^n \cong SO(3)$ on this product is just the  standard
action by rotation with moment map
\begin{displaymath}
\begin{array}{rcl}
\mu_{SO(3)}: \displaystyle\prod_{i=1}^n S^2_{\alpha_i} & \rightarrow & \R^3\\ 
(v_1, \ldots, v_n) & \mapsto & v_1 + \cdots +v_n.\\
\end{array}
\end{displaymath}
Performing the second step of reduction one gets 
$$M(\alpha) = \prod S^2_{\alpha_i}/ \!\! /_{\!0} SO(3)= \mu_{SO(3)}^{-1} (0) / SO(3).$$
The level set $ \mu_{SO(3)}^{-1} (0)$ is then the set of all closed polygons in $\R^3$ with $n$ 
edges $v_1, \ldots, v_n$ of lengths $\alpha_1, \ldots, \alpha_n$ respectively and the quotient
$M(\alpha)$ is the moduli space of all such polygons modulo rigid motions in $\R^3$. 
Note that this space is empty if $\alpha_i > \sum_{j \neq i} \alpha_j$ 
for some $i \in \{1, \ldots n \}$ since, in this case, the closing condition $\sum_{i=1}^n v_i=0$
is not verified for any $v \in \prod S^2_{\alpha_i}.$ 

If $\alpha$ is generic the polygon space  $M(\alpha)$  is a smooth manifold of complex dimension $n-3$ (when not empty).
Here generic has a geometric interpretation. It means  that no element  in $M(\alpha)$ is represented by a polygon contained in a line. In fact, if such a polygon existed,  
the $SO(3)$-action would not be free 
since the stabilizer of this polygon would be the circle  of rotations around the corresponding line. 
The quotient $M(\alpha)$ would then have a singularity.

Reduction in stages can also be performed in the opposite order.
The quotient $C^{2n} / \! \! /_{0} SU(2)$ is then identified with the Grassmannian $Gr(2,n)$ 
of 2-planes in $\C^{n}$, (see \cite{hk} for details).
The remaining $U(1)^n$-action has moment map 
\begin{equation} \label{eq:Grass}
\begin{array}{rcl}
 \mu_{\small{ U(1)^n}} : Gr(2,n)& \longrightarrow & \R^n\\
q &\mapsto&  \frac{1}{2}\Big(|q_1|^2, \ldots, |q_n|^2 \Big)
\end{array} 
\end{equation}
and the polygon space $M(\alpha)$ is the symplectic quotient $Gr(2,n) / \! \! /_{\alpha} U(1)^n$.

Hyperpolygon spaces can be described from an
algebro-geometric point of view as GIT quotients. For that we need the stability criterion developed by 
Nakajima \cite{n1,n2} for quiver varieties and adapted by Konno \cite{konno} to hyperpolygon spaces. 

Let $\alpha$ be generic. A set $S \subset \{1, \ldots, n\}$ is called \emph{short} if
 \begin{equation}\label{eq:short}
 \varepsilon_S(\alpha) <0
 \end{equation}
and \emph{long} otherwise.
Given  $(p,q)\in T^* \C^{2n}$ and  a set $S \subset \{1, \ldots, n\}$, we say that $S$  is \emph{straight} at $(p,q)$ 
if $q_i$ is proportional to $q_j$ for all $i, j \in S.$

\begin{Theorem}\cite{konno} 
 \label{alpha-stability}
Let $\alpha \in \R^n_{+}$ be generic. A  point $(p,q) \in T^* \C^{2n}$ is \emph{$\alpha$-stable} 
if and only if the following two conditions hold:
\begin{itemize}
\item[(i)] $q_i \neq 0$ for all $i$,  and
\item[(ii)] if $S\subset \{1, \ldots, n \}$ is straight at $(p,q)$ and $p_j =0$ for all $j \in S^c$, then $S$ is short.
\end{itemize}
\end{Theorem}

\begin{Remark}\label{maximalstraight}
Note that it is enough to verify  (ii) for all maximal straight sets, that is for those that are not contained in any other straight set at $(p,q)$.
\end{Remark}
Let us denote by $\mu_{\C}^{-1} (0)^{\alpha -st}$ the set of points in $\mu_{\C}^{-1} (0)$ that are 
$\alpha$-stable and let $K^{\C}:= (SL(2,\C) \times (\C^*)^n) / \Z_2$ be the complexification of $K$.
\begin{proposition}\label{git}\cite{konno}
Let $\alpha \in \R^n_{+}$ be generic. Then
$$\mu_{HK}^{-1} \big( (0, \alpha),(0,0) \big) \subset \mu_{\C}^{-1} (0)^{\alpha -st}$$
and there exists a natural bijection 
$$ \iota: \mu_{HK}^{-1} \big( (0, \alpha),(0,0) \big)/K \to  \mu_{\C}^{-1} (0)^{\alpha -st}/ K^{\C}.$$
\end{proposition} 
It follows that
$$X(\alpha) = \mu_{\C}^{-1} (0)^{\alpha -st} / K^{\C}. $$
As in  \cite{hp} we denote the elements in $\mu_{\C}^{-1} (0)^{\alpha -st}/ K^{\C}$ by 
$[p,q]_{\alpha-\text{st}}$, and by $[p,q]_{\R}$ the elements in $\mu_{HK}^{-1} \big( (0, \alpha),(0,0) \big)/K$ when we need to make explicit use of one of the two constructions. In all other cases, we will simply write $[p,q]$ for a hyperpolygon in $X(\alpha)$. 

\subsubsection{The Core}\label{core} 
Let us assume throughout this section that $\alpha$ is generic. The core of a hyperpolygon space $X(\alpha)$
has been studied in detail in \cite{konno,hp}, and here we give a brief overview of the results 
therein that will be relevant to our study. 

Consider the $S^1$-action on $X(\alpha)$ defined by 
\begin{equation}\label{action}
\lambda \cdot  [p,q] = [\lambda \, p, q].
\end{equation} 
This action is Hamiltonian with respect to 
symplectic structure $\omega_{\R}$  and the associated moment map $\phi:X(\alpha) \to \R$,  given by
\begin{equation}\label{eq:phi}
\phi([p,q]_{\R}) = \frac{1}{2}\sum_{i=1}^n |p_i|^2,
\end{equation}
is a Morse--Bott function. Following Konno\cite{konno} consider $\mathcal S(\alpha)$,
the collection of short sets for $\alpha$, and its subset 
$$ \mathcal S' (\alpha):= \big\{ S \subset \{1, \ldots,n \} \mid S \text{ is } \alpha\text{-short}, |S| \geq 2 \big\}. $$
Then,
\begin{Theorem}\cite{konno}\label{thm:fixedpoints}
The fixed point set for the $S^1$-action \eqref{action} is
$$ X(\alpha)^{S^1}= M(\alpha) \cup \bigcup_{S \in \mathcal S'(\alpha)} X_S$$
where,  for each element of $\mathcal S'(\alpha)$, 
$$ X_S := \big\{[p,q] \in X(\alpha) \mid S \text{ and } S^c\, \text{are straight},\,  p_j=0   \text{ for all } j \in S^c \big\}.$$ 
Moreover,  $X_S$ is diffeomorphic to $\C\P^{|S|-2}$ and has index $2(n-1-|S|)$. 
\end{Theorem}

For  $S \in \mathcal S'(\alpha)$ let $U_S$ be the closure of
$$ \big\{ [p,q] \in X(\alpha) \mid \lim_{\lambda \rightarrow \infty}[\lambda\, p,q ] \in X_S \big\} .$$ 
Then the \emph{core} $\mathfrak L_{\alpha}$ of $X(\alpha)$ is defined as 
$$\mathfrak L_{\alpha} := M(\alpha) \cup \bigcup_{S \in \mathcal S'(\alpha)} U_S$$
and is a deformation retraction of $X(\alpha)$. In fact $U_S$ is the closure of the flow-down set for the critical component $X_S$ and the polygon space (when non-empty) is the minimal set of $\phi$.
The core components $U_S$ are 
smooth compact submanifolds of complex dimension $n-3$, and can equivalently be described as
\begin{equation}
\label{eq:us}
 U_S= \{ [p,q] \mid S \text{ is straight and } p_j=0 \, \text{ for all } j \in S^c\}
\end{equation}
(see \cite{hp} for details). Moreover, they can be
nicely described  as moduli spaces of  pairs of  
polygons in $\R^3$ (see \cite{hp}). For that,
given a short set $S$ in $\mathcal S'(\alpha),$ and a point $[p,q]_{\R} \in U_S,$ define a $(n+1)$-tuple of vectors in $\R^3$,  $( u_i, v_j, w)$, $i \in S$, $j \in S^c$ as
$$u_i= q_i p_i + p_i^*q_i^*, \quad \forall i \in S$$
$$v_j= (q_j q_j^*)_0, \quad \quad \forall j \in S^c$$
$$w= \sum_{i \in S} (q_i q_i^*)_0 -(p_i^* p_i)_0,$$
where we make the usual identification 
${\bf i} \cdot \mathfrak{su}(2) \cong \mathfrak{su}(2)^* \cong \R^3$. These $n+1$ vectors  define
two polygons: one in $\R^3$ with edges $w$ and $ v_j$, with $ j \in S^c$,  and
one  lying in the orthogonal plane to $w$ with edges $u_i$ for  $i\in S$.
Note that $\|v_j\| = \alpha_j$ and that
$$\sum_{i \in S} \alpha_i \leq \| w\| \leq \sum_{j \in S^c } \alpha_j,$$
where the variations in $\|w\|$ are determined by the lengths of the vectors $u_i$.
The lower bound $\| w \|= \sum_{i \in S} \alpha_i $ is reached when $u_i=0$ for all $i$, meaning that the 
planar polygon collapses to a point and one obtains a polygon in $\R^3$ of edges $w$ and  $\{ v_j \mid j \in S^c \}.$
In this case, the point $[p,q]_{\R}$ defining this polygon is in the intersection $U_S \cap M(\alpha).$
When the upper bound $\| w  \| =\sum_{j \in S^c } \alpha_j$ is reached, the spatial polygon is forced to be 
in a line and the planar polygon has maximal perimeter. 
\begin{Theorem}\label{hp}\cite{hp}
For any $S\in \mathcal S'(\alpha) $ the associated core component $U_S$ is homeomorphic to the moduli space $\mathcal{Z}$ of $n+1$ of vectors 
$$\{ u_i, v_j, w \in \R^3 \mid  i \in S, j\in S^c \}$$ 
taken up to rotation, satisfying the conditions: 
\begin{align*}
1)& \quad  w + \displaystyle\sum_{j \in S^c} v_j=0;\\ 
2)&\quad   \displaystyle\sum_{i \in S} u_i =0;\\
3)& \quad u_i \cdot w =0  \text{ for all } i \in S;\\
4)&\quad  \|v_j \| = \alpha_j  \text{ for all } j \in S^c;\\
5)& \quad \| w \| = \displaystyle\sum_{i \in S} \sqrt{\alpha_i^2 + \|u_i \|^2}.
\end{align*}
\end{Theorem}
\begin{figure}[htbp]
\begin{center}
\psfrag{w}{$w$}
\psfrag{v1}{$v_{j_1}$}
\psfrag{v2}{$v_{j_2}$}
\psfrag{vSc}{$v_{j_{\lvert S^c \rvert}}$}
\psfrag{u1}{$u_{i_1}$}
\psfrag{uS}{$u_{i_{\lvert S \rvert}}$}
\includegraphics[width=8cm]{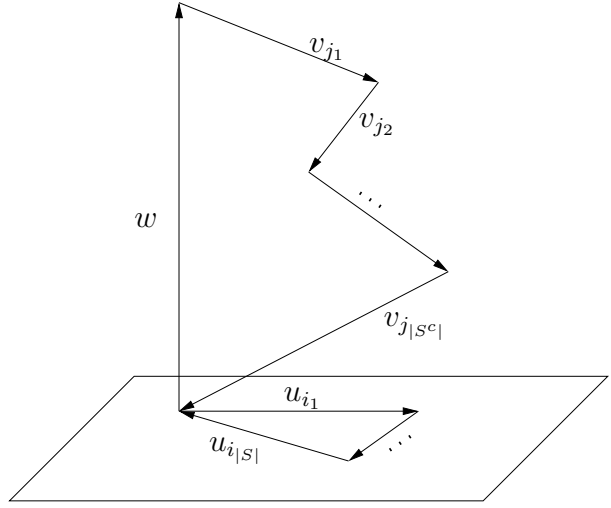}
\caption{A hyperpolygon in the core component $U_S$ described as a pair of
a spacial polygon and a planar one (where $S=\{i_1, \ldots, i_{\lvert S \rvert}\}$ and $S^c=\{j_1, \ldots, j_{\lvert S^c \rvert}\}$).}
\label{hyperpolygon}
\end{center}
\end{figure}

If the polygon space $M(\alpha)$ is non empty, then all the core components $U_S$ intersect $M(\alpha)$. More precisely, for any 
$S \in \mathcal S'(\alpha),$ 
$$ U_S \cap M(\alpha)\cong  M_S (\alpha),$$
where
\begin{equation}
\label{eq:ms}
M_S (\alpha):= \Big\{ v \in \prod_{i=1}^n S^2_{\alpha_i} \mid \sum_{i=1}^n v_i =0, v_i \text{ proportional to } v_j \, \forall i,j \in S  \Big\}/ SO(3).
\end{equation}
This intersection is a $(|S^c|-2)$-dimensional submanifold of $M(\alpha)$ that can be identified with the moduli space 
of polygons in $\R^3$ with $|S^c|+1$ edges of lengths $\sum_{i \in S} \alpha_i$ and $ \alpha_j$, for $j \in S^c$. 

The intersection of any other two core components $U_S$ and $U_T$, with $S, T \in \mathcal S'(\alpha),$ 
depends upon the intersection of the short sets $S$ and $T$.
\begin{itemize}
\item If $S \cap T = \emptyset $ then $U_S \cap U_T = M_S(\alpha) \cap M_T(\alpha).$ (Note that this intersection might be empty.)
\item If $S \cap T \neq \emptyset $ and $ S \cup T $ is long, then $U_S \cap U_T = \emptyset$.
\item If $S \cap T \neq \emptyset $ and $ S \cup T $ is short, then
$$ U_S \cap U_T = \big\{[p,q] \mid S \cup T \text{ straight }, p_j=0 \text{ for all } j \in (S \cap T)^c \big\} \subseteq U_{S \cup T}.$$
\end{itemize}
Finally, if $S \subset T $, the critical submanifold $X_T$ intersects $U_S,$ and $U_S \cap X_T \cong \C\P^{|S|-2}$ (cf. \cite{hp}). 
In particular,
\begin{proposition}\label{Smaximal}
If $S \in \mathcal S'(\alpha)$ is maximal with respect to inclusion then
$$U_S \cong \C\P^{n-3}. $$
\end{proposition} 
This was conjectured in \cite{hp}, and is  a simple consequence of  the following result of Delzant. 
\begin{Theorem}\label{delzant}\cite{del}
Let $(M, \omega)$ be a compact symplectic $2n$-dimensional manifold equipped with  a Hamiltonian
$S^1$-action with moment map $\phi$. If $\phi$ has only two critical values, one of which is
non-degenerate, then $M$ is isomorphic to $(\C\P^n, \lambda \omega_{FS}),$  where $\lambda \omega_{FS}$ is
some multiple of the Fubini--Study symplectic form.
\end{Theorem}

\begin{proof}(Proposition~\ref{Smaximal})
Since $S$ is maximal with respect to inclusion, the core component $U_S$ is just the closure of the flow down set of $X_S \cong \C\P^{|S|-2}$. 

If $|S|=n-1$ then, assuming without loss of generality that  $S=\{1,\ldots,n-1\}$, we have  
$$\sum_S \alpha_i < \alpha_{n} $$
($S$ is short), meaning that the polygon space $M_S(\alpha)$ is empty. Therefore $U_S= X_S \cong \C \P^{n-3}$.

If $|S| < n-1$ then $X_S$ has index 
$2(n-1-|S|)$ and $\phi(X_S)$ is a 
non-degenerate critical value of the restriction of $\phi $ to $U_S$. 
The only other critical value of $\phi$ on $U_S$ is its minimum value $\phi(M(\alpha))=0.$
We can then apply Theorem \ref{delzant} to $U_S$ equipped with the restriction
of the $S^1$-action on $X(\alpha)$ to conclude the proof.   
\end{proof}

\begin{example}\label{ex:h4}
When $n=4$ there are four critical components of the moment map $\phi$ for any generic choice of $\alpha$. In fact, since either $S$ or $S^c$ is short, there are always exactly three short sets ($S_1$, $S_2$ and $S_3$) of cardinality $2$ in $\mathcal{S}^\prime(\alpha)$. Moreover, the polygon space $M(\alpha)$ is empty if and only if there is a short set $S_0$ of cardinality $3$ in  $\mathcal{S}^\prime(\alpha)$. Note that in this case there is exactly one such set in  $\mathcal{S}^\prime(\alpha)$. The critical components $X_{S_i}$, $i=1,2,3$, are isolated points of index $2$, while $X_{S_0}$ and $M(\alpha)$, when nonempty, are diffeomorphic to $\C \P^1$ and have index $0$. The core components $U_{S_i}$, for $i=1,2,3$, are three copies of $\C\P^1$ intersecting the minimal component in three distinct points. Consequently, the core $\mathfrak{L}_\alpha$ is a union of $4$ spheres arranged in a $D_4$ configuration  \cite{Fri} as in Figure~\ref{fig:core4}.
\begin{figure}[htbp]
\psfrag{X1}{$X_{S_1}$}
\psfrag{X2}{$X_{S_2}$}
\psfrag{X3}{$X_{S_3}$}
\psfrag{M}{$X_{S_0}$ or $M(\alpha)$}
\includegraphics[scale=.4]{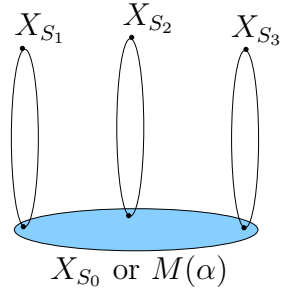}
\caption{Core of $X(\alpha)$ when $n=4$: four spheres arranged in a $D_4$
configuration.}
\label{fig:core4}
\end{figure}
\end{example}

\subsubsection{Walls}\label{Walls}
We now set some notation and basic definitions relative to the 
wall-crossing analysis that will be carried out in Section \ref{wall crossing}. 
Moreover,  we summarize the wall-crossing behavior for polygon spaces which is described in detail in \cite{m}.

Let $\Gamma \subset \R^n_+$ be the set of generic values of  $\alpha$. 
If $\alpha \notin \Gamma$ then there exists an index set $S \subset \{1, \ldots,n\}$ for which
$\varepsilon_S (\alpha)= 0$. Hence $\Gamma$ is the complement of the union of finitely many  walls
$$W_S:=\{ \alpha\in \R^n_+ \mid \varepsilon_S (\alpha)= 0\}$$
with $S \subset \{1,\ldots,n\}$. The set $S$ will be called the \emph{discrete data} of $W_S$.

Note that an index set $S$ and its complement $S^c$ define the same wall. Moreover, a wall $W_S$ separates two adjacent connected components of $\Gamma$, called \emph{chambers}, say $\Delta^+$ and $\Delta^-$, such that 
$\varepsilon_S(\alpha^+)>0$ for every $\alpha^+ \in \Delta^+$ and $\varepsilon_S(\alpha^-)<0$ for every
$\alpha^- \in \Delta^-$. Consequently, $S$ is maximal short (with respect to inclusion) for values of $\alpha^-$ in $\Delta^-$ and long for those in $\Delta^+$.

The collection of short sets $\mathcal S(\alpha)$ completely determines the 
chamber of $\alpha$ and, since only one of $S$ and $S^c$ is short, there is a $1$-$1$ correspondence between the elements of
$\mathcal S(\alpha)$ and the walls in $\R_+^n.$

\begin{Remark}
The image
$$ \Xi :=  \mu_{U(1)^n} (Gr(2,n))= \Big\{ (\alpha_1, \ldots, \alpha_n) \in \R^n_+ \mid 0 \leq \alpha_i \leq \frac{1}{2}  \text{ and } \sum_{i=1}^n \alpha_i =1\Big\}$$
of the moment map defined in \eqref{eq:Grass} is formed by values of $\alpha$ for which  $M(\alpha)$ is nonempty.
Since  $M(\alpha)$ is diffeomorphic to $M(\lambda \alpha)$ for every $\lambda\in \R_+$, one can easily see that $M(\alpha)\neq \varnothing$ if and only if $\alpha$  is in the cone $C_{\Xi}$ over $\Xi$. 
The walls $W_S$ with $|S|=1$ or $|S|=n-1$ form the boundary of $C_{\Xi}$ and so are called \emph{vanishing walls}. (When $\alpha$ crosses one of these walls the whole space $M(\alpha)$ vanishes.) The chambers in $\R_+^n \setminus C_{\Xi}$ are called \emph{null chambers} and each of these  is separated from $C_{\Xi}$ by a unique vanishing wall. 
\end{Remark}

By the Duistermaat--Heckman Theorem, $M(\alpha^+)$ and $M(\alpha^-)$ 
are diffeomorphic for $\alpha^+$ and $\alpha^-$ in the same chamber but the diffeotype of $M(\alpha^\pm)$ changes if $\alpha^+$ and $\alpha^-$ are in different chambers. In particular, if $\alpha^+$ and $\alpha^-$ lie in opposite sides of a single wall $W_S$, then $M(\alpha^+)$ and $M(\alpha^-)$ are related  by a blow up followed by a blow down. This is a classical result for reduced spaces
(see, for example \cite{gs,bp}) and has been worked out in detail for the case of polygon spaces in \cite{m}, where the submanifolds involved in the birational transformation are characterized in terms of lower dimensional polygon spaces. More precisely, these submanifolds are the intersections 
$$M_S(\alpha^+)= U_S \cap M(\alpha^+)\quad \text{and} \quad M_S(\alpha^-)= U_S \cap M(\alpha^-)$$
defined in \eqref{eq:ms}.
\begin{Theorem}\label{wall-crossing}\cite{m} If $\Delta^+$ and $\Delta^-$ are two chambers lying in opposite sides of a wall $W_S$ and $S$ is short for $\alpha^-\in \Delta^-$ and long for $\alpha^+ \in \Delta^+$, then $M(\alpha^+)$ is obtained from $M(\alpha^-)$ by
 a blow up along $M_S(\alpha^-)\cong \C\P^{|S^c| - 2}$ followed by 
a blow down of the projectivized normal bundle of $M_{S^c} (\alpha^+) \cong \C\P^{|S| - 2}.$
\end{Theorem}

The situation for hyperpolygon spaces  is quite different. The diffeotype of $X(\alpha)$ does not depend on the 
value $\big( (\alpha,0)(0,0) \big)$ of the hyperk\"ahler moment map as long as $\alpha$ is generic (see  \cite{konno}). Nevertheless, if $\alpha^+$ and $\alpha^-$ are in different chambers of $\Gamma$ 
the hyperk\"ahler structures on $X(\alpha^\pm)$ are 
not the same. Moreover, if we equip these spaces with the $S^1$-action defined in  \eqref{action} we see that $X(\alpha^+)$ and $X(\alpha^-)$ are not isomorphic as Hamiltonian $S^1$-spaces since their cores 
 $\mathfrak L_{\alpha^\pm}$ are different. The transformations suffered by $X(\alpha^\pm)$ and its core will be studied in Section~\ref{sec:wchyper}.

Another difference in the behavior of hyperpolygon spaces is that, even though $M(\alpha)=\varnothing$ for every value 
of $\alpha$ in a null chamber, the corresponding 
hyperpolygon space $X(\alpha)$ is always non empty as we can see in Example~\ref{esempio1}.

\begin{example}\label{esempio1} Let $\alpha=(10,1,1,2,3)$ be in the null chamber of $\Gamma$ determined by the vanishing wall $W_{\{1\}}$. 
The polygon space $M(\alpha)$ is empty since $\alpha_1> \sum_{i=2}^5 \alpha_i$. However,
the hyperpolygon space $X(\alpha)\neq \varnothing$. For example, taking the short set $S= \{ 4,5 \}$,  we see that the core component 
$U_{\{ 4,5 \} } \subset X(\alpha) $ is non empty. Indeed, it can be identified with the moduli space of pairs of polygons as depicted in Figure \ref{hyperpolygon}  (cf. Theorem \ref{hp}).
\begin{figure}[htbp]
\begin{center}
\psfrag{2}{\small{$2$}}
\psfrag{3}{\small{$3$}}
\psfrag{10}{\small{$10$}}
\psfrag{1}{\small{$1$}}
\psfrag{u}{\small{$u_4=-u_5$}}
\psfrag{M}{$k$}
\includegraphics[width=2cm]{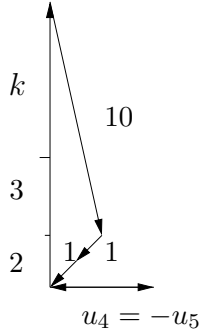}
\caption{A hyperpolygon in the core component $U_{\{ 4,5 \} }$ for $\alpha=(10,1,1,2,3)$.}
\label{hyperpolygon}
\end{center}
\end{figure}
The spatial polygon has edges $w, v_1, v_2, v_3$ respectively of lengths $5+k, 10, 1, 1$ with $k \in [3, 7]$. (For 
$k > 7 $ or $k<3$ the polygon would not close.)  The planar polygon lies on a line and has edges $u_4, u_5$ with $u_4=- u_5$ satisfying 
\begin{equation} \label{ex1}
5+k = \sqrt{4+\|u_4\|^2} + \sqrt{9+\|u_4\|^2}.
\end{equation}
 Any choice of $\|u_4\|$ satisfying \eqref{ex1} for some $k \in [3,7]$ determines a 
family of hyperpolygons in $U_{\{4,5\}}$ that is isomorphic to the polygon space $M(\|w\|, 10, 1, 1).$ For example, 
choosing $\|u_4\|=4,$ we get that $U_{\{4,5\}}$ contains the non-empty polygon space $M(5+2\sqrt{5}, 10, 1,1).$
\end{example}

\subsection{Moduli spaces of parabolic Higgs bundles}
Let $\Sigma$ be a connected smooth projective algebraic curve of genus $g$ with $n$ distinguished marked points $x_1, \ldots, x_n$ and let $D$ be the divisor $x_1+ \cdots + x_n$. A \emph{parabolic structure} on a holomorphic bundle $E \to \Sigma$ consists of weighted flags
\begin{align*}
E_x & = E_{x,1} \supset \cdots \supset E_{x,s_x} \supset 0, \\
0 & \leq \beta_1(x) < \cdots < \beta_{s_x} (x)< 1
\end{align*}
over each point $x\in D$. Given two parabolic bundles $E,F$ over $\Sigma$ with parabolic structures at $x_1, \ldots, x_n$ and weights $\beta_i^E(x)$ and $\beta_j^F(x)$ respectively, a holomorphic map $\phi:E \to F$  is called \emph{parabolic} if  $\phi(E_{x,i}) \subset F_{x, j+1}$ whenever $\beta_i^E(x) > \beta_j^F(x)$ and   \emph{strongly parabolic} if $\phi(E_{x,i}) \subset F_{x, j+1}$  whenever $\beta_i^E(x) \geq \beta_j^F(x)$. 

Let $Par Hom(E, F)$ and $S Par Hom(E, F)$ be the subsheaves of $Hom(E,F)$ formed by the parabolic and strongly parabolic morphisms between $E$ and $F$, respectively. In particular, $Par End(E):= Par Hom (E,E)$ and $S Par End(E):= S Par Hom (E,E)$.

Considering  $m_i(x):= \dim E_{x,i} - \dim E_{x, i+1}$, the \emph{multiplicity} of the weight $\beta_i(x)$, one defines the \emph{parabolic degree} $\text{pdeg}(E)$ and \emph{parabolic slope} $\mu (E)$ of a parabolic bundle $E$ as 
$$
\text{pdeg} (E)= \deg (E) +  \sum_{x \in D} \sum_{i=1}^{s_x} m_i(x) \beta_i(x),
$$
and
$$
\mu (E)  = \frac{ \text{pdeg} (E)}{\text{rank}(E)}.
$$
A subbundle $F$ of a parabolic bundle $E$ can be given a parabolic structure by intersecting the flags with the fibers $F_{x}$, and discarding any subspace $E_{x,j} \cap F_{x}$ which coincides with $E_{x,j+1} \cap F_{x}$. The weights are assigned accordingly. Similarly, the quotient $E/F$ can be given a parabolic structure by projecting the flags to $E_{x}/F_{x}$. The weights of $E/F$ are precisely those discarded for $F$.

A parabolic bundle $E$ is said to be \emph{semistable} if $\mu(F) \leq \mu(E)$ for all proper parabolic subbundles $F$ of $E$ and \emph{stable} if the inequality is always strict. 
\begin{example}
\label{ex:1}
We will now consider a very simple example which we will need later. Let $E$ be a rank-two parabolic bundle over $\Sigma$ with parabolic structure
\begin{align*}
\mathbb{C}^2 & = E_{x,1} \supset E_{x,2} = \mathbb{C} \supset 0, \\
0 & \leq \beta_1(x) < \beta_2(x) < 1
\end{align*}
over each point $x \in D$. Then
$$
\text{pdeg}(E)=\deg(E) + \sum_{x \in D}^n \left( \beta_1(x) + \beta_2(x)\right).
$$
If $L$ is a parabolic line subbundle of $E$, its parabolic structure is  given by the trivial flag over each point of $D$
$$
\mathbb{C} = L_{x,1} \supset 0,
$$
with weights 
$$
\beta^L(x) =  \left\{ \begin{array}{l} \beta_1(x), \quad \text{if} \quad L_x \cap E_{x,2} = \{ 0 \},  \\ \\ \beta_2(x),  \quad \text{if} \quad L_x \cap E_{x,2} = \mathbb{C}. \end{array} \right.
$$
Then, assuming $D=\{x_1, \ldots, x_n\}$,
$$
\text{pdeg}(L)=\deg(L) + \sum_{i\in S_L} \beta_2(x_i) + \sum_{i\in S_L^c} \beta_1(x_i),
$$
where $S_L:=\{ i \in \{1, \ldots, n\} \mid\,  \beta^L(x_i)=\beta_2(x_i)\}$.
(Note that the quotient bundle $E/L$ is also a parabolic line bundle over $\Sigma$ with parabolic structure given by the trivial flag over each point of $D$ weighted by the weights of $E$ not used in $L$.)

Hence, the parabolic bundle $E$ is stable if and only if its parabolic line subbundles $L$ satisfy
\begin{equation}
\label{eq:stable}
\deg E - 2\deg(L) > \sum_{i \in S_L} \big(\beta_2(x_i) -  \beta_1(x_i)\big) - \sum_{i \in S_L^c} \big(\beta_2(x_i) -  \beta_1(x_i)\big).
\end{equation}
\end{example}

Let $K_\Sigma$ denote the canonical bundle over $\Sigma$ (i.e. the bundle of holomorphic $1$-forms in $\Sigma$), let $\mathcal{O}_\Sigma(D)$ be the line bundle over $\Sigma$ associated to the divisor $D$ and give $E\otimes K_\Sigma(D):= E \otimes K \otimes \mathcal{O}_\Sigma(D)$ the obvious parabolic structure. A \emph{parabolic Higgs bundle} or PHB is a pair ${\bf E} := (E, \Phi)$, where $E$ is a parabolic bundle and 
$$\Phi \in H^0(\Sigma, S Par End(E) \otimes K_\Sigma(D))$$ 
is called an \emph{Higgs field} on $E$. Note that $\Phi$ is a meromorphic, endomorphism-valued one-form with simple poles along $D$, whose residue at $x$ is nilpotent with respect to the flag, i.e. 
$$
(\text{Res}_x \Phi) (E_{x,i}) \subset E_{x,i+1} ,
$$
for all $i=1, \ldots, s_x$ and $x\in D$. The definitions of stability and semistability are extended to Higgs bundles as expected. A PHB ${\bf E}= (E,\Phi)$ is \emph{stable} if
$\mu(F) <  \mu(E)$ for all proper parabolic subbundles $F\subset E$ which are preserved by $\Phi$ and similarly for \emph{semistability}, where the strict inequality is substituted by the weak inequality.

The usual properties of stable bundles also apply to stable parabolic Higgs bundles. For instance, if ${\bf E}$ and ${\bf F}$ are two stable PHBs  then there are no parabolic maps between them unless they are isomorphic  \cite{K}  (in which case they must have the same parabolic slope) and the only parabolic endomorphisms of a stable parabolic Higgs bundle are the scalar multiples of the identity.

We will say that a vector $\beta$ of  weights $\beta_i(x_j)$ is \emph{generic} when every semistable parabolic Higgs bundle is stable (i.e. if there are no properly semistable Higgs bundles).  Fixing a generic $\beta$ and the topological invariants  $r=\text{rank} (E)$ and $d=\deg(E)$, the moduli space $\mathcal{N}_{\beta,r,d}$ of $\beta$-stable,  rank-$r$,  degree-$d$ parabolic Higgs bundles  was constructed by Yokogawa in \cite{Y} using GIT. In particular, he shows that this space is a smooth irreducible complex variety of dimension
$$
\dim \mathcal{N}_{\beta,r,d}= 2 (g-1)r^2 + 2 + \sum_{i=1}^n \left(r^2 - \sum_{j=1}^{s_{x_i}} m_j(x_i)^2 \right),
$$
containing the cotangent bundle of the moduli space of stable parabolic bundles. For that, he worked out a deformation theory for PHBs as described next (see also \cite{GGM} for details). 

\subsubsection{Deformation Theory}

Given PHBs ${\bf E}=(E,\Phi)$ and ${\bf F}=(F, \Psi)$ one defines a complex of sheaves
\begin{align*}
C^{\bullet} ({\bf E}, {\bf F}): Par Hom(E,F) & \to S Par Hom(E,F) \otimes K_\Sigma(D) \\
 f & \mapsto (f \otimes 1)\Phi - \Psi f,
\end{align*} 
and write $C^\bullet({\bf E}):=C^\bullet({\bf E},{\bf E})$. Then the following proposition holds (see for instance \cite{T} for a detailed proof).
\begin{proposition}
\label{prop:GP}
\begin{enumerate}
\item The space of infinitesimal deformations of  a PHB ${\bf E}$ is isomorphic to the first hypercohomology group of the complex
$C^{\bullet} ({\bf E})$. Consequently the tangent space to $\mathcal{N}_{\beta,r,d}$ at a point ${\bf E}$ is isomorphic to $\mathbb{H}^1(C^\bullet(E))$.
\item The space of homomorphisms between PHBs ${\bf E}$ and ${\bf F}$ is isomorphic to the hypercohomology group $\mathbb{H}^0(C^\bullet ({\bf E}, {\bf F}))$.
\item The space of extensions $0 \to {\bf E} \to {\bf F} \to {\bf G} \to 0$ of PHBs ${\bf E}$ and  ${\bf G}$ is isomorphic to the hypercohomology group $\mathbb{H}^1(C^\bullet({\bf G}, {\bf E}))$.
\item There is a long exact sequence 
\begin{align*}
0 & \to \mathbb{H}^0(C^\bullet({\bf E}, {\bf F})) \to H^0(Par Hom (E,F)) \to H^0(S Par Hom(E,F)\otimes K_\Sigma(D)) \to \\
  & \to \mathbb{H}^1(C^\bullet({\bf E}, {\bf F})) \to H^1(Par Hom (E,F)) \to H^1(S Par Hom(E,F)\otimes K_\Sigma(D)) \to \\
  & \to \mathbb{H}^2(C^\bullet({\bf E}, {\bf F})) \to 0.
\end{align*}
\end{enumerate}
\end{proposition}
Moreover, we have the following duality result whose proof can be found in \cite{GGM}.
\begin{proposition}
\label{prop:Serre}
If ${\bf E}$ and ${\bf F}$ are PHBs then there exists a natural isomorphism
$$
\mathbb{H}^i(C^\bullet({\bf E}, {\bf F})) \cong \mathbb{H}^{2-i}(C^\bullet ({\bf F}, {\bf E}))^*.
$$   
In particular for any stable PHB ${\bf E}$ there is a natural isomorphism $T_{\bf E} \mathcal{N}_{\beta,r,d} \cong T_{\bf E}^* \mathcal{N}_{\beta,r,d}$.  
\end{proposition}

\subsubsection{Fixed determinant}\label{sec:fixed determinant}

If ${\bf E} \in \mathcal{N}_{\beta,r,d}$ and $E$ is the underlying parabolic bundle, its determinant $\Lambda^rE$ is a parabolic line bundle of degree
$$
\widetilde{d}= d + \sum_{i = 1}^{n} \left[\sum_j m_j(x_i) \beta_j(x_i)\right]
$$
and weight $\sum_j m_j(x) \beta_j(x) - \left[\sum_j m_j(x) \beta_j(x)\right]$, at any $x\in D$, where the square brackets denote the integer part. For fixed weights the moduli space of rank-$1$ parabolic Higgs bundles of degree $\widetilde{d}$ is naturally identified with the total space of the cotangent bundle to the Jacobian of degree-$\widetilde{d}$ line bundles on $\Sigma$.  Hence one has the map
\begin{align}\label{eq:mapdet}
\det: \mathcal{N}_{\beta,r,d} & \to T^* Jac^{\tilde{d}}(\Sigma), \\ \nonumber
(E,\Phi) & \mapsto (\Lambda^r E, \text{Tr}\,  \Phi).
\end{align}
Fixing $\Lambda$, a  line bundle of degree $\widetilde{d}$,  Konno \cite{K} defines the moduli space $\mathcal{N}^{0,\Lambda}_{\beta,r,d}$ of stable parabolic Higgs bundles with fixed determinant $\Lambda$ and trace-free Higgs field as the fibre of the map \eqref{eq:mapdet} over $(\Lambda, 0)$ i.e.
$$
\mathcal{N}^{0,\Lambda}_{\beta,r,d}:= \text{det}^{-1} (\Lambda, 0).
$$
In particular, he shows  that, for any $\Lambda$ and generic $\beta$, this space is a smooth, hyperk\"{a}hler manifold of complex dimension
$$
\dim \mathcal{N}^{0,\Lambda}_{\beta,r,d} = 2(g-1) (r^2-1)  + \sum_{i=1}^n \left(r^2 - \sum_{j=1}^{s_{x_i}} m_j(x_i)^2 \right).
$$ 
The deformation theory of ${\bf E}=(E, \Phi)$ in $\mathcal{N}^{0,\Lambda}_{\beta,r,d}$ is determined by the complex
\begin{align*}
C^\bullet_0({\bf E}): Par End_0(E) & \to S Par End_0(E) \otimes K_\Sigma(D) \\
f & \mapsto (f\otimes 1)\Phi - \Phi f,
\end{align*}
where the subscript $0$ indicates trace $0$.

We will now give a brief description of $\mathcal{N}^{0,\Lambda}_{\beta,r,d}$ following \cite{K} and \cite{GGM}.  Given a PHB ${\bf E}$ of rank $r$ with underlying topological bundle $E$, one says that a  local frame  $\{e_1, \ldots, e_r\}$ around $x$ \emph{preserves the flag} at $x$ if $E_{x,i}$ is spanned by the vectors $\{e_{M_i+1}(x), \ldots, e_r(x)\}$, where $M_i=\sum_{k\leq i} m_k$.  Then one fixes a hermitian metric $h$ on $E$ which is smooth in $\Sigma\setminus D$ and whose behavior around the points in $D$ is as follows: if $z$ is a centered local coordinate around $x$ (i.e. such that $z(x)=0$), then one requires $h$ to have the form
\begin{equation}
\label{eq:metric}
h=\left(\begin{array}{ccc} \vert z\vert^{2 \lambda_1} & & 0 \\ & \ddots&  \\ 0 & & \vert z\vert^{2 \lambda_r}\end{array} \right)
\end{equation}
with respect to some local frame around $x$ which preserves the flag at $x$. Let us denote by $\mathcal{J}$ the affine space of holomorphic structures on $E$ and by $\mathcal{A}$ the space of associated $h$-unitary connections. Note that the unitary connection $A$ associated to some element $\overline{\delta}_A$ of $\mathcal{J}$ via the hermitian metric $h$ is singular at the punctures. Indeed, writing $z= \rho e^{\bold{i} \theta}$ and considering the local frame $\{e_i\}$ used in \eqref{eq:metric}, the connection $A$ has the form
 \begin{equation}
\label{eq:connection}
d_A = d + {\bf i} \left(\begin{array}{ccc}  \lambda_1 & & 0 \\ & \ddots&  \\ 0 & &  \lambda_r \end{array} \right) d\theta + A^\prime
\end{equation} 
with respect to the local frame $\{ e_i/\vert z \vert^{\lambda_i} \}$, where $A^\prime$ is regular. 
The space of trace-free Higgs fields on a parabolic bundle $E$ is
$$
{\bf \Omega}  := \Omega^{1,0} \big(S Par End_0(E) \otimes K_\Sigma(D)\big).
$$
Let $\mathcal{G}_\mathbb{C}$ denote the group of complex parabolic gauge transformations (i.e. the group of smooth determinant-$1$ bundle automorphisms of $E$ which preserve the flag structure) and let $\mathcal{G}$ denote the subgroup of $h$-unitary parabolic gauge transformations. Using  the weighted Sobolev norms defined by Biquard \cite{Bi}  on the above spaces (see \cite{Bi} and \cite{K} for details) let us denote by $\mathcal{J}^p$, ${\bf \Omega}^p$, $\mathcal{G}^p$ and $\mathcal{G}_\mathbb{C}^p$ the corresponding Sobolev completions. Following Konno we consider the space
$$
\mathcal{H} := \big\{ (\overline{\delta}_A, \Phi) \in \mathcal{J} \times {\bf \Omega}\mid\, \overline{\delta}_A \Phi =0 \big\}
$$
and the corresponding subspace $\mathcal{H}^p$ of $\mathcal{J}^p \times {\bf \Omega}^p$. 
The gauge group $\mathcal{G}_\mathbb{C}$ acts on $\mathcal{H} $ by conjugation, i.e. on the residues $N_i := Res_{x_i}  \Phi$ the $\mathcal{G}_\mathbb{C}$-action is $g^{-1} N_i g$ for any $g \in \mathcal{G}_\mathbb{C}$ (cf.  \cite{K}).
Let $F(A)^0$ be the trace-free part of the curvature of the $h$-unitary connection  corresponding to $\overline{\delta}_A$. Then we consider the moduli space $\mathcal{E}^0$ defined as the subspace of $\mathcal{H}^p$ satisfying \emph{Hitchin's equation} 
$$
\mathcal{E}^0:= \big\{(\overline{\delta}_A , \Phi) \in \mathcal{H}^p\mid\, F(A)^0 + [\Phi, \Phi^*]=0\big\}/\mathcal{G}^p.
$$
Taking the usual definition of semi-stability on $\mathcal{H}$, Konno shows in \cite{K} that, for some $p>1$,
\begin{equation}
\label{eq:quotient}
\mathcal{N}^{0,\Lambda}_{\beta,r,d} := \mathcal{H}_{ss}/ \mathcal{G}_{\mathbb{C}} \cong \mathcal{E}^0
\end{equation}
and this second quotient endows $\mathcal{N}^{0,\Lambda}_{\beta,r,d}$ with a hyperk\"{a}hler structure.

There is a natural circle action on the moduli space $\mathcal{N}^{0,\Lambda}_{\beta,r,d}$ given by
\begin{equation}
\label{eq:action}
e^{\bold{i} \theta}  \cdot (E, \Phi)= ( E, e^{\bold{i} \theta}  \Phi)
\end{equation}
which is respected by the identification in \eqref{eq:quotient}. This action is Hamiltonian with respect to the symplectic structure of $\mathcal{N}^{0,\Lambda}_{\beta,r,d}$ compatible with the complex structure induced by the complex structure
$$
I(\overline{\delta}_A, \Phi)=({\bf i} \overline{\delta}_A,{\bf i} \Phi)
$$
on $\mathcal{H}^p$ (see \cite{BY} for details). The corresponding moment map is
$$
[(A, \Phi)] \mapsto -\frac{1}{2} \vert \vert \Phi \vert \vert^2 = - \bold{i} \int_\Sigma \text{Tr} (\Phi \Phi^*).
$$
Let us consider the positive function 
\begin{equation}
\label{eq:f}
f:=\frac{1}{2} \vert \vert \Phi \vert \vert^2.
\end{equation}
Boden and Yokogawa in \cite{BY} show that this map is proper. By a general result of Frankel \cite{F} which states that a proper moment map of a circle action on a K\"{a}hler manifold is a perfect Morse-Bott function, we conclude that $f$ is Morse-Bott.  Its critical set (which corresponds to the fixed point set of the circle action) was studied by Simpson in \cite{S} who shows the following result.
\begin{proposition}\textrm{(Simpson)} \label{prop:Simpson} The equivalence class of a stable PHB ${\bf E}= (E,\Phi)$ is fixed by the $S^1$-action \eqref{eq:action} if and only if $E$ has a direct sum decomposition
$$
E=E_0 \oplus \cdots \oplus E_m
$$
as parabolic bundles, such that $\Phi$ is strongly parabolic and of degree one with respect to this decomposition, i.e., 
$$
\Phi_{\vert_{E_l}} \in H^0(S Par Hom (E_l, E_{l+1}) \otimes K_\Sigma(D)).
$$
Moreover, stability implies that $\Phi_{\vert_{E_l}}\neq 0$ for $l=0, \ldots, m-1$, and ${\bf E}= (\bigoplus_l E_l, \Phi)$ is stable as a parabolic Higgs bundle if and only if the stability condition is satisfied for all  proper parabolic subbundles which respect the decomposition $E=\bigoplus_l E_l$ and are preserved by $\Phi$.
\end{proposition}

\begin{Remark}
\label{rmk:1}
Note that if $m=0$, then $E=E_0$ and $\Phi=0$ and one obtains the fixed points $(E,0)$, where $E$ is a stable parabolic bundle. Hence the moduli space  $\mathcal{M}_{\beta,r,d}^{0, \Lambda}$ of $\beta$-stable rank-$r$ parabolic bundles of fixed degree and determinant is a component of the fixed-point set.
\end{Remark}

The Morse index of a critical point of $f$, which equals the dimension of the negative weight space of the circle action on the tangent space at the fixed point (see \cite{F}), was computed by Garc\'{i}a-Prada, Gothen and Mu\~{n}oz:

\begin{proposition} \cite{GGM}
\label{prop:morseindex}
Let the PHB ${\bf E}=(\oplus_{l=0}^m E_l , \Phi)$ represent a critical point of $f$. Then the Morse index of $f$ at ${\bf E}$ is given by
\begin{align*}
\lambda_{{\bf E}} & = 2r^2(g-1) + \sum_{i=1}^n \Bigl(r^2 - \sum_{j=1}^{s_{x_i}} m_j(x_i)^2\Bigr) \\ & + 2\sum_{l=0}^m \Bigl((1-g-n) \mathrm{rank}(E_l)^2   + \sum_{i=1}^n \dim P_{x_i}(E_l,E_l)\Bigr) \\ & + 2 \sum_{l=0}^{m-1}\Bigl((1-g)\mathrm{rank}(E_l)\mathrm{rank}(E_{l+1})  - \mathrm{rank}(E_l) \deg (E_{l+1}) \Bigr. \\ &  \Bigl.\hspace{2cm}  + \mathrm{rank}(E_{l+1}) \deg (E_{l})   - \sum_{i=1}^n \dim N_{x_i}(E_l,E_{l+1})\Bigr), 
\end{align*}
where, given two parabolic bundles $F$ and $G$, $P_x(F,G)$ denotes the subspace of $Hom(F_x,G_x)$ formed by parabolic maps, and $N_x(F,G)$ denotes the subspace of strongly parabolic maps. 
\end{proposition}
\subsubsection{The rank-two situation}\label{sec:rk2}

Let us now restrict ourselves to the {\bf rank two} situation. Most of what is presented in this section is essentially contained in \cite{BY} but we will  give an exposition adapted to our purposes. 

If ${\bf E}=(E,\Phi)$ is a fixed point of the circle action defined in \eqref{eq:action} then we have two possible cases:
\begin{enumerate}
\item $E$ is a stable rank-$2$ parabolic  bundle and $\Phi=0$ (see Remark~\ref{rmk:1});
\item $E= E_0 \oplus E_1$ where $E_0$ and $E_1$ are parabolic line bundles and $\Phi$ induces a strongly parabolic map 
$$\Phi_0 := \Phi_{\vert_{E_0}}: E_0 \to E_1 \otimes K_\Sigma(D).$$ 
\end{enumerate}
In the first case, the corresponding critical submanifold can be identified with the moduli space  $\mathcal{M}_{\beta,2,d}^{0, \Lambda}$ of ordinary rank-$2$ parabolic bundles of fixed degree and determinant and it is the only critical component where the Morse-Bott function $f$ takes its minimum value $f=0$.

The fixed points in the second situation occur when $e^{{\bf i} \theta}\cdot (\overline{\delta}_A,\Phi)$ is gauge equivalent to $(\overline{\delta}_A, \Phi)$. In particular, this implies that there exists a $1$-parameter family  $g_\theta\in \mathcal{G}^p$ such that $g_\theta^{-1} \Phi g_\theta = e^{{\bf i}\theta} \Phi$ which  is diagonal with respect to the decomposition $E=E_0\oplus E_1$ (in fact the splitting of the holomorphic parabolic bundle $E$ is determined by the eigenvalues of $g_\theta$). Hence $\Phi$ is either strictly upper or lower triangular, meaning that one of $E_0$ or $E_1$ is $\Phi$-invariant. Since we also have that $\Phi_0 := \Phi_{\vert_{E_0}}$ is a map from $E_0$ to $E_1 \otimes K_\Sigma(D)$, we conclude that
$$
\Phi= \left(\begin{array}{cc} 0 & 0 \\ \phi & 0\end{array} \right),
$$ 
with $0\neq \phi \in S Par Hom(E_0,E_1 \otimes K_\Sigma(D))$. Then $E_1$ is preserved by $\Phi$ which, by $\beta$-stability of ${\bf E}$, implies  that $\mu (E_1) < \mu(E)$. By Example~\ref{ex:1} this is equivalent to requiring
\begin{equation}
\label{eq:stability}
\deg E - 2 \deg E_1 > \sum_{i \in S_{E_1}} \big(\beta_2(x_i) - \beta_1(x_i)\big) - \sum_{i \in S_{E_1}^c} \big(\beta_2(x_i) - \beta_1(x_i)\big),
\end{equation}
where $0\leq \beta_1(x_i) < \beta_2 (x_i)<1$ are  the parabolic weights of $E$ at $x_i\in D$ and 
$$
S_{E_1}=\big\{i \in \{1, \ldots, n\}\mid\, \beta^{E_1}(x_i)=\beta_2(x_i)\big\}
$$ 
with $0\leq \beta^{E_1}(x_1) <1$ the weight of $E_1$ at $x_i$. 

On the other hand, the existence of a strongly parabolic map
$$0\neq \Phi_0 := \Phi_{\vert_{E_0}}: E_0 \to E_1 \otimes K_\Sigma(D)$$ 
implies that
$$
H^0 \Big(S Par Hom \big(E_0,E_1 \otimes K_\Sigma(D)\big)\Big) \neq 0. 
$$
Moreover,
$$
S Par Hom\big(E_0, E_1 \otimes K(D)\big) \cong Hom \left(E_0, E_1 \otimes K\left(D \setminus \cup_{i \in S_{E_1^c}} \{x_i\} \right) \right),
$$
since, denoting the parabolic weights of $E_0$ and $E_1$ at $x_i$ respectively by $\beta^{E_0}(x_i)$ and  $\beta^{E_1}(x_i)$, we have  
\begin{align*}
S_{E_0} = S_{E_1}^c & = \big\{ i \in \{1, \ldots, n\} \mid \, \beta^{E_0}(x_i) = \beta_2(x_i)\big\}  \\ & = \big\{ i \in \{1, \ldots, n\} \mid \, \beta^{E_0}(x_i) > \beta^{E_1}(x_i)\big\}.
\end{align*}
Hence, a necessary condition for $(E_0 \oplus E_1, \Phi)$ to be a critical point is that
\begin{align*}
0   \leq \deg Hom &  \Big( E_0, E_1 \otimes K_\Sigma\big( D \setminus \cup_{i \in S_{E_1}^c} \{x_i\} \big) \Big) \\ & = \deg \Big( E_0^* \otimes  E_1 \otimes K_\Sigma\big(D \setminus \cup_{i \in S_{E_1}^c} \{x_i\} \big) \Big) 
\\ &  =\deg \Big( E_0^* \otimes E_1 \otimes  K_\Sigma \otimes \mathcal{O}_\Sigma\big(D \setminus \cup_{i \in S_{E_1}^c} \{x_i\} \big) \Big)  \\ & =  \deg( E_1) -\deg( E_0) + 2(g-1) + \big\lvert  D \setminus \cup_{i \in S_{E_1}^c} \{x_i\} \big\rvert \\ & =  \deg( E_1) -\deg( E_0) + 2(g-1) + n - \vert  S_{E_1}^c \vert \\ & = \deg( E) - 2\deg( E_0) + 2(g-1) + \vert S_{E_1} \vert,
\end{align*}
where we used the fact that $\deg K_\Sigma = 2(g-1)$ and that, for any divisor $\widetilde{D}=\sum_{x \in \Sigma} n_x\, x$, we have  
$$\deg \mathcal{O}_\Sigma(\widetilde{D})= \deg(\widetilde{D})= \sum_{x \in \Sigma}n_x.$$
Using \eqref{eq:stability} we conclude that if $(E_0 \oplus E_1, \Phi)$ is a critical point then
$$
\varepsilon_{S_{E_1}}(\beta_2-\beta_1) + d < 2d_0 \leq d + 2(g-1) +  \vert S_{E_1} \vert,
$$
where $d_0=\deg E_0$, $d=\deg E$, $\beta_2 -\beta_1$ is the vector 
$$
\big(\beta_2(x_1)-\beta_1(x_i), \ldots, \beta_2(x_n)-\beta_1(x_n)\big)
$$ 
and $\varepsilon_{S_{E_1}}(\beta_2-\beta_1)$ is the sum defined in \eqref{eq:epsilon}.

Given $S\subset \{1,\ldots, n\}$ and $d_0\in \mathbb{Z}$,  let $\mathcal{M}_{(d_0,S)}$ be the critical submanifold  formed by parabolic Higgs bundles ${\bf E}=(E_0 \oplus E_1, \Phi)\in \mathcal{N}^{0,\Lambda}_{\beta,2,d}$, where $E_0$ is a parabolic line bundle of topological degree $d_0$ and parabolic weights $\beta^{E_0}$ satisfying  
$S_{E_0}=S^c$ (i.e. $\beta^{E_0}(x_i)=\beta_2(x_i)$ if and only if $i \in S^c$).  Then
\begin{proposition} 
\label{prop:cond}
Given $S\subset \{1,\ldots, n\}$ and $d_0\in \mathbb{Z}$,  the critical submanifold $\mathcal{M}_{(d_0,S)}\subset \mathcal{N}^{0,\Lambda}_{\beta,2,d}$ is nonempty if and only if
\begin{equation}
\label{eq:cond}
\varepsilon_S(\beta_2-\beta_1) + d < 2 d_0 \leq d + 2(g-1) + \vert S \vert.
\end{equation}

Moreover, denoting by $\widetilde{S}^m \Sigma$ the $2^{2g}$ cover of the symmetric product $S^m \Sigma$ under the map $x \mapsto 2x$ on $\text{Jac}(\Sigma)$, the map
\begin{align} \label{eq:map2}
\mathcal{M}_{(d_0,S)} & \to \widetilde{S}^m \Sigma, \\ \nonumber
(E_0 \oplus E_1, \Phi) & \mapsto (E_0, \text{div}\, \Phi_0)
\end{align}
is an isomorphism for 
$$
m= d - 2d_0 + 2(g-1) + \vert S\vert,
$$
where $\text{div} \, \Phi_0$ (the zero set of $\Phi_0:=\Phi_{\vert_{E_0}}$) is a non-negative divisor of degree $m$.
\end{proposition}
\begin{proof}
The discussion preceding this statement shows that \eqref{eq:cond} is necessary for $\mathcal{M}_{(d_0,S)}$ to be nonempty.

Suppose now that a pair $(d_0, S)$ satisfies \eqref{eq:cond}. Given an effective divisor $D_m\in S^m \Sigma$ with $m= d - 2d_0 + 2(g-1) + \vert S\vert$ one gets a line bundle $\mathcal{O}_\Sigma(D_m)$ with a nonzero section $\Phi_0$ determined up to multiplication by a nonzero scalar, as well as the bundle 
$$
U := K_\Sigma \otimes \mathcal{O}_\Sigma(\cup_{i \in S} \,\{x_i\}) \otimes \mathcal{O}_\Sigma(-D_m)
$$
of degree $2d_0 - d$. Then,  one can choose a line bundle $L_0 \in \text{Jac}^{d_0}(\Sigma)$, such that
\begin{equation}
\label{eq:lzero}
L_0^{\otimes 2} = \Lambda \otimes U
\end{equation}
and  equip it with the parabolic structure given by the trivial flag over each point $x_i \in D$ and the  weight assignment
$$
\beta^{L_0}(x_i)= \left\{ \begin{array}{l} \beta_1(x_i), \,\, \text{if} \,\, i\in S \\ \\ \beta_2(x_i), \,\, \text{if} \,\, i\in \{1, \ldots, n\} \setminus S. \end{array} \right. 
$$
In addition one considers the bundle 
$$
L_1 : = L_0 \otimes U^*
$$
equipped with the complementary parabolic structure. Defining $\Phi$ to have component $\Phi\vert_{L_0}=\Phi_0$ one  obtains a PHB ${\bf E}=(L_0 \oplus L_1, \Phi)$ which clearly has the desired invariants $(d_0, S)$, has the required determinant (since $\Lambda^2 (L_0 \oplus L_1) = L_0 \otimes L_1=L_0^{\otimes 2} \otimes U^* = \Lambda$) and is stable if \eqref{eq:cond} is satisfied. Hence  \eqref{eq:cond} is a sufficient condition for $\mathcal{M}_{(d_0,S)}$ to be nonempty. Note that there exist $2^{2g}$ possible choices of $L_0$ satisfying \eqref{eq:lzero} (since the $2$-torsion points in the Jacobian form a group 
$$
\Gamma_2=\{ L\mid\, L^{\otimes 2} =\mathcal{O} \}
$$
isomorphic to $\mathbb{Z}^{2g}$), and that each choice gives a stable PHB. Hence, the map \eqref{eq:map2} is surjective. 

To see that it is injective we note that by taking non-zero scalar multiples of the Higgs field $\Phi_0 \in H^0(L_0^* \otimes L_1 \otimes K(\cup_{i \in S} \, \{x_i\}) )$ (in order to obtain the same divisor $\textit{div}\, \Phi$) one obtains two isomorphic PHBs since $(E, \Phi)$ is gauge equivalent to $(E, \lambda \Phi)$ for $\lambda \neq 0$.
\end{proof}
To compute  the Morse index at the points in $\mathcal{M}_{(d_0, S)}$ we use Proposition~\ref{prop:morseindex} to obtain the following proposition.
\begin{proposition}\label{prop:index}
The index of the critical submanifold $\mathcal{M}_{(d_0, S)}$ is
$$
\lambda_{(d_0, S)}= 2(g-1+n)+ 4d_0 - 2 d -2 \vert S\vert.
$$
\end{proposition}
\begin{proof}Noting that all the multiplicities are equal to $1$ and that $s_x=2$ for every point in $D$, the proof follows from Proposition~\ref{prop:morseindex} after we compute the dimensions of the spaces $P_x(E_l,E_l)$, $l=0,1$, and $N_x(E_0,E_1)$ for every point $x\in D$. The space $P_x(E_l,E_l)$ is formed by the parabolic endomorphisms of $(E_l)_x$ and so, in this case,
$$
\dim P_x(E_l,E_l)= \dim End((E_l)_x) = 1.
$$ 
The space $N_{x_i}(E_0,E_1)$ is the space of strongly parabolic maps from $(E_0)_{x_i}$ to $(E_1)_{x_i}$ and so
$$
N_{x_i}(E_0,E_1) = \left\{ \begin{array}{l} 0, \,\, \text{if}\,\, \beta^{E_0}(x_i) > \beta^{E_1}(x_i) \\ \\ Hom((E_0)_{x_i},(E_1)_{x_i}), \,\, \text{otherwise}. \end{array} \right.
$$ 
Hence,
$$
\dim N_{x_i}(E_0,E_1) = \left\{ \begin{array}{l} 0, \,\, \text{if}\,\, i \in \{1, \ldots, n\} \setminus S \\ \\ 1, \,\, \text{if}\,\, i \in S. \end{array} \right.
$$
\end{proof}
With this we have the following proposition.
\begin{proposition}\label{prop:vanishing}
\begin{enumerate}
\item If $g\geq 1$ then $\lambda_{(d_0,S)}>0$ for all $(d_0, S)$ satisfying \eqref{eq:cond}. 
\item If $g=0$ and $n\geq 3$ then there is at most one pair $(d_0,S)$ satisfying \eqref{eq:cond} with $\lambda_{(d_0,S)}=0$. Moreover, this pair exists if and only if $\mathcal{M}_{\beta,2,d}^{0,\Lambda}=\varnothing$ and, in this case, $\mathcal{M}_{(d_0,S)}=\mathbb{C} \P^{n-3}$.
\end{enumerate}
\end{proposition}
\begin{proof}
If $\lambda_{(d_0,S)}=0$ then $2d_0 = 1-g-n+d+\vert S\vert$. Since, from \eqref{eq:cond}, we have $2d_0> \varepsilon_S(\alpha) + d$, with $\alpha=\beta_2-\beta_1$, we conclude that
$\varepsilon_S(\alpha) < 1-g-n+\vert S\vert$. Moreover, since by definition
$$
\varepsilon_S(\alpha)= \sum_{i \in S} \alpha_i - \sum_{i \in S^c} \alpha_i
$$ 
and $0 < \alpha_i < 1$, we have $\varepsilon_S(\alpha) >-\vert S^c \vert = \vert S \vert - n$ and so
$$
\vert S \vert - n < \varepsilon_S (\alpha) < 1-g-n+\vert S\vert,
$$ 
implying that $0 < 1-g$ and thus $g=0$. 

Let us assume now that $g=0$. Then \eqref{eq:cond}  and $\lambda_{(d_0,s)}=0$ imply that 
$$\vert S \vert - n < \varepsilon_S(\alpha) < 1+ \lvert S \rvert -n,$$ 
and so 
$$
0 < \sum_{i \in S} \alpha_i - \sum_{i \in S^c} \alpha_i + \vert S^c \vert < 1, 
$$
which is equivalent to
\begin{equation}
\label{eq:scomp}
0  <  \sum_{i \in S} \alpha_i + \sum_{i \in S^c} (1- \alpha_i) < 1,
\end{equation}
with the advantage that  now all the summands in \eqref{eq:scomp} are positive. If $\lambda_{(d_0^\prime, S^\prime)} = 0$ for some other $(d_0^\prime, S^\prime) \neq (d_0, S)$ then
$$
2(d_0^\prime - d_0) = \vert S^\prime \vert - \vert S \vert
$$
and so $\vert S^\prime \vert - \vert S \vert$ is even. This implies that there exist at least two indices  in $S\cup S^\prime$ that are not in $S^\prime \cap S$ and so  
$$\big\lvert (S\cup S^\prime) \cap (S\cap S^\prime)^c\big\rvert = \big\lvert (S^\prime \cup S)\cap \big((S^\prime)^c \cup S^c\big)  \big\rvert \geq 2.$$ 
Hence, since  both $S$ and $S^\prime$ satisfy \eqref{eq:scomp} we have that
$$
2 < \sum_{i \in S^\prime \cup S} \alpha_i + \sum_{i \in (S^\prime)^c \cup S^c} (1- \alpha_i) < 2. 
$$ 
which is impossible. Hence there is at most one pair $(d_0,S)$ satisfying \eqref{eq:cond} with $\lambda_{(d_0,S)}=0$.

Still assuming $g=0$, one has from Proposition~\ref{prop:cond} that
$$
\mathcal{M}_{(d_0,S)} \cong S^m \mathbb{C} \P^1 \cong \mathbb{C} \P^m 
$$
with $m= d-2d_0 - 2 + \vert S \vert$. In particular, if $\lambda_{(d_0,S)} =0$, we have that $m=n-3$ and so $\mathcal{M}_{(d_0,S)} \cong\mathbb{C} \P^{n-3}$.

To show that such a pair exists if and only if $\mathcal{M}_{\beta,2,d}^{0,\Lambda}=\varnothing$ we first define for any $(d_0,S)$  the hyperplane 
$$
H_{(d_0, S)}=\{(\beta_1, \beta_2)\in Q \mid\, \varepsilon_S(\alpha) + d = 2d_0\},
$$
where $Q:=\{ (\beta_1, \beta_2) \in \mathbb{R}^{2n} \mid \, 0 <\beta_{1,i} < \beta_{2,i} < 1,\, i=1, \ldots,n \}$ is the so-called \emph{weight space}. Boden and Hu show in \cite{BH} that, if  $\beta$ and $\beta^\prime$ are weights in adjacent connected components of $Q \setminus \cup_{(d_0,S)} H_{(d_0,S)}$,  (usually called \emph{chambers})  then the corresponding moduli spaces are related by a special birational transformation which is similar to a flip in Mori theory which will be studied in detail in Section~\ref{wall crossing}. Moreover, when $g=0$, there exist \emph{null chambers} formed by weights $\beta\in Q$ for which $\mathcal{M}_{\beta,2,d}^{0,\Lambda}=\varnothing$. Let $\beta$ and $\beta^\prime$ be weights on either side of a (unique) hyperplane separating a null chamber from the rest (called a \emph{vanishing wall}), and let $\delta$ be a weight on this hyperplane. Then, assuming $\mathcal{M}_{\beta^\prime,2,d}^{0,\Lambda}=\varnothing$, Boden and Hu show that there exists a canonical projective map
$$
\phi: \mathcal{M}_{\beta,2,d}^{0,\Lambda} \to \mathcal{M}_{\delta,2,d}^{0,\Lambda} 
$$  
which is a fibration with fiber $\mathbb{C} \P^a$, where $a=\dim \mathcal{M}_{\beta,2,d}^{0,\Lambda} - \dim \mathcal{M}_{\delta,2,d}^{0,\Lambda}=n-3$. Moreover, $\mathcal{M}_{\delta,2,d}^{0,\Lambda}$ consists of classes of strictly semistable bundles $E=L\oplus F$ for parabolic line bundles $L$ and $F$ with $S_F=S$ and $\deg(L)=d_0$.  Assuming, without loss of generality, that $\varepsilon_{S_F}(\widetilde{\beta}) > \varepsilon_{S_F}(\widetilde{\delta}) > \varepsilon_{S_F}(\widetilde{\beta}^\prime)$, the fact that $\mathcal{M}_{\beta^\prime,r,d}^{0,\Lambda}=\varnothing$ implies that there are no nontrivial extensions of $L$ by $F$, when regarded with weight $\beta^\prime$, i.e. $Par Ext^1_{\beta^\prime}(L,F)=0$ (cf. \cite{BY} for details). Then, the short exact sequence of sheaves
$$
0 \to Par Hom (L,F) \to Hom (L,F) \to Hom(L_D, F_D) /P_D(L,F) \to 0,
$$
(where, denoting by $P_x(L,F)$ the subspace of $Hom(L_x,F_x)$ consisting of parabolic maps, we write $P_D(L,F)=\oplus_{x \in D} P_x(L,F)$), gives us
\begin{align}
\label{seq:par}
\nonumber \chi\big( Par Hom (L,F)\big)& =\chi \big(Hom(L,F)\big) - \chi\big(Hom(L_D, F_D) /P_D(L,F)\big) \\ & = \chi \big(Hom(L,F)\big)+ \sum_{i=1}^n (\dim P_{x_i} -1).
\end{align}
Moreover, since $H^0\big(Par Hom_{\beta^\prime} (L,F)\big)=0$, 
\begin{align*}
0& = \dim Par Ext^1_{\beta^\prime}(L,F) = \dim H^1\big(Par Hom_{\beta^\prime} (L,F)\big) \\ & = - \chi\big(Par Hom_{\beta^\prime} (L,F)\big) = - \chi \big(Hom(L,F)\big) - \sum_{i=1}^n (\dim P_{x_i} -1) \\ & = - \chi (L^* \otimes F) + \vert S_L \vert 
 = 2d_0 - d - 1 + n - \vert S \vert,
\end{align*}
where we used the Riemann-Roch theorem and the fact that $S_L=S_F^c=S^c$. 
Hence, every vanishing wall is given by $H_{(d_0,S)}$ with $2d_0 - d - 1 + n - \vert S\vert = 0$. Conversely, if $d+1 -n + \vert S\vert$ is even and $d_0=( d + 1-n + \vert S\vert)/2$, then $H_{(d_0,S)}$ is a vanishing wall. We conclude that if $\beta^\prime$ is in a null chamber separated from the rest by a (unique) hyperplane $H_{(d_0,S)}$ then $2d_0 - d > \varepsilon_S(\alpha^\prime)$ with $\alpha^\prime=\beta^\prime_2 -\beta^\prime_1$, as usual, and $2d_0 - d - 1 + n - \vert S\vert = 0$ and so, when $n\geq 3$, $(d_0,S)$ originates a critical component with index $0$ (since this pair satisfies \eqref{eq:cond}).
\end{proof}
\begin{example}\label{ex2}
Let us now consider the case where  $g=0$ (i.e. $\Sigma=\mathbb{C} \P^1$) and  $\deg(E)=0$, and make the additional restriction of only considering rank-$2$ PHBs which are trivial as holomorphic vector bundles. 
Let $\mathcal H (\beta) \subset \mathcal N^{0,\Lambda}_{\beta,2,0}$ be the moduli space of such PHBs. The $S^1$-action on $ \mathcal N^{0,\Lambda}_{\beta,2,0}$ defined in \eqref{eq:action} restricts to an $S^1$-action 
on  $\mathcal H (\beta)$ with moment map the restriction to $\mathcal H (\beta)$ of the moment map $f$ defined in \eqref{eq:f}.
For a generic weight vector $\beta$ (with $0<\beta_1(x_j)<\beta_2(x_j)<1$ at the parabolic points $x_j \in D=\{x_1, \ldots, x_n\}$), the critical components  of $f=\frac{1}{2} \vert \vert \Phi \vert \vert^2$ where $f$ is nonzero are those $\mathcal{M}_{(0,S)} \subset \mathcal H (\beta)$ for which
$$
\varepsilon_S(\beta_2 -\beta_1) < 0 \leq \vert S \vert - 2.
$$
Indeed, by Proposition~\ref{prop:Simpson}, an element of $\mathcal{M}_{(0,S)}$ decomposes as $E=E_0\oplus E_1$, and so $d_0 = \deg(E_0)=0$. 

Hence, there is a one-to-one correspondence between the components $\mathcal{M}_{(0,S)}$ and the sets $S\subset \{1, \ldots, n\}$ with $\vert S \vert \geq 2$ which are short for $\alpha\in \mathbb{R}_+^n$, with $\alpha_i:=\beta_2(x_i) -\beta_1(x_i)$ (see \eqref{eq:short} for the definition of a short set). 

The Morse indices of the critical submanifolds $\mathcal{M}_{(0,S)}$ are 
$$
\lambda_{(0,S)}= 2(n-1-\vert S \vert).
$$
If one of these has index zero then the corresponding short set $S$ has cardinality $\vert S \vert =n-1$. As we will see later, the space $\mathcal{M}_{\beta,2,0}^{0,\Lambda}$ of ordinary rank-$2$ parabolic bundles of degree zero and fixed determinant can be identified with the set of spatial polygons in  $\mathbb{R}^3$ with $n$ edges of prescribed lengths equal to $\alpha_i$. Then, the existence of a short set with  cardinality $n-1$ implies that these polygons do not close and so $\mathcal{M}_{\beta,2,0}^{0,\Lambda} =\varnothing$ (thus verifying Proposition~\ref{prop:vanishing}).

To end this example we explore in detail the implications of the genericity condition on the weight vector $\beta$. Let $E$ be any rank-$2$ semistable parabolic bundle over $\mathbb{C} \P^1$ which is trivial as a holomorphic vector bundle. By Grothendieck's Theorem the underlying holomorphic bundle is isomorphic to the sum
$$
\mathcal{O}_{\mathbb{C} \P^1}(0) \oplus \mathcal{O}_{\mathbb{C} \P^1}(0).
$$
Hence, given an arbitrary $i\in\{1, \ldots,n\}$ there is a uniquely determined parabolic degree-$0$ line subbundle $L$ of $E$ with fiber over $x_i$ equal to $L_{x_i}=E_{x_i,2}$ (the underlying line bundle is just $\mathbb{C}\P^1 \times E_{x_i,2}$). Any other parabolic line subbundle $\widetilde{L}$ of $E$ admits a nontrivial parabolic map to $E/L$ . In particular its degree is also zero (cf. Lemma 2.4 in \cite{B}). We conclude that any parabolic line subbundle of $E$ must have degree zero and hence it is trivial as a holomorphic line bundle.

Knowing this, any rank-$2$ holomorphically trivial PHB which is semistable but not stable with respect to the weights $\beta$ must have an invariant line subbundle ${\bf L}$ satisfying 
\begin{equation}
\label{eq:gen1}
0 = \sum_{i \in S_L} \big(\beta_2(x_i)-\beta_1(x_i)\big) - \sum_{i \in S_L^c}  \big(\beta_2(x_i)-\beta_1(x_i)\big)
\end{equation}
(just use \eqref{eq:stable} with both $\deg(E)=\deg(L)=0$). For any $S\subset \{1, \ldots, n \}$ one can construct a parabolic line bundle which is trivial as a holomorphic line bundle and has parabolic weights 
$$
\beta^L (x_i) = \left\{ \begin{array}{l} \beta_2(x_i), \,\text{if}\,\, i \in S \\ \\ \beta_1(x_i),\, \text{if} \,\, i \notin S. \end{array} \right.
$$
Hence one may write $L= \mathbb{C} \P^1 \times \mathbb{C}$ and see it as a line subbundle ${\bf L}$ of the PHB
$$
{\bf E}=\left(E:= \mathbb{C} \P^1 \times \mathbb{C}^2, (\beta_j(x_i))_{x_i \in D}, \Phi=0\right)
$$  
with the flag structure defined by
\begin{align*}
\mathbb{C}^2 & = E_{x_i,1} \supset  E_{x_i,2}= \mathbb{C}  \supset 0, \\
0 & \leq \beta_1(x_i) <  \beta_2(x_i) < 1,
\end{align*}
where the class $[E_{x_i,2}] \in \mathbb{C} \P^1$  is the same for all $i \in S$ and satisfies 
$$
[E_{x_i,2}] = [L_{x_i}],\,\,\text{for} \,\, i \in S.
$$
(Note that $\mathbb{C} \P^1$ is the projective space of the fiber of $E$.)
Then ${\bf E}$ and ${\bf L}$ satisfy \eqref{eq:gen1} if and only if
$$
\sum_{i \in S}  \big(\beta_2(x_i)-\beta_1(x_i)\big) - \sum_{i \in S^c}  \big(\beta_2(x_i)-\beta_1(x_i)\big)= 0.
$$
We conclude that  a  weight vector $\beta$ is generic if and only if
$$
\varepsilon_S(\alpha):= \sum_{i \in S} \alpha_i - \sum_{i \in S^c} \alpha_i \neq 0
$$
for every $S \subset \{1, \ldots, n\}$, where $\alpha:= \beta_2-\beta_1$. Note that  this condition is  the same as the one used for  polygon and hyperpolygon spaces in Section~\ref{sec:pol}.
\end{example}

\begin{example}\label{ex:n4}
Let us consider the moduli space $\mathcal{N}_{\beta,2,0}^{0, \Lambda}$ of PHBs over $\C\P^1$ with $n=4$ parabolic points. By Proposition~\ref{prop:cond}, given $S\subset\{1,2,3,4\}$ and $d_0 \in \Z$, the critical submanifold $\mathcal{M}_{(d_0,S)}$ is nonempty if and only if 
$$
\varepsilon_S(\alpha) < 2d_0 \leq \lvert S \rvert -2,
$$
with $\alpha=\beta_2-\beta_1$. If $d_0=0$ then one obtains all short sets of cardinality at least $2$. One can easily check that the only other possible value is $d_0=1$, in which case one obtains $S=\{1,2,3,4\}$. 
There are then four index-$2$ critical points: $\mathcal{M}_{(1,\{1,2,3,4\})}$ and  $\mathcal{M}_{(0,S_i)}$, $i=1,2,3$, corresponding to the three possible short sets $S_i$ of cardinality $2$ (cf. Proposition~\ref{prop:index}). Moreover, from Proposition~\ref{prop:vanishing} we know that $\mathcal{M}_{\beta,2,0}^{0,\Lambda}$ is empty if and only if there is a short set $S_0$ of cardinality $3$, in which case  the critical component  $\mathcal{M}_{(0,S_0)}\cong \C\P^1$ has index-$0$. When nonempty $\mathcal{M}_{\beta,2,0}^{0,\Lambda}\cong \C\P^1$ is the critical component of index-$0$. 

Note that if we restrict the circle action to the moduli space $\mathcal{H}(\beta)$ as in Example~\ref{ex2}  we are left with the three index-$2$ critical points $\mathcal{M}_{(0,S_i)}$, $i=1,2,3$, corresponding to the three possible short sets $S_i$ of cardinality $2$, together with the minimal sphere (either $\mathcal{M}_{\beta,2,0}^{0,\Lambda}$ or $\mathcal{M}_{(0,S_0)}$ with $S_0$ the short set of cardinality $3$).
\end{example}

\section{Trivial rank-$2$ parabolic Higgs bundles over $\C\P^1$ versus hyperpolygons}
\label{sec:isomorphism}

In this section we give an explicit isomorphism between hyperpolygons spaces and  moduli spaces of parabolic Higgs bundles. 

Given a divisor $D=\{x_1, \ldots, x_n\}$  in $\C\P^1$, let $\mathcal H (\beta)$ be the subspace of $\mathcal N^{0,\Lambda}_{\beta,2,0}$ formed by 
rank-$2$ $\beta$-stable PHBs $E$ over $\C\P^1$ that are topologically trivial, (see Example \ref{ex2}) with generic parabolic weights $\beta_2(x_i), \beta_1(x_i)$.
The fact that the parabolic weights are generic implies that the vector $\alpha:= \beta_2 - \beta_1 \in \R_+^n$ is also generic  (see \eqref{eq:epsilon}), 
and hence we can consider the hyperpolygon space $X(\alpha)$.
Then we have the following result.

\begin{Theorem} \label{isomorphism}
The hyperpolygon space $X(\alpha)$ and the moduli space $\mathcal H(\beta)$ of PHBs  are isomorphic. 
\end{Theorem}

\begin{proof}
Consider the map 
\begin{equation}
\label{eq:isom}
\begin{array}{rl}
\mathcal{I}:X(\alpha) \rightarrow & \mathcal H(\beta)\\ \\
{[p,q]}_{\alpha \textnormal{-st}} \mapsto & [E_{(p,q)}, {\Phi}_{(p,q)} ] =: \mathbf E_{(p,q)} \\
\end{array}
\end{equation}
where $E_{(p,q)}$ is the trivial vector bundle $\C\P^1\times \C^2$ with the parabolic structure consisting of weighted flags 
\begin{align*}\label{flag}
 \C^2 & \supset \langle q_i  \rangle \supset 0\\ 
0 \leq \beta_{1}(x_i) & < \beta_{2}(x_i) < 1
\end{align*}
over each $x_i \in D$, and where $\Phi_{[p,q]} \in H^0\big(S Par End(E_{(p,q)}) \otimes K_{\C\P^1}(D)\big)$ is the Higgs field uniquely determined by setting the residues at the parabolic points $x_i$ equal to
\begin{equation}
\label{eq:res}
\text{Res}_{x_i} \Phi := (q_i p_i)_0.
\end{equation}

We first show that the map $\mathcal I$ is well-defined, that is, the Higgs field $\Phi_{(p,q)}$ is uniquely defined,  the PHB $\mathbf E_{(p,q)}$ is stable, and the map $\mathcal I$ is independent of the choice of representative in ${[p,q]}_{\alpha \textnormal{-st}}$.

$\bullet$ Given a prescribed set of residues adding up to zero, Theorem II.5.3 in \cite{Rsurfaces} allows one  to construct a meromorphic 1-form (since  $\C\P^1$ is compact). This defines $\Phi$ up to addition of a holomorphic 1-form. However, by Hodge theory, the space of holomorphic 1-forms on a Riemann surface of genus $g$ has dimension $g$ (see Proposition III.2.7 in \cite{Rsurfaces}), and so on $\C\P^1$ a collection of residues adding up to zero uniquely determines a meromorphic 1-form.
Since  $(p,q) \in \mu_{\C}^{-1} (0)^{\alpha \textnormal{-st}}$, the set of residues  \eqref{eq:res}
adds up to $0$ by the complex moment map condition \eqref{complex} and so it uniquely determines the Higgs field  $\Phi_{(p,q)} \in H^0\big(S Par End(\C\P^1 \times \C^2) \otimes K_{\C\P^1}(D)\big).$

$\bullet$ Recall that the PHB $\mathbf E_{(p,q)} $ is stable if $\mu (L) < \mu (E_{(p,q)}) $ for all proper parabolic subbundles $L$ that are preserved by $ \Phi_{(p,q)} .$
Note that, since the bundle $E_{(p,q)}$ is topologically trivial, any parabolic Higgs subbundle ${\bf L}$ of ${\bf E}_{(p,q)}$ is also trivial (see Example \ref{ex2}) and its parabolic structure at each point $x_i \in D$ consists of the fiber $L_{x_i}$ with weight 
$$\beta^{L}(x_i) = \left\{ \begin{array}{ll} \beta_2 (x_i),&  \text{if} \, L_{x_i} =\langle q_i \rangle \\ \\ \beta_1(x_i), & \text{otherwise}. \end{array}\right.
$$ 
Consider the index set $S_L := \{ i \in \{1, \ldots, n \} \mid L_{x_i} = \textnormal{Im } q_i\}$ associated to any such subbundle. Since $L$ is topologically trivial, then $S_L$ is clearly straight. Let us assume without loss of generality that  the fiber of $L$ at each point of $\C\P^1$ is the space generated by $(1,0)^t$. Then, writing $q_i = (c_i,d_i)^t$, one has $d_i=0$ for $i\in S_L$ and $d_i \neq 0 $ for $i \in S_L^c.$
Since $\Phi_{(p,q)} $ preserves ${\bf L}$, then, writing $p_i=(a_i,b_i)$ the residues $(q_ip_i)_0$ satisfy 
\begin{displaymath}
(q_ip_i)_0 \left(\begin{array}{c}
1\\ \\
0
\end{array}\right) =
\left(\begin{array}{cc}
\frac{1}{2}(a_i c_i -b_i d_i) & b_i c_i\\ \\
a_i d_i & - \frac{1}{2}(a_i c_i -b_i d_i)\\
\end{array}\right)
\left(\begin{array}{c}
1\\ \\
0
\end{array}\right)
= 
\left(\begin{array}{c}
\lambda_i \\ \\
0
\end{array}\right)
\end{displaymath}
for some $\lambda_i \in \C.$
This implies that $a_id_i=0$ for every $i$ and so $a_i=0$ for every $i \in S_L^c$. 
Then, using the moment map condition  \eqref{complex1}, one has $b_i=0$  and thus $p_i=0$ for $i \in S_L^c$. 

Consequently, by the $\alpha$-stability of $(p,q)$ (see Theorem \ref{alpha-stability}) the index set $S_L$ is short. This, by \eqref{eq:stable} with $deg L = deg E = 0,$ is equivalent to $\mu(L) <\mu (E), $ and the stability of ${\bf E}_{(p,q)}$ follows.

$\bullet$ To see that $\mathcal I$ is independent of the choice of a representative in ${[p,q]}_{\alpha \textnormal{-st}}$ let $(\tilde p , \tilde q )$ be an element in the $K^\C$-orbit of $(p,q)$ and consider $[E_{(\tilde p, \tilde q)}, \Phi_{(\tilde p, \tilde q)} ]$ as before. The Higgs field $\Phi_{(\tilde p, \tilde q)}$ is defined by the residues
\begin{align*}
\text{Res}_{x_i}  \Phi_{(\tilde p_i, \tilde q_i)} & := (\tilde{q}_i \tilde{p}_i)_0 = \big(B q_i z_i^{-1} z_i p_i B^{-1}\big)_0 =
B (q_i p_i)_0 B^{-1} \\ & = B \, \text{Res}_{x_i} \Phi_{(p,q)} B^{-1}
\end{align*}
for some $B \in SL(2, \C)$ and $z_i \in \C^*$.
Similarly, the flags in $E_{(\tilde p, \tilde q)}$ are determined by $\tilde q_i = B q_i z_i^{-1}$. 
Note that $q_i z_i^{-1}$ is just another generator of $\langle q_i \rangle$, and $B$ acts on the whole bundle leaving the flag structure unchanged. Since the weights are obviously the same, we can conclude that $[E_{(p,q)}, \Phi_{(p,q)}] = [E_{(\tilde p, \tilde q)}, \Phi_{(\tilde p, \tilde q)} ].$
This completes the proof that the map $\mathcal{I}$ is well-defined.

Let us consider the map
$\mathcal{F}: \mathcal{H}(\beta) \to X(\alpha)$ defined by 
\begin{equation}\label{eq:mathcalf}
{\mathcal F} ([E, \Phi])= [p,q]_{\alpha \text{-st}}
\end{equation}
where $(p,q)$ is determined as follows. For every parabolic point $x_i \in D, $ let $q_i = (c_i, d_i)^t$ be a generator of the flag 
$E_{x_i,2}$ and, considering the residue of the Higgs field $\Phi$ at the parabolic point $x_i$ 
\begin{equation*}
 N_i:= \text{Res}_{x_i} \Phi = 
\left( \begin{array}{cc}
r_{11}^i & r_{12}^i\\ \\
r_{21}^i & r_{22}^i
\end{array} \right),
\end{equation*}
let  $p_i $ be 
\begin{equation}
\label{eq:pi}
 p_i = (a_i, b_i) := \left\{ \begin{array}{ll}
 \big( \frac{r_{21}^i}{d_i}, \frac{r_{12}^i}{c_i} \big), & \textnormal{if } c_i, d_i \neq 0; \\
  & \\
 \big( \frac{r_{21}^i}{d_i}, 0 \big), & \textnormal{if } c_i=0, d_i \neq 0; \\
& \\
 \big(0,\frac{r_{12}^i}{c_i} \big), & \textnormal{if } c_i \neq 0,  d_i = 0. \\
\end{array} \right.
\end{equation}
(Note that the case $c_i=d_i=0$ never occurs since the flags are complete.) To see that $\mathcal{F}$ is well-defined one needs 
to check that $(p,q)$, defined as above, is in $\mu_{\C}(p,q)=0$,  it is $\alpha$-stable and also that the value of $\mathcal{F}$ does not 
depend on the choice of generators of the flags $E_{x_i,2}$ nor on the choice of representative of the class $[E,\Phi]$.

$\bullet$ Since $N_i$ is by assumption trace-free, one gets $r_{22}^i= - r_{11}^i$.
Moreover, since $N_i$ preserves the flag, one has that $c_i=0$ implies $r_{12}^i=0$ and that $d_i=0$ implies $r_{21}^i=0$. Hence, in all cases one has $r_{12}^i=b_ic_i$ and $r_{21}^i=a_id_i$ and then
\begin{equation}
\label{eq:flag}
\left( \begin{array}{cc}
r_{11}^i & b_i c_i\\ \\
a_i d_i & - r_{11}^i\\
\end{array} \right)
\left( \begin{array}{c}
c_i\\ \\
d_i\\
\end{array} \right)= \left( \begin{array}{c}
\lambda c_i\\ \\
\lambda d_i
\end{array} \right)
\end{equation}
for some $\lambda \in \C.$  On the other hand,  since the residue of the Higgs field  is nilpotent, one has $\det N_i=0$ and so 
\begin{equation}\label{eq:r11 squared}
(r_{11}^i)^2 = - r_{12}^i r_{21}^i = - a_i b_i c_i d_i.
\end{equation}
Using \eqref{eq:flag} and \eqref{eq:r11 squared} one gets that
\begin{equation}\label{eq:r11}
r_{11}^i = \frac{a_i c_i - b_i d_i}{2}
\end{equation}
and so  the residue can be rewritten as
\begin{equation}\label{eq:Ni}
N_i = \left(\begin{array}{cc}
\frac{1}{2}(a_i c_i -b_i d_i) & b_i c_i\\ \\
a_i d_i & - \frac{1}{2}(a_i c_i -b_i d_i)\\
\end{array}\right).
\end{equation}
This, together with the fact that the sum of the residues $N_i$ is  $0$, implies condition \eqref{complex2}.
Moreover, \eqref{eq:r11 squared} and \eqref{eq:r11} give us
$$ \frac{(a_i c_i - b_i d_i)^2}{4}  = - a_i b_i c_i d_i,$$
which implies 
$$a_i c_i + b_i d_i =0.$$
Hence, the nilpotency of the residue $N_i$ implies  the complex moment map condition \eqref{complex1}.
This proves that $(p,q) \in \mu_{\C}^{-1}(0).$ 

$\bullet$ To show that $(p,q)$ is $\alpha$-stable, we need to check that conditions $(i)$ and $(ii)$ of Theorem~\ref{alpha-stability}
are verified. The first one ($q_i \neq 0$ for all $i$), is trivially verified since the flags are complete by assumption. To show  the second condition, let $S \subset \{ 1, \ldots, n\}$ be a maximal straight set such that $p_i=0$ for all $i \in S^c.$
As in Example \ref{ex2} one can construct a line subbundle $L_S$ of the trivial bundle $\C \P^1\times \C^2$ which is trivial as an holomorphic line bundle, with fiber the complex line generated by the $q_i$ for  $i \in S$. We then give $L_S$ a parabolic structure at the parabolic points $x_1, \ldots, x_n$ by assigning the parabolic weights 
$$
\beta^{L_S} (x_i) = \left\{ \begin{array}{l} \beta_2(x_i), \,\text{if}\,\, i \in S \\ \\ \beta_1(x_i),\, \text{if} \,\, i \notin S. \end{array} \right.
$$
By construction $L_S$ is a parabolic subbundle of ${\bf E}$. Moreover, it is also trivially preserved by the Higgs field $\Phi$ since, by the moment map condition \eqref{complex1}, one has $$ N_i q_i =0, \quad \forall \, i=1, \ldots, n.$$
Therefore, by stability of ${\bf E} $, one gets that $L_S$ satisfies $\mu(L_S) < \mu (E)$, which implies that $S$ is short since both bundles have degree zero.  By Remark \ref{maximalstraight}, this is equivalent to condition $(ii)$.

$\bullet$ To show that the value of $\mathcal{F}$ is independent of the choice of generator $q_i$ of the flag $E_{x_i,2}$, let $q_i, \tilde{q}_i$ be two different generators of  $E_{x_i,2}$. Then $\tilde{q}_i=\lambda_i q_i$ for some $\lambda_i \in \C^*$ and so \eqref{eq:pi}  clearly implies that $\tilde{p}_i=\lambda_i^{-1} p_i$ and then $[p,q]_{\alpha-st}=[(\tilde{p},\tilde{q})]_{\alpha-st}$.

$\bullet$ To show that $\mathcal J$ does not depend on the choice of  representative of  the class of $ [E , \Phi]$ one considers another PHB $\widetilde{\mathbf E}= (\widetilde E, \widetilde{ \Phi})$ in $[E , \Phi]$. Let $(\tilde p, \tilde q)$ be coordinates  determined from $\widetilde{\mathbf E}$ by the recipe above and denote by $\widetilde N_i$ the residues of the Higgs field $\widetilde{\Phi}$. Then there exists  $g\in SL(2,\C)$ such that
$\widetilde{E}_{x_i, 2} = g E_{x_i, 2}$
and so one can take
$ \tilde{q}_i= g\, q_i $, where $q_i$ is a generator of $E_{x_i,2}$. Moreover, since the Higgs field $\widetilde{\Phi}$ is obtained from $\Phi$ by conjugation with $g$, the residues  $\widetilde{N}_i $ of $\widetilde{\Phi}$ satisfy
$$ \widetilde{N}_i  = g\, N_i \, g^{-1} \quad \forall i=1, \ldots, n.$$ 
Since $\tilde{p}_i $ is determined by the equation $( \tilde{q}_i \tilde{p}_i)_0 = \widetilde{N}_i,$
one can easily see that 
$$ \tilde{p}_i= p_i g^{-1}$$
and so $(\tilde p , \tilde q)$ is in the $K^{\C}$-orbit of $(p,q)$.

Finally, from what was shown above it is clear that
$$ \mathcal{F}  = \mathcal I^{-1}.$$

\end{proof}
This isomorphism allows us to identify $X(\alpha)$ and $\mathcal{H}(\beta)$ as $S^1$-spaces.

\begin{proposition}
The isomorphism $\mathcal I$ is $S^1$-equivariant with respect to the $S^1$-actions on $X(\alpha)$ and on $\mathcal H (\beta)$  
defined in \eqref{action} and in \eqref{eq:action} respectively.
\end{proposition}
\begin{proof}
The bundles $e^{{\bf i} \theta} \cdot \mathcal I ([p,q])$ and $\mathcal I (e^{{\bf i} \theta} \cdot [p,q])$ are both topologically trivial and have the same parabolic structure. Moreover, the Higgs field $\Phi_{(e^{{\bf i} \theta}p,q)}$ on $\mathcal I (e^{{\bf i} \theta} \cdot [p,q])$ is uniquely 
determined by the residues 
$$
Res_{x_i} \Phi_{(e^{{\bf i} \theta}p,q)} = (e^{{\bf i} \theta} q_i p_i)_0 = e^{{\bf i} \theta} Res_{x_i} \Phi_{(p,q)}
$$
and hence 
$$
\Phi_{(e^{{\bf i} \theta}p,q)} = e^{{\bf i} \theta} \Phi_{(p,q)}.
$$
Therefore, as PHBs, 
$$
e^{{\bf i} \theta} \cdot \mathcal I ([p,q])= \mathcal I (e^{{\bf i} \theta} \cdot [p,q])
$$
and the isomorphism $\mathcal I $ is $S^1$-equivariant.
\end{proof}
Since the isomorphism $\mathcal I: X(\alpha) \to  \mathcal H (\beta)$ is $S^1$-equivariant it maps the critical components of the moment map $\phi$ on $X(\alpha)$  to the critical components of the moment map $f$ on $\mathcal H (\beta)$ as well as the corresponding flow downs. This flow down is the restriction to $\mathcal H (\beta)$ of the  \emph{nilpotent cone}  of $\mathcal{N}_{\beta,2,0}^{0,\Lambda}$, following \cite{n1}-Section 5 and \cite{GGM}-Section 3.5. 

In particular,  the moduli space of polygons $M(\alpha)$ is mapped to the moduli space $\mathcal{M}^{0,\Lambda}_{\beta,2,0}$ of rank-$2$, holomorphically trivial, fixed determinant parabolic bundles over $\C \P^1$. The fact that these two spaces are isomorphic has already been noted in \cite{AW} for small values of $\beta$. 

Moreover, the  critical components $X_S$ in $X(\alpha)$ are mapped to the critical components $\mathcal{M}_{(0,S)}$ in $\mathcal H (\beta)$ and each connected component of the core $U_S$ is isomorphic through $\mathcal I $ to the component $\mathcal{U}_{(0,S)}:=\mathcal{I}(U_S)$ of the nilpotent cone defined as the closure inside $\mathcal{H}(\beta)$ of the set
\begin{equation}\label{eq:nilcone}
\left\{ [E, \Phi] \in  \mathcal H (\beta) \mid \lim_{t \to \infty}    [E, t\cdot \Phi] \in  \mathcal M_{(0,S)} \right\}.
\end{equation}
The nilpotent cone $\mathcal L_{\beta}$ of $ \mathcal H (\beta)$ is then
\begin{displaymath}
\mathcal L_{\beta} := \mathcal M^0_{\beta,2,0} \cup  \bigcup_{S \in \mathcal S'(\alpha)} \mathcal U_{(0,S)},
\end{displaymath}
and so $\mathcal L_{\beta} = \mathcal I (\mathfrak L_{\alpha})$.

\begin{example}
Consider the case of $4$ parabolic points as in Example~\ref{ex:n4}. The closure of the flow down of the index-$2$ critical points $\mathcal{M}_{(1,\{1,2,3,4\})}$ and  $\mathcal{M}_{(0,S_i)}$, $i=1,2,3$ is a union of four spheres intersecting the minimal component at four distinct points. Consequently, the nilpotent cone of $\mathcal{N}_{\beta,2,0}^{0,\Lambda}$ is a union of five spheres arranged in a $\widetilde{D}_4$ configuration \cite{Fri} as in Figure~\ref{fig:cone4}. Restricting this  nilpotent cone to $\mathcal{H}(\beta)$ we loose the critical point $\mathcal{M}_{(1,\{1,2,3,4\})}$ and the corresponding flow down. Hence, the nilpotent cone of   $\mathcal{H}(\beta)$ is a union of  four spheres arranged in a $D_4$ configuration just like the core of the associated hyperpolygon space $X(\alpha)$ (cf. Example~\ref{ex:h4}).
\begin{figure}[htbp]
\psfrag{1}{\footnotesize{$\mathcal{M}_{\big(1,\{1,2,3,4\}\big)}$}}
\psfrag{2}{\footnotesize{$\mathcal{M}_{(0,S_1)}$}}
\psfrag{3}{\footnotesize{$\mathcal{M}_{(0,S_2)}$}}
\psfrag{4}{\footnotesize{$\mathcal{M}_{(0,S_3)}$}}
\psfrag{M}{\footnotesize{$\mathcal{M}_{\beta,2,0}^{0,\Lambda}$ or $\mathcal{M}_{(0,S_0)}$
with $\lvert S_0 \rvert =3$}}
\includegraphics[scale=.5]{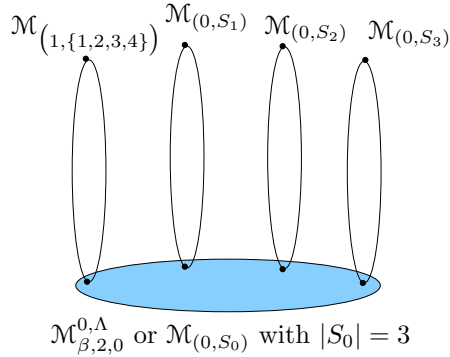}
\caption{Nilpotent cone of $\mathcal{N}_{\beta,2,0}^{0,\Lambda}$ when $n=4$: union of five spheres arranged in a $\widetilde{D}_4$
configuration.}
\label{fig:cone4}
\end{figure}

\end{example}
\section{Wall crossing}\label{wall crossing}

The variation of moduli of PHBs has been studied in detail by Thaddeus in  \cite{T}. 
The construction in this Section is an adaptation of his work to the moduli space $\mathcal{H}(\beta)$ of rank-$2$, topologically trivial PHBs over $\C \P^1$ with  fixed determinant and trace-free Higgs field considered in the previous section.
As we have seen in Example~\ref{ex2}, rank-$2$ PHBs over $\C \P^1$ which are trivial as holomorphic bundles are semistable but not stable with respect to the parabolic weights $\beta_{1}(x_i), \beta_{2}(x_i)$ if and only if 
$$ \varepsilon_S (\alpha)=0$$
for some set $S \subset \{ 1, \ldots, n \}$, with $\alpha=\beta_2 - \beta_1$. 
Hence, any such PHB must have an invariant line subbundle ${\bf L}$ which is trivial as a holomorphic line bundle and satisfies  \begin{equation}\label{eq:1}
 0= \sum_{i \in S_L} \alpha_i - \sum_{i \in S_L^c} \alpha_i 
\end{equation}
for $S_L = \{i \in \{1, \ldots, n \} \mid  \beta^L(x_i)= \beta_2 (x_i) \}$, where $\beta^L (x_i)$ is the parabolic weight of $L$ at $x_i$.
We will call the set $S_L$ the \emph{discrete data} associated to a line subbundle ${\bf L}$ of a strictly semistable PHB satisfying equation \eqref{eq:1}.

Let $Q$ be the weight space of all possible values of $(\beta_1 (x_j), \beta_2(x_j))$. It can be seen as the product 
$$ Q = \mathcal S_2^n\subset (\R_+)^{2n}$$
of $n$ open simplices of dimension $2$ determined by 
$$ 0 \leq \beta_1 (x_j ) < \beta_2 (x_j) <1.$$

If the discrete data of a line subbundle is fixed, then \eqref{eq:1} requires that the point $\beta\in Q$ belongs to the intersection of an affine hyperplane with $Q$. We will call such an intersection a \emph{wall}. There is therefore a finite number of walls. Note that a set $S\subset \{1, \ldots, n\}$ and its complement give rise to the same wall and that on the complement of these walls the stability condition is equivalent to semistability.  A connected component  of this complement will be called a \emph{chamber}. In this section we study how the moduli spaces $\mathcal H(\beta)$ change when a wall is crossed.

Let us then choose a point in $Q$ lying on only one wall $W$. A small neighborhood of this point intersects exactly two chambers, say $\Delta^+$ and  $\Delta^-$ and a 
PHB is $\Delta^+$-stable (respectively  $\Delta^-$-stable) if it is stable with respect to the weights $\beta \in \Delta^+$ (respectively $\Delta^-$). 
If a PHB $\bf E$ is $\Delta^-$-stable but $\Delta^+$-unstable then it has a PH line subbundle $\bf L$ (called a \emph{destabilizing subbundle}) for which the stabilizing condition holds in  $\Delta^-$ but fails in $\Delta^+$.

Let $\mathcal H^+$ and $\mathcal H^-$ respectively denote the moduli space of $\Delta^+$ and $\Delta^-$-stable rank-$2$, fixed-determinant PHBs which are trivial as holomorphic bundles. Choosing the wall $W$ is equivalent to choosing a set $S \subset \{1, \ldots, n\}$ for which \eqref{eq:1} holds whenever $\beta \in W$. The only ambiguity is the possibility of exchanging $S$ with $S^c$. Interchanging these sets if necessary one can assume without loss of generality that $\varepsilon_S (\alpha) >0$ whenever $\beta \in \Delta^+$ with $\alpha=\beta_2-\beta_1$. The following propositions then hold.
\begin{proposition}
 If $\bf E$ is $\Delta^-$-stable but $\Delta^+$-unstable then any destabilizing subbundle has discrete data $S$.
\end{proposition}
\begin{proof}
As the weight $\beta$ crosses from $\Delta^+$ to $\Delta^-$ any destabilizing subbundle ${\bf L}^+$ of $\bf E$ stops destabilizing. Hence the corresponding values of $\varepsilon_{S_L}(\alpha)$ change from positive to negative. This implies that ${\bf L}^+$ has discrete data $S_L=S$.
\end{proof}
\begin{proposition}
\label{prop:unique}
 If $\bf E$ is $\Delta^-$-stable but $\Delta^+$-unstable then the destabilizing subbundle ${\bf L}^+$ is unique. 
\end{proposition}
\begin{proof}
Let ${\bf L}^-$ be the quotient of $\bf E$ by a destabilizing subbundle ${\bf L}^+$ (topologically trivial as well). If $\bf F$ is another $\Delta^+$-destabilizing trivial line subbundle, then it must have discrete data $S$. There is then a non-trivial homomorphism $\bf F \to {\bf L}^-$ of PHBs and hence a nontrivial element of $\mathbb H^0(C^\bullet (\bf F, \bf L^-))$ (both $\bf F$ and ${\bf L}^-$ are trivially $\Delta^+$ and $\Delta^-$-stable). By Proposition~\ref{prop:GP}, this is impossible since the two PHBs are not isomorphic. Indeed,
$$
\textnormal{pdeg} \, {\bf F} = \sum_{i \in S} \beta_2(x_i) + \sum_{i \in S^c} \beta_1(x_i), 
$$
while
$$
\textnormal{pdeg} \, {\bf L}^-  = \sum_{i \in S} \beta_1(x_i) + \sum_{i \in S^c} \beta_2(x_i),
$$
and so 
\begin{align*}
\beta \in \Delta^+ & \Leftrightarrow \varepsilon_{S}(\alpha) >0 
 \Leftrightarrow \sum_{i \in S} (\beta_2(x_i) -  \beta_1(x_i)) >\sum_{i \in S^c} (\beta_2(x_i) -  \beta_1(x_i)) \\ &  \Leftrightarrow \text{pdeg} \, \bf F > \text{pdeg} \, {\bf L}^- .
\end{align*}
\end{proof}
\begin{proposition}
Let ${\bf L}^+$ and ${\bf L}^-$ be two line PHBs which are trivial as holomorphic line bundles with discrete data $S$ and $S^c$. Then any extension of ${\bf L}^-$  by ${\bf L}^+$ is $\Delta^+$-unstable and it is $\Delta^-$-stable if and only if it is not split.
\end{proposition}
\begin{proof}
The bundle ${\bf L}^+$ would be the destabilizing subbundle of such an extension $\bf E$ so this extension would be $\Delta^+$-unstable. Moreover, if $\bf E$ splits as ${\bf L}^+ \oplus {\bf L}^-$ then ${\bf L}^-$ is the $\Delta^-$-destabilizing subbundle of $\bf E$ which would then be $\Delta^-$-unstable.

Conversely, if the extension $\bf E$ is $\Delta^-$-unstable, the $\Delta^-$-destabilizing bundle $\bf F$ must not be $\Delta^+$-destabilizing and so it has discrete data $S^c$. The composition map 
$$
\bf F \hookrightarrow \bf E \to {\bf L}^-
$$
must then be a nontrivial homomorphism of PHBs since $\bf F$ and ${\bf L}^-$ have the same incidences with the flags ($\bf F$ and ${\bf L}^-$ both have discrete data $S^c$). Hence there is an element of $\mathbb H^0(C^\bullet (\bf F, {\bf L}^-))$ which, by Proposition~\ref{prop:GP}, must be an isomorphism and so $\bf E$ splits.
\end{proof}
The above three propositions then give the following result.
\begin{Theorem}
 If $\bf E$ is $\Delta^-$-stable but $\Delta^+$-unstable then it can be expressed uniquely as a nonsplit extension of PHBs
$$
0 \to {\bf L}^+ \to {\bf E} \to {\bf L}^- \to 0,
$$
where ${\bf L}^\pm$ are parabolic Higgs line bundles with discrete data $S$ and $S^c$. 
Conversely, any such extension is $\Delta^-$-stable but $\Delta^+$-unstable. 
\end{Theorem}

We will use this Theorem to see that $\mathcal H^+$ and $\mathcal H^-$ have a common blow up with the same exceptional divisor. The loci in $\mathcal H^\pm$ which are blown up (flip loci) are isomorphic to projective bundles $\P U^\pm \cong \C\P^{n-3}$ over a product $\mathcal N^+ \times \mathcal N^-$ (a $0$-dimensional manifold) of moduli spaces of parabolic Higgs line bundles which are trivial as holomorphic line bundles. Moreover, as we will see, the bundles $U^+$ and $U^-$ are dual to each other and so $\P U^+ $ and $\P U^-$ are projective bundles of the same rank over the same basis.

Let then $\mathcal N^+$ and $\mathcal N^-$ be the moduli spaces of parabolic line Higgs bundles over $\C\P^1$ which are trivial as holomorphic line bundles and have discrete data $S$ and $S^c$ respectively. By \cite{BY} the dimension of these spaces is
$$
\dim\mathcal{N}^- = \dim \mathcal{N}^+ = 2(g-1)(r^2-1) + (r^2 -r)= 0.
$$
Moreover, $\mathcal{N}^+$ and $\mathcal{N}^-$ are composed of just one point as any two parabolic line Higgs bundles which are trivial as holomorphic line bundles and have discrete data $S$ (or $S^c$) are isomorphic (there is always a parabolic map between them). Hence the product $\mathcal{N}^+\times \mathcal{N}^-=\{\text{pt}\}$.

Define ${\bf L}^\pm$ to be the element in $\mathcal N^\pm$. 
Considering  the complex $C^\bullet({\bf L}^-, {\bf L}^+)$ and taking the hypercohomology 
$$
\mathbb{H}^*\big(C^\bullet({\bf L}^-, {\bf L}^+)\big)
$$
one defines
$$
U^-:= \mathbb H^1 \big(C^\bullet({\bf L}^-,{\bf L}^+)\big)=(R^1)_*\big(C^\bullet({\bf L}^-, {\bf L}^+)\big)
$$
and then, from the long exact sequence presented in Proposition~\ref{prop:GP}, one obtains
\begin{align}\label{seq:exact}
0 & \to \mathbb{H}^0\big(C^\bullet({\bf L}^-, {\bf L}^+)\big) \to H^0\big(Par Hom ( L^-, L^+)\big) \to\\ 
\nonumber
&  \to  H^0\big(S Par Hom ( L^-, L^+) \otimes K_{\C \P^1}(D)\big) \to    U^- \to H^1\big(Par Hom (L^-, L^+)\big) \to \\ 
\nonumber
& \to H^1\big(S Par Hom (L^-, L^+) \otimes K_{\C \P^1}(D)\big) \to \mathbb{H}^2\big(C^\bullet({\bf L}^-, {\bf L}^+)\big)  \to 0.
\end{align}
Analogously, one can consider the complex $C^\bullet({\bf L}^+, {\bf L}^-)$ and define 
$$
U^+:= \mathbb H^1\big(C^\bullet({\bf L}^+, {\bf L}^-)\big)=(R^1)_*\big(C^\bullet({\bf L}^+, {\bf L}^-)\big)
$$
and obtain a similar sequence.
By Proposition~\ref{prop:GP} and Serre duality for hypercohomology (cf. Proposition~\ref{prop:Serre}) $\mathbb{H}^0$ and $\mathbb{H}^2$ vanish and so $U^+$ and $U^-$ are locally free sheaves (hence vector bundles \cite{A}) dual to each other:
$$
U^-:= \mathbb H^1\big(C^\bullet({\bf L}^-, {\bf L}^+)\big) = \mathbb H^1\big(C^\bullet({\bf L}^+, {\bf L}^-)\big)^* = (U^+)^*. 
$$ 
As stated in Proposition~\ref{prop:GP}(3), $U^-$ parameterizes all extensions of the PHB in $\mathcal{N}^-$ by that in $\mathcal{N}^+$ and so, as usual, the projectivization $\P U^-$ parameterizes all {\bf nonsplit} extensions of the parabolic Higgs line bundle in $\mathcal{N}^-$ by that in $\mathcal{N}^+$ (see for instance \cite{L}). Following the exact sequence \eqref{seq:exact} one can see that the dimension of $U^-$ is given by
$$
\dim U^- = \chi\big(S Par Hom (L^-, L^+) \otimes K_{\C \P^1}(D)\big) - \chi\big(Par Hom (L^-, L^+)\big).
$$
Using \eqref{seq:par} one obtains
$$
\chi\big(Par Hom ( L^-, L^+)\big)  = \chi\big(Hom(L^-, L^+)\big) + \sum_{i=1}^n \big(\dim P_{x_i}(L^-,L^+) - 1\big),
$$
where $P_{x_i}(L^-,L^+)$ denotes  the subspace of $Hom(L^-_{x_i}, L^+_{x_i})$ formed by parabolic maps.
Then, since 
$$S_{L^-}=\{i\in \{1, \ldots, n\}\mid \beta^{L^-}(x_i)= \alpha_2(x_i)\} = S^c$$ 
and
$$
\text{dim}\,P_{x_i}(L^-,L^+)=\left\{ \begin{array}{ll} 1, & \text{if}\,\, i \in S_{L^-}^c \\ \\
0, & \text{otherwise}, \end{array} \right.   
$$
one gets
\begin{align*}
\chi\big( Par Hom (L^-, L^+)\big)= 1- \vert S^c \vert,
\end{align*}
where we used Riemann-Roch to compute   
\begin{align*}
\chi\big(Hom(L^-, L^+)\big) & = \chi\big(Hom(\mathcal O(0),\mathcal O(0))\big)= \chi\big(\C \P^1, \mathcal{O}(0)\big) \\ & = \text{rank} (\mathcal O(0))(1-g)=1.
\end{align*}
On the other hand, consider the short exact sequence
\begin{align*}
0 \to & S Par Hom (L^-, L^+)\otimes K_{\C \P^1}(D)  \to  Hom (L^-, L^+)\otimes K_{\C \P^1}(D) \to \\
&\to Hom (L^-_D, L^+_D)/N_D(L^-, L^+)\to 0,
\end{align*}
where $Hom (L^-_D, L^+_D)= \bigoplus_{x \in D} Hom(L^-_x, L^+_x)$ and, where, denoting by $N_x(L^-, L^+)$ the subspace of $Hom(L^-_x, L^+_x)$ formed by strictly parabolic maps,  $N_D(L^-, L^+)= \bigoplus_{x \in D} N_x(L^-,L^+)$. Then,  
\begin{eqnarray*}
\lefteqn{\chi\big(S Par Hom (L^-, L^+) \otimes K_{\C \P^1}(D)\big)   =} \\ & & \chi\big(Hom (L^-, L^+)\otimes K_{\C \P^1}(D)\big)+ \sum_{x \in D} (\dim N_x - 1)
\end{eqnarray*}
and so, since in this case $P_x(L^-,L^+)=N_{x}(L^-,L^+)$, one obtains
\begin{align*}
\chi\big(S Par Hom (L^-, L^+) \otimes K_{\C \P^1}(D)\big)= \chi\big(K_{\C \P^1}(D)\big) - \vert S^c \vert = n-1 - \vert S^c \vert.
\end{align*}
Here we used the fact that $Hom(L^-, L^+)=\mathcal{O}(0)$, and Riemann-Roch with $\deg(K_{\C \P^1})=-2$ and $\deg(\mathcal O(D))=n$. 
One concludes that
$$
\dim U^- = n-1 - \vert S^c \vert - (1- \vert S^c \vert)= n-2
$$
and so $U^- \cong \C^{n-2}$ and $\P U^- \cong \C \P^{n-3}$.

Every parabolic Higgs bundle given by an element in $ \P U^-$ is $\Delta^-$-stable and so, by the universal property of the moduli space $\mathcal H^-$, 
there exists a morphism
$$
\C \P^{n-3} \cong \P U^- \to \mathcal H^-
$$
whose image is precisely the locus of PHBs which become unstable when the wall is crossed.

Let $V^-$ be the cotangent bundle to $\P U^-$ and consider the corresponding map $\pi^- : \P V
^- \to \P U^-$. On the other hand, consider the Euler sequence of the cotangent bundle (see \cite{Huy}) 
\begin{equation}
\label{eq:Euler}
0 \longrightarrow V^- \stackrel{\pi^+}{\longrightarrow} (U^-)^* \otimes \mathcal
 O_{\P U^-}(-1) \longrightarrow \mathcal O_{\P U^-} \longrightarrow 0.
\end{equation}
More explicitly, using the fact that 
$$ (U^-)^* \otimes \mathcal O_{\P U^-}(-1) = (U^-)^* \times U^- = U^+ \times U^-, $$
one has
$$
\begin{array}{ccccccc}
0 \longrightarrow & T^* \C \P^{n-3} & \stackrel{\pi^+}{\longrightarrow} &  (U^-)^* \times U^- & \longrightarrow & \C \P^{n-3} \times \C & \longrightarrow 0 \\
 & ([\omega], \xi) & \longmapsto & (\omega, \xi) & \longmapsto & ([\omega], \xi(\omega)) & \\
\end{array}
$$
where $[\omega] \in \C \P^{n-3}$ and
$$\xi \in T^*_{[\omega]} \C \P^{n-3} = [\omega]^{\perp} = \{ \xi \in (U^-)^* \mid \xi(\omega) =0 \}.$$
Hence
\begin{align*}
\P V^- = \P (T^* \P U^-) & \subset \P\big((U^-)^* \otimes \mathcal O_{\P U^-}(-1)\big) = \P\big((U^-)^* \times U^-\big) \\
& =\P\big((U^-)^*\big) \times \P(U^-) = \P(U^+) \times \P(U^-),
\end{align*}
the fiber $V^-_{[\omega]}$ over a line $[\omega] \in \P U^-$ is naturally isomorphic to the space of linear functionals $\xi : U^- \to \C$ with $\xi (\omega) =0$, and there is an induced map $\pi^+: \P V^- \to \P U^+$. Moreover, one can identify $\P (T^*_{[\omega]} \P U^-)$ with $\P ( [\omega]^{\perp})$ in a canonical way and for
$[\xi] \in \P \big( [\omega]^{\perp}\big)$ one defines an element $\sigma_{\xi} \in Gr_{n-3} (\C^{n-2})$ with $ [\omega] \subset \sigma_{\xi}$ by 
$$
\sigma_{\xi} = \{ v\in \C^{n-2} \mid \xi(v) = 0 \}.
$$
Then $[\xi] \mapsto ([\omega], \sigma_{\xi})$ gives a diffeomorphism of $\P V^-$ onto the manifold of partial flags in $U^- = \C^{n-2}$ of type $(1, n-3)$ and $\pi^{\pm}: V^- \to \P U^{\pm}$ are the forgetful morphisms that discard one subspace.

As noted before $\P U^-$ parameterizes all nonsplit extensions of the bundle ${\bf L}^-$ in $\mathcal{N^-}$ by the bundle
${\bf L}^+$ in $\mathcal{N^+}$. Over $\P U^- \times \C \P^1$ there is a universal extension 
\begin{equation} 
\label{seq:universal}
0 \to \widetilde{{\bf L}}^+ \otimes \mathcal{O}_{\P U^-}(1) \to \widetilde{{\bf E}} \to  \widetilde{{\bf L}}^- \to 0,
\end{equation}
where, for $([\omega],x)\in \P U^-\times \C \P^1$,   
$$
\widetilde{{\bf L}}^+_{([\omega],x)}={\bf L}_x^+ \quad \text{and} \quad \widetilde{{\bf L}}^-_{([\omega],x)}={\bf L}_x^- 
$$
i.e., if we consider the projection $pr:\P U^-\times \C \P^1 \to \C \P^1$, we have 
$$
\widetilde{{\bf L}}^+ = pr^* {\bf L}^+ \quad \text{and} \quad \widetilde{{\bf L}}^- = pr^* {\bf L}^-. 
$$
Moreover, by the universal property, the extension $\widetilde{{\bf E}}$ restricted to $\{[\omega]\} \times \C\P^1$ is the extension ${\bf E}([\omega])$ of ${\bf L}^-$ by ${\bf L}^+$ determined by the element $[\omega]\in \P U^-$. Extensions like \eqref{seq:universal} are parameterized by
$$
\mathbb{H}^1\Big( \P U^-\times \C \P^1, C^\bullet\big(\widetilde{{\bf L}}^-, \widetilde{{\bf L}}^+ \otimes \mathcal{O}_{\P U^-}(1)\big)\Big)
$$
which by the Kunneth formula is isomorphic to
$$
\mathbb{H}^1\big( \C \P^1,  C^\bullet({\bf L}^-,{\bf L}^+)\big)\otimes \mathbb{H}^0\big( \P U^-,\mathcal{O}_{\P U^-}(1)\big) = U^- \otimes (U^-)^* \cong \text{End}(U^-) 
$$
and one can show that the identity element in $\text{End}(U^-)$ defines the universal extension described above.

Now consider the long exact sequence associated to
\begin{equation}
\label{seq:bprime}
0 \to C^\bullet\big(\widetilde{{\bf L}}^-, \widetilde{{\bf L}}^+ \otimes \mathcal{O}_{\P U^-}(1)\big) \to C^{\bullet \prime}_0(\widetilde{\bf E}) \to \left(C^\bullet(\widetilde{{\bf L}}^+) \oplus C^\bullet(\widetilde{{\bf L}}^-)  \right)_0 \to 0,
\end{equation}
where $C^{\bullet \prime}_0(\widetilde{{\bf E}})$ is the subcomplex of $ C^{\bullet}_0(\widetilde{\bf E})$ associated to the subsheaves $Par End^\prime_0 (\widetilde{E})$ and $SPar End^\prime_0 (\widetilde{E})$ of $Par End_0 (\widetilde{E})$ and $SPar End_0 (\widetilde{E})$ preserving ${\bf L}^+$, and $\left(C^\bullet(\widetilde{{\bf L}}^+) \oplus C^\bullet(\widetilde{{\bf L}}^-)  \right)_0$ is the complex formed by the direct sum of elements of 
$C^\bullet(\widetilde{{\bf L}}^+)$ and $C^\bullet(\widetilde{{\bf L}}^-)$ with symmetric trace. 
By Serre duality and Proposition~\ref{prop:GP} we know that
\begin{align*}
& \mathbb{H}^0\big(C_0^\bullet(\widetilde{{\bf L}}^-)\big) = \mathbb{H}^2\big(C_0^\bullet(\widetilde{{\bf L}}^-)\big) = 0, \\
& \mathbb{H}^0\big(C^\bullet(\widetilde{{\bf L}}^+)\big) = \mathbb{H}^2\big(C^\bullet(\widetilde{{\bf L}}^+)\big) = \C, \\
& \mathbb{H}^0\Big(C^\bullet\big(\widetilde{{\bf L}}^-, \widetilde{{\bf L}}^+ \otimes \mathcal{O}_{\P U^-}(1)\big)\Big) = 
\mathbb{H}^2\Big(C^\bullet\big(\widetilde{{\bf L}}^-, \widetilde{{\bf L}}^+ \otimes \mathcal{O}_{\P U^-}(1)\big)\Big) = 0.
\end{align*}
Moreover, again by Kunneth formula, 
\begin{eqnarray*}
\lefteqn{\text{dim} \, \mathbb{H}^1 \big(\P U^- \times \C\P^1, C^\bullet(\widetilde{{\bf L}}^+)\big) = } \\  & 
\text{dim} \, \Big(\mathbb{H}^1 \big(\C\P^1, C^\bullet({\bf L}^+)\big)\otimes  \mathbb{H}^0 \big(\P U^-,  \mathcal{O}_{\P U^-}(1)\big) \Big)= 0,
\end{eqnarray*}
since $\text{dim} \, \mathbb{H}^1 (\C\P^1, C^\bullet({\bf L}^+))=0$ is the dimension of the moduli space of line PHBs over $\C\P^1$ (cf. \cite{BY}).
Then the long exact sequence associated to
$$
0 \to C^\bullet_0({\bf L}^-) \to \left(C^\bullet(\widetilde{{\bf L}}^+) \oplus C^\bullet(\widetilde{{\bf L}}^-)  \right)_0 \to C^\bullet({\bf L}^+) \to 0
$$
gives 
\begin{align*}
& \mathbb{H}^0\Big(\big(C^\bullet(\widetilde{{\bf L}}^+) \oplus C^\bullet(\widetilde{{\bf L}}^-)  \big)_0 \Big) =\mathbb{H}^2\Big(\big(C^\bullet(\widetilde{{\bf L}}^+) \oplus C^\bullet(\widetilde{{\bf L}}^-)  \big)_0 \Big) = \C \\ & \mathbb{H}^1\Big(\big(C^\bullet(\widetilde{{\bf L}}^+) \oplus C^\bullet(\widetilde{{\bf L}}^-)  \big)_0 \Big)  = 0. 
\end{align*}
Moreover, $\mathbb{H}^0\big(C^{\bullet \prime}_0(\widetilde{{\bf E}}) \big)= 0$
since $\mathbb{H}^0\big(C^{\bullet}_0(\widetilde{{\bf E}}) \big)=0$ and $C^{\bullet \prime}_0(\widetilde{{\bf E}})$ is a subcomplex of $C^{\bullet}_0(\widetilde{{\bf E}})$.
Hence, the long exact sequence associated to \eqref{seq:bprime} gives 
\begin{equation}
\label{seq:prime}
0 \to \C \stackrel{a}{\to} \mathbb{H}^1\Big(C^\bullet\big(\widetilde{{\bf L}}^-, \widetilde{{\bf L}}^+ \otimes \mathcal{O}_{\P U^-}(1)\big)\Big) \stackrel{b}{\to} \mathbb{H}^1\big( C^{\bullet \prime}_0(\widetilde{{\bf E}})\big) \to 0
\end{equation}
and 
\begin{equation}
\label{seq:h2}
0 \to   \mathbb{H}^2\Big(C^\bullet\big(\widetilde{{\bf L}}^-, \widetilde{{\bf L}}^+ \otimes \mathcal{O}_{\P U^-}(1)\big)\Big) \to  \mathbb{H}^2 \big(C^{\bullet \prime}_0(\widetilde{{\bf E}})\big) \to \C \to 0. 
\end{equation}
The image of the map $a$ must be the line spanned by the extension class $\rho$ of $\widetilde{\bf E}$. This follows from exactness of \eqref{seq:prime} and, from the fact that $ \mathbb{H}^1\big( C^{\bullet \prime}_0(\widetilde{{\bf E}})\big)$ classifies infinitesimal deformations of extensions and the deformation of any extension along its extension class is isomorphic to the trivial one, thus implying $\Ker \, b = \langle \rho \rangle$.

On the other hand, since $\P U^-$ parameterizes a family of extensions of the PHB ${\bf L}^-\in \mathcal{N}^-$ by  ${\bf L}^+\in \mathcal{N}^+$, there is a natural map
$$
T_{[\omega]}\P U^- \to \mathbb{H}^1\Big(\C \P^1,C_0^{\bullet \prime}\big({\bf E}([\omega])\big)\Big), 
$$
where the bundle $\mathbb{E}([\omega])$ is the extension determined by $[\omega]$. Therefore, one has the following maps between exact sequences
\begin{equation}\label{seq:down}
\xymatrix{
0 \ar[r] \ar[d] & (V^-)^* \ar[r]^\simeq \ar[d]^m & T \P U^- \ar[r] \ar[d] &  0 \ar[d] \\  
0 \ar[r] & \mathbb{H}^1\Big(C^\bullet\big(\widetilde{{\bf L}}^-, \widetilde{{\bf L}}^+ \otimes \mathcal{O}_{\P U^-}(1)\big)\Big)/\langle \rho \rangle    \ar[r]^<(.2){\simeq}  & \mathbb{H}^1(C^{\bullet \prime}_0 (\widetilde{{\bf E}})) \ar[r] & 0
}
\end{equation}
and, since the map $m$ is an isomorphism, one has that
\begin{equation}
\label{eq:m}
T \P U^- \cong \mathbb{H}^1\big(C^{\bullet \prime}_0(\widetilde{{\bf E}})\big). 
\end{equation}

Let us now consider the long exact sequence associated to
$$
0 \to C_0^{\bullet \prime}(\widetilde{{\bf E}}) \to C_0^\bullet(\widetilde{{\bf E}}) \to C^\bullet\big( \widetilde{{\bf L}}^+ \otimes \mathcal{O}_{\P U^-}(1),\widetilde{{\bf L}}^-\big) \to 0
$$
which is
$$
0 \to T \P U^- \to T \mathcal{H}^- \to \mathbb{H}^1\Big(C^\bullet\big( \widetilde{{\bf L}}^+ \otimes \mathcal{O}_{\P U^-}(1),\widetilde{{\bf L}}^-\big)\Big) \to \C \to 0,
$$
where we used \eqref{eq:m}, \eqref{seq:h2} and the fact that $T \mathcal{H}^- \cong \mathbb{H}^1 \big(C^\bullet_0(\widetilde{{\bf E}})\big)$. One concludes that the map $\P U^- \to \mathcal{H}^-$ is an embedding (it is injective by Proposition~\ref{prop:unique}) and that the map
$$
T \mathcal{H}^-\cong \mathbb{H}^1\Big(C_0^\bullet\big(\widetilde{{\bf E}}\big)\Big) \to   \mathbb{H}^1\big(C^\bullet(\widetilde{{\bf L}}^+ \otimes \mathcal{O}_{\P U^-}(1),\widetilde{{\bf L}}^-)\big), 
$$ 
whose image is the normal bundle of $\P U^-$ inside   $\mathcal{H}^-$, has corank $1$. This map is Serre dual to the map
$$
\mathbb{H}^1\Big(C^\bullet\big(\widetilde{{\bf L}}^-, \widetilde{{\bf L}}^+ \otimes \mathcal{O}_{\P U^-}(1)\big)\Big) \to \mathbb{H}^1\big(C_0^\bullet(\widetilde{{\bf E}})\big)
$$
which maps a deformation of the extension class $\rho$ of $\widetilde{{\bf E}}$ to a deformation of the bundle itself. Since a deformation in the direction of $\rho$ itself is isomorphic to a trivial deformation, the kernel of this map is the line through $\rho$. We conclude then that the normal bundle of $\P U^-$ inside $\mathcal{H}^-$ is the annihilator of $\rho$ in $\mathbb{H}^1\Big(C^\bullet\big( \widetilde{{\bf L}}^+ \otimes \mathcal{O}_{\P U^-}(1),\widetilde{{\bf L}}^-\big)\Big)$ which by \eqref{seq:down} is $V^-$.    

Let $\widetilde{\mathcal{H}}^-$ be the blow up of $\mathcal{H}^-$ along the image of the embedding $\P U^- \to \mathcal{H}^-$ with exceptional divisor $\P V^-$. Moreover, since the roles of plus and minus in the above arguments are completely interchangeable one can consider the blow up  $\widetilde{\mathcal{H}}^+$ of  $\mathcal{H}^+$ along the image of the embedding $\P U^+ \to \mathcal{H}^+$ with exceptional divisor $\P V^+$. Then we have the following result.

\begin{proposition}
There is an isomorphism $\widetilde{\mathcal{H}}^- \leftrightarrow   \widetilde{\mathcal{H}}^+$ such that the following diagram commutes
$$
\xymatrix{
\mathcal{H}^-\setminus \P U^-  \ar[r] \ar@{<->}[d] & \widetilde{\mathcal{H}}^-  \ar@{<->}[d] \ar@{<-^{)}}[r] &  \ar@{<->}[d] \,\, \P V^-\\
\mathcal{H}^+\setminus \P U^+  \ar[r] &  \widetilde{\mathcal{H}}^+   \ar@{<-^{)}}[r]  &  \,\, \P V^+.
}
$$
\end{proposition}
\begin{proof}
Let $\widetilde{{\bf E}}$ be the universal PHB over $\mathcal{H}^- \times \C \P^1$. By uniqueness of families of extensions, the restriction  $\widetilde{{\bf E}}\vert_{\P U^- \times \C\P^1}$ is isomorphic to the universal extension of $\widetilde{{\bf L}}^-$ by $\widetilde{{\bf L}}^+\otimes \mathcal{O}_{\P U^-}(1)$ tensored by the pull-back of a line bundle $F$ over $\P U^-$. Then the pull-back of $\widetilde{{\bf E}}$ to $\widetilde{\mathcal{H}}^- \times \C \P^1$ restricted to $\P V^-  \times \C \P^1$ has $\widetilde{{\bf L}}^+ \otimes F(1)$ as a sub PHB. Let $\widetilde{{\bf E}}^\prime$ be the elementary modification of the pull-back of $\widetilde{{\bf E}}$ to $\mathcal H^- \times \C\P^1$ along $\widetilde{{\bf L}}^+ \times F(1)$ as in Proposition 4.1 of \cite{T}. Then, for $x\notin \P V^-$,  $\widetilde{{\bf E}}^\prime_{\{ x\} \times \C \P^1} = \widetilde{{\bf E}}_{\{ x\} \times \C \P^1} $ while for $x \in \P V^-$,    $\widetilde{{\bf E}}^\prime_{\{ x\} \times \C \P^1}$ is an extension of $\widetilde{{\bf L}}^+$ by  $\widetilde{{\bf L}}^-$ with extension class $\rho_x \in \mathbb{H}^1(C^\bullet(\widetilde{{\bf L}}^+,\widetilde{{\bf L}}^-))$ obtained as the image of the normal space $\mathcal{N}_x(\P V^- / \widetilde{\mathcal{H}}^-)$ (see \cite{T} for details). Indeed, at every point $x\in \P V^-$ there are deformation maps 
$$
T_x  \widetilde{\mathcal{H}}^- \to \mathbb{H}^1\big(C^\bullet_0(\widetilde{{\bf E}})\big) \quad \text{and} \quad T_x \P V^- \to \mathbb{H}^1\big(C^{\bullet \prime}_0(\widetilde{{\bf E}})\big)
$$
and then the short exact sequence 
$$
0 \to C^{\bullet \prime}_0(\widetilde{{\bf E}})  \to C^\bullet_0(\widetilde{{\bf E}}) \to C^\bullet(\widetilde{{\bf L}}^+,\widetilde{{\bf L}}^-) \to 0
$$
determines a well-defined map from the ($1$-dimensional) normal space  $\mathcal{N}_x(\P V^- / \widetilde{\mathcal{H}}^-)$ to $\mathbb{H}^1\big(C^\bullet(\widetilde{{\bf L}}^+,\widetilde{{\bf L}}^-)\big)$, giving a class $\rho_x$ well-defined up to a scalar.

We then have the following commutative diagram for $x \in \P V^-$
$$
\xymatrix{
T_x \widetilde{\mathcal{H}}^- \ar[r] \ar[d] &  T_{\pi^-(x)}  \mathcal{H}^- \ar[r] \ar[d] & \,\,\,\mathbb{H}^1\Big(C^\bullet_0\big(\widetilde{{\bf E}}(x)\big)\Big) \ar[d] \\ 
\mathcal{N}_x(\P V^- / \widetilde{\mathcal{H}}^-) \ar[r] & V^-_{\pi^-(x)} \ar[r]^>(.7){\pi^+} & \,\,\, \mathbb{H}^1\big(C^\bullet(\widetilde{{\bf L}}_x^+,\widetilde{{\bf L}}_x^-)\big),
}
$$
where we used the fact that 
$$
\pi^- \big(\mathcal{N}_x(\P V^- / \widetilde{\mathcal{H}}^-)\big) = \mathcal{N}_{\pi^-(x)} \big(\P U^- / \mathcal{H}^- \big)=V^-_{\pi^-(x)},
$$
as well as Proposition~\ref{prop:GP} adapted to the traceless situation. This defines a map $\widetilde{\mathcal{H}}^- \stackrel{\varphi}{\to} \mathcal{H}^+$ which is an isomorphism away from the exceptional divisor $\P V^-$ and such that for $x \in \P V^-$ gives $\varphi(x)=\pi^+ (x)$, where $\pi^+$ is the forgetful morphism defined by the Euler sequence as in \eqref{eq:Euler}.

Interchanging plus and minus signs in the above argument one obtains maps
$$
\widetilde{\mathcal{H}}^- \to \mathcal{H}^+ \quad \text{and} \quad \widetilde{\mathcal{H}}^+ \to \mathcal{H}^-.
$$
Using these along with the blow-down maps $\widetilde{\mathcal{H}}^\pm \to \mathcal{H}^\pm$ one obtains injections of $\widetilde{\mathcal{H}}^+$ and $\widetilde{\mathcal{H}}^-$ into $\mathcal{H}^+ \times \mathcal{H}^-$. Clearly these maps are embeddings and their images are both equal to the closure of the image of $\widetilde{\mathcal{H}}^\pm \setminus \P V^\pm$ and the result follows. 
\end{proof}

\subsection{Wall-crossing for hyperpolygons}\label{sec:wchyper}
Now that we have studied the changes in $\mathcal{H}(\beta)$ as $\beta$ crosses a wall we will use the isomorphism constructed in Section~\ref{sec:isomorphism} to analyze the behavior of the corresponding spaces of hyperpolygons $X(\alpha)$ (with $\alpha=\beta_2-\beta_1$). First note that by rescaling if necessary one can assume that all hyperpolygon spaces considered in this section have weights $\alpha_i<1$.  

Let $W$ be a wall separating two adjacent chambers $\widetilde{\Delta}^-$ and $\widetilde{\Delta}^+$ of admissible values of $\alpha$ and let $S$ be an index set in $\{1, \ldots, n\}$ associated to $W$. Exchanging $S$ with $S^c$ if necessary one can assume that $S$ is short for every $\widetilde{\alpha}^-\in \widetilde{\Delta}^-$ and long for every  $\alpha^+\in \widetilde{\Delta}^+$. Then one sees that the corresponding spaces of PHBs suffer a Mukai transformation as described above for $\mathcal{H}(\beta^\pm)$. Note that the wall $W$ uniquely determines a wall in $Q$ (defined by the same equation $\varepsilon_S(\alpha)=0$) separating two chambers $\Delta^+, \Delta^- \subset Q$ of nongeneric parabolic weights. 

Let $X^\pm$ be  hyperpolygon spaces for values $\alpha^\pm \in \widetilde{\Delta}^\pm$. Then $X^+$ and $X^-$ suffer a Mukai transformation where $X^-$ is blown up along an embedded $\C \P^{n-3}$ and then blown down in the dual direction giving rise to a new embedded $\C \P^{n-3}$. Therefore, one sees (as observed by Konno in \cite{konno}) that $X^+$ and $X^-$ are diffeomorphic. Let us study this transformation in more detail. 

The embedded $\C \P^{n-3}$ that is blown-up in $X^-$ corresponds to $\P U^-$ in $\mathcal{H}^-$ by the isomorphism of Theorem~\ref{isomorphism}. In fact, $\P U^-$ is the space of PHBs in $\mathcal{H}^-$ that are not stable for $\beta^+ \in \Delta^+$. Hence, any PHB $\bf{E}$ in  $\P U^-$ has a destabilizing subbundle $\bf L$ which is topologically trivial and is such that 
$$
S=S_L=\{i \in \{1, \ldots, n\}\mid L_{x_i}=E_{x_i,2}\}
$$
is a maximal straight set. Moreover, as seen in the proof of Theorem~\ref{isomorphism} the fact that $\bf E$ is $\Delta^-$-stable implies that the corresponding hyperpolygon $\mathcal{F}({\bf E})=[p,q]_{\alpha^--\text{st}}$ in $X^-$ satisfies $p_i=0$ for every $i \in S^c$. By stability of hyperpolygons (cf. Theorem~\ref{alpha-stability}) one has that $S$ is $\widetilde{\Delta}^-$-short. Hence, the image of $\P U^-$ under the isomorphism $\mathcal{F}$ is the core component $U_S^- \cong \C \P^{n-3}$ (cf. Theorem~\ref{Smaximal}).

Similarly, one concludes that $\P U^+$ corresponds to $U_{S^c}^+\cong  \C \P^{n-3}$ in $X^+$ and so we have the following result.

\begin{Theorem}\label{thm:wch}
Let $X^+$ and $X^-$ be hyperpolygon spaces for $\alpha^+$ and $\alpha^-$ on either side of a wall $W$ of discrete data $S$. Then $X^-$ and $X^+$ are related by a Mukai transformation where $X^\pm$  have a common blow up obtained by blowing up $X^-$ along the core component $U_S^-$ and by blowing up $X^+$ along the core component $U_{S^c}^+$. The common exceptional divisor is a partial flag bundle $\P( T^* \C \P^{n-3})\cong \P( T^* U_S^-)\cong\P(U_{S^c}^+)$.  
\end{Theorem}

Even though $X^+$ and $X^-$ are diffeomorphic they are not isomorphic as $S^1$-spaces, for the $S^1$-action in \eqref{action}, and the corresponding cores 
$$\frak{L}_{\alpha^\pm}= M(\alpha^\pm) \cup \bigcup_{B \in \mathcal{S}^\prime (\alpha^\pm)} U_B^\pm$$ do change under the Mukai transformation. 

All the fixed point set components $X_B^-$ with $B\in \mathcal{S}^\prime(\alpha^-)$ remain unchanged except for $X_S^-\simeq \C \P^{\vert S \vert -2}$ which is substituted by $X_{S^c}^+\cong \C \P^{\vert S^c \vert -2}$. 

The fixed point set component $M(\alpha^-)$ suffers a blow up along 
$$U_S^-\cap M(\alpha^-)=  M_S(\alpha^-)$$ 
followed by a blow down resulting in a new polygon space $M_{S^c}(\alpha^+)=U_{S^c}^+ \cap M(\alpha^+)$ embedded in $U_{S^c}^+$ (see Section~\ref{Walls}). 

The core components $U_B^-$ for which $B\cap S \neq \varnothing$ but $B \not\subset S$ are not affected by the Mukai transformation and remain unchanged as $U_B^+$. Indeed, since $S$ is a maximal $\Delta^-$-short set, $B\cup S$ is long and so $U_S^- \cap U_B^- = \varnothing$.

If $B\varsubsetneq S$ then
$$
U_B^- \cap U_S^- = \big\{[p,q]\in U_S^-\mid p_j=0 \, \text{for all}\, j \in S\setminus B \big\}  
$$  
and so $U_B^-$ suffers a blow up along $U_B^-\cap U_S^-$ followed by a blow down of the exceptional divisor 
$$V_B=\P\big(T^*(U_B^-\cap U_S^-)\big),$$ 
resulting in the core component $U_B^+$. Note that if one blows up $U_B^+$ along $U_B^+\cap U_{S^c}^+= M_B^+ \cap M_{S^c}^+$ (since  $B \cap S^c = \varnothing$), one obtains the exceptional divisor $V_B$ inside the common blow up of $U_B^-$ and $U_B^+$.

Finally, if $B \subset S^c$ then $U_B^-$ suffers a blow up along 
$$U_B^-\cap U_S^-= M_B^- \cap M_S^-$$ 
followed by a blow down of the exceptional divisor $V_B$ resulting in the core component $U_B^+$. Again, if one blows up $U_B^+$ along 
$$
U_B^+ \cap U_{S^c}^+ = \big\{[p,q]\in U_{S^c}^+\mid p_j=0 \, \text{for all}\, j \in S^c\setminus B \big\},  
$$ 
one obtains the  exceptional divisor $V_B$.    

\begin{example}
Let $n=5$ and consider $\alpha^-=(2,1,5,1,2)$ and $\alpha^+=(3,1,5,1,2)$ on either side of the wall $W_S$ with $S={\{1,2,5\}}$. The corresponding collections of short sets of cardinality greater or equal to $2$ are
\begin{align*}
\mathcal{S}^\prime(\alpha^-)=\Big\{ &  \{1,2\},\{1,4\},\{1,5\}, \{2,4\}, \{2,5\}, \{4,5\},\{1,2,4\}, {\bf \{1,2,5\}},  \\ &  \{1,4,5\},\{2,4,5\}\Big\}
\end{align*}
and 
\begin{align*}
\mathcal{S}^\prime(\alpha^+)=\Big\{ & \{1,2\},\{1,4\},\{1,5\}, \{2,4\}, \{2,5\}, {\bf \{3,4\}}, \{4,5\},\{1,2,4\},  \\ & \{1,4,5\},\{2,4,5\}\Big\}.
\end{align*}
Crossing the wall $W_S$ we see that the core component $U^-_{\{1,2,5\}}\cong \C \P^2$ disappears as a result of the Mukai transformation, being replaced by the new core component $U^+_{\{3,4\}}\cong \C \P^2$. The other core components affected are those relative to elements of $\mathcal{S}^\prime(\alpha^-)$ which are subsets of $S$ (i.e. $\{1,5\}$, $\{1,2\}$ and $\{2,5\}$). In Figures~\ref{fig:wallcrossing1} and ~\ref{fig:wallcrossing2} we represent these changes. There, the critical components are pictured by shaded ellipses or dots (when $0$-dimensional) while other ellipses represent copies of $\C \P^1$ flowing between two fixed points.
\begin{figure}
\includegraphics[width=7cm]{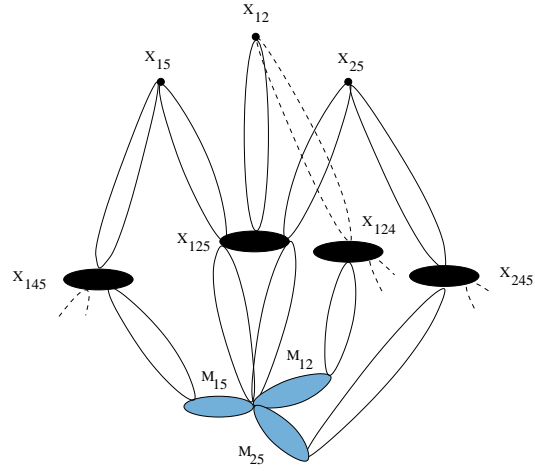}
\caption{Relevant part of the core of $X(\alpha^-)$ before crossing the wall $W_{\{1,2,5\}}$.}
\label{fig:wallcrossing1}
\end{figure}

\begin{figure}
\includegraphics[width=12cm]{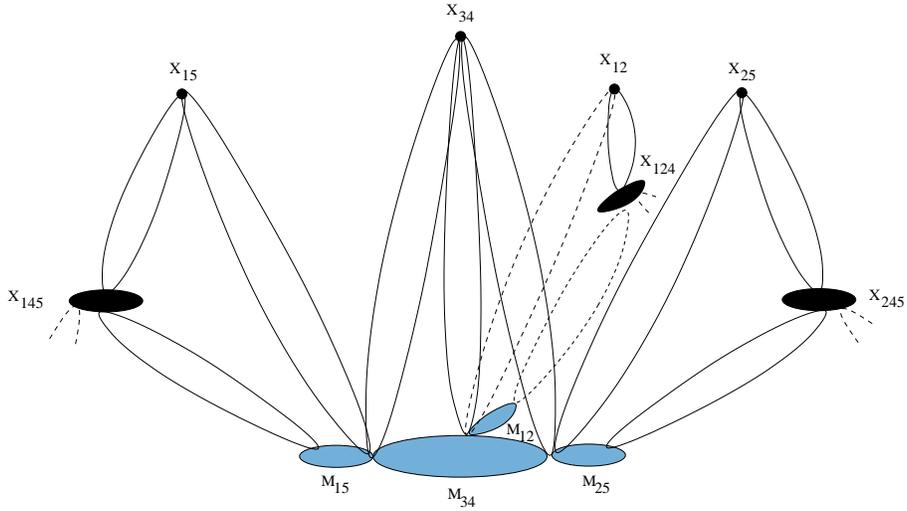}
\caption{Relevant part of the core of $X(\alpha^+)$ after crossing the wall $W_{\{1,2,5\}}$.}
\label{fig:wallcrossing2}
\end{figure}
\end{example}

\begin{Remark} By the above arguments it is clear that the submanifolds $\P U^-$ and $\P U^+$ of $\mathcal{H}^-$ and $\mathcal{H}^+$ involved in the Mukai flop are the nilpotent cone components $\mathcal U_{(0,S)} \subset \mathcal H^-$ and $\mathcal U_{(0,S^c)} \subset \mathcal H^+$, defined as the closure of the flow-down set \eqref{eq:nilcone}. Moreover, the changes in the different core components of $X^\pm$ as one crosses a wall translate to changes in the corresponding components of the nilpotent cone in $\mathcal{H}^\pm$. In particular the birational map between polygon spaces $M(\alpha^\pm)$ studied in \cite{m} and described in Section~\ref{Walls} translates to the birational map between $\mathcal{M}_{\beta^\pm,2,0}^{0,\Lambda}$ studied in \cite{BH} and described in Section~\ref{sec:rk2}.
\end{Remark}

\section{Intersection numbers for hyperpolygon spaces}\label{sec:intnum}
The intersection numbers of polygon spaces $M(\alpha)$ are computed in \cite{AG}. In this section we give explicit formulas for the computation of the intersection numbers of the remaining core components $U_S$ for $S\in \mathcal{S}^\prime(\alpha)$.

\subsection{Circle bundles}

As in \cite{konno} one constructs circle bundles over $X(\alpha)$ as follows. For each $1\leq i \leq n$ one can define the spaces
$$
\widetilde{Q}_i=\left\{ (p,q) \in \mu_\mathbb{R}^{-1}(0,\alpha) \cap \mu_\mathbb{C}^{-1}(0)\mid  (q_iq_i^*-p_i^*p_i)_0= \left(\begin{array}{cc} t & 0 \\ 0 & -t \end{array} \right), \, t>0\right\}. 
$$
Note that the vectors $(q_iq_i^*-p_i^*p_i)_0$ live in ${\bf i} \, \frak{su}(2) \cong \frak{su}(2)^* \cong \mathbb{R}^3$ and that, under this identification, $\widetilde{Q}_i$ is the set of points $(p,q)$ for which $(q_iq_i^*-p_i^*p_i)_0=(0,0,\alpha_i+\vert p_i\vert^2)$. 
One then considers the representation
$$
\rho_{SO(3)}:K \to SO(3) \simeq SO(\frak{su}(2))
$$
defined by
$$
\rho_{SO(3)}([A, e_1, \ldots, e_n])=\text{Ad}(A),
$$
where $\text{Ad}$ is the adjoint representation of $SU(2)$, and take the quotient 
$$
Q_i:=\widetilde{Q}_i/\ker \rho_{SO(3)}.
$$ 
Define an $S^1$-action on $Q_i$ by the following injective homomorphism of $S^1$ into $K$
\begin{equation}
\label{eq:actionqi}
\iota_{Q_i}(e^{\mathbf{i} t}) = \left[\left(\begin{array}{cc} e^{\mathbf{i} t} & 0 \\ 0 & e^{-\mathbf{i} t} \end{array}\right), 1, \ldots, 1  \right].
\end{equation}
Since $\iota_{Q_i}^{-1}(\ker \rho_{SO(3)})=\{\pm 1\}$, one gets an effective (right) $S^1 / \{\pm 1\}$-action on $Q_i$ thus obtaining a principal $S^1 / \{\pm 1\} $-bundle over $X(\alpha)$. The line bundle associated to $Q_i$ is then
$$
L_i=Q_i \times_{\rho_i} \C,
$$
where $\rho_i:K \to S^1$ is the representation given by
$$
\rho_i([A,e_1, \ldots,e_n])=e_i^2
$$
(see Section $6$ in \cite{K}).
Restricting the bundle $Q_i$ to the polygon space $M(\alpha)$ one obtains a principal circle bundle $Q_i\vert_{M(\alpha)} \to M(\alpha)$. Comparing it with the $S^1$-bundle $V_i \to M(\alpha)$
considered in \cite{AG} and given by
\begin{equation}\label{eq:vi}
V_i:= \left\{ v \in \prod_{j=1}^n S^2_{\alpha_j}\mid \sum_{j=1}^n v_j=0, \,\, \text{and}\,\, v_i=(0,0,\alpha_i)\right\},
\end{equation}
where the circle acts by standard rotation around the $z$-axis, one sees that 
$$
c_1(V_i)=-c_1\left(Q_i\vert_{M(\alpha)} \right)
$$
since the $S^1$-action  on $Q_i$ is a right action.

For this reason, we will work instead with the circle bundles
$$
\widetilde{V}_i \to X(\alpha)
$$
defined as the principal circle bundles over $X(\alpha)$ associated to the dual line bundles $L^*_i$. Note that, under the identification of ${\bf i}\, \frak{su}(2) \cong \frak{su}(2)^* \cong \mathbb{R}^3$, the circle acts on $\widetilde{V}_i$ by standard (left) rotation around the $z$-axis and so
$$
\widetilde{V}_i\vert_{M(\alpha)} = V_i.
$$
From now on we will denote  the first Chern classes of these bundles by
$$
c_j:=c_1(\widetilde{V}_j)\in H^2(X(\alpha),\mathbb{R}).
$$
Performing reduction in stages one can see hyperpolygon spaces
$$
X(\alpha):= \frac{\mu_{\mathbb{R}}^{-1}(0,\alpha) \cap \mu_{\mathbb{C}}^{-1}(0)}{K}
$$
as a quotient of a product of the cotangent bundles $T^*S^2_{\alpha_i}$ by $SO(3)$. Consider then the diagonal $S^1$-action on 
\begin{equation}
\label{eq:cot}
T^* S^2_{\alpha_1} \times \cdots \times T^* S^2_{\alpha_n}
\end{equation}
given by the following injective homomorphism of $S^1$ into $SU(2)/\pm I$
$$
\iota(e^{{\bf{i}}t})=\left[ \left( \begin{array}{cc} e^{{\bf{i}}t} & 0 \\ 0 & e^{{\bf{i}}t} \end{array}\right) \right].
$$
This action is Hamiltonian with moment map
$$
\begin{array}{rcl}
\mu_{S^1}: \displaystyle\prod_{i=1}^n   T^* S^2_{\alpha_i} & \to &  \mathbb{R} \\
 (p,q) & \mapsto &   \zeta\Big(\displaystyle\sum_{i=1}^n (q_iq_i^* - p_i^*p_i)_0\Big), 
\end{array}
$$
where $\zeta(x,y,z)=z$ is the height of the endpoint of $\sum_{i=1}^n (q_iq_i^* - p_i^*p_i)_0$ under the usual identification of $\frak{su}(2)^*$ with $\mathbb{R}^3$.

In analogy with the polygon space case one defines the abelian hyperpolygon space
$$
\mathcal{A} X(\alpha) =\left\{ (p,q) \in \prod_{i=1}^{n-1} T^* S^2_{\alpha_i} \mid \, \zeta\left( \sum_{i=1}^{n-1} (q_iq_i^* -p_i^*p_i)_0 \right)= \alpha_n \right\} 
$$
which is the set of those $(p,q)$ for which the vector $\sum_{i=1}^{n-1} (q_iq_i^* - p_i^*p_i)_0$ in $\mathbb{R}^3$ ends on the plane $z=\alpha_n$ modulo rotations around the $z$-axis. (Here we take $S^1 \simeq SO(2)$ as a subgroup of $SO(3)$ acting on the right.) It is the symplectic quotient of 
\begin{equation}
\label{eq:prodcot}
\prod_{i=1}^{n-1} T^* S^2_{\alpha_i}
\end{equation}
by the above circle action,
$$
\mathcal{A} X(\alpha) = \mu_{S^1}^{-1}(\alpha_n)/S^1,
$$
and so it is a symplectic manifold of dimension $4n-6$.

\begin{Remark}\label{rmk:6}
It is always possible to act on any element $[p,q]$ of $X(\alpha)$ by an element of $K$ in such a way that the vector $\sum_{i=1}^{n-1} (q_iq_i^* - p_i^*p_i)_0$ ends not only on the plane $z=\alpha_n$ but also so that $(q_nq_n^* - p_n^*p_n)_0$ points downwards.
\end{Remark}
Since $\alpha$ is generic, the circle acts freely on the level set $B:=\mu_{S^1}^{-1}(\alpha_n)$ and so $B \to \mathcal{A} X(\alpha)$ is a principal circle bundle. Moreover, one has the following commutative diagram
\begin{eqnarray*}
Q_n(\alpha) & \stackrel{\tilde{i}}{\to} &  B \\
\downarrow & & \downarrow \\
X(\alpha) & \stackrel{i}{\hookrightarrow} & \mathcal{A} X(\alpha)
\end{eqnarray*} 
where the inclusion $\tilde{i}:Q_n(\alpha) \to B$ is anti-equivariant since, in the identification of $X(\alpha)$ as  a submanifold of $\mathcal{A} X(\alpha)$, the vector $(q_n q_n^* - p_n^*p_n)_0$ must face downward (see Remark~\ref{rmk:6}).
Therefore,
$$
c_n := c_1(\widetilde{V}_n) = -c_1(Q_n) = i^*(c_1(B)).
$$
On the other hand, since $\mathcal{A} X(\alpha)$ is the reduced space
$$
\mu_{S^1}^{-1}(\alpha_n)/S^1 = B/S^1,
$$
one has by the Duistermaat Heckmann Theorem that 
$$
c_1(B)=\frac{\partial}{\partial \alpha_n}[\omega_\mathbb{R}]
$$
in $H^2(\mathcal{A} X(\alpha), \mathbb{R})$, and so
$$
c_n = \frac{\partial}{\partial \alpha_n}[\omega_\mathbb{R}]
$$
in $H^2(X(\alpha), \mathbb{R})$.
By symmetry, interchanging the order of the spheres in \eqref{eq:prodcot}, one obtains
\begin{equation}
\label{eq:cj}
c_j=\frac{\partial}{\partial \alpha_j}[\omega_\mathbb{R}].
\end{equation}
It is shown in \cite{konno} and \cite{hp} that these classes generate $H^*(X(\alpha),\mathbb{Q})$.
\subsection{Dual homology classes}

In this section we determine homology classes representing the first Chern classes $c_j\in H^2(X(\alpha), \mathbb{Q})$. 
For that consider $i$ and $j$, $1\leq i,j \leq n$, with $i\neq j$ and denote by $D_{i,j}(\alpha)$ the submanifold of $X(\alpha)$ formed by hyperpolygons $[p,q]$ for which $(q_iq_i^*-p_i^*p_i)_0$ and $(q_jq_j^*-p_j^*p_j)_0$ are parallel as vectors in $\mathbb{R}^3$. It is not restrictive to assume that both these vectors are parallel to the $z$-axis. Clearly $D_{i,j}(\alpha)$ has two connected components
\begin{eqnarray*}
D_{i,j}^+(\alpha) & =\{ [p,q] \in D_{i,j}(\alpha)\mid \, \langle(q_iq_i^*-p_i^*p_i)_0,(q_jq_j^*-p_j^*p_j)_0 \rangle >0 \} \\
D_{i,j}^-(\alpha) & =\{ [p,q] \in D_{i,j}(\alpha)\mid \, \langle(q_iq_i^*-p_i^*p_i)_0,(q_jq_j^*-p_j^*p_j)_0 \rangle <0 \}. 
\end{eqnarray*}
Moreover one has the following result.
\begin{proposition}
The circle bundle
$$
\widetilde{V}_j \vert_{X(\alpha)\setminus D_{i,j}(\alpha)} \stackrel{\pi_j}{\to} X(\alpha)\setminus D_{i,j}(\alpha) 
$$
has a section $s_{i,j}: X(\alpha)\setminus D_{i,j}(\alpha) \to \widetilde{V}_j \vert_{X(\alpha)\setminus D_{i,j}(\alpha)}$.
\end{proposition}
\begin{proof}
Let $[p,q]\in X(\alpha)$ and take $i \neq j$. Then assign to $[p,q]$ the unique element in $\pi_j^{-1}([p,q])$ for which $(q_iq_i^*-p_i^*p_i)_0$ projects onto the $xOy$-plane along the positive $y$-axis. Such a representative always exists in $\pi_j^{-1}([p,q])$ as long as $[p,q]\notin  D_{i,j}(\alpha)$. 
\end{proof}
On the other hand, let us consider the function
$$
\tilde{t}_j:\mu_\mathbb{R}^{-1}(0,\alpha) \cap \mu_\mathbb{C}^{-1}(0) \to \mathbb{C}
$$
defined by
$$
\tilde{t}_j(p,q)= \left\{ \begin{array}{c}  \frac{b_j}{c_j},\,\,\, \text{if}\,\, c_j \neq 0 \\ \\ -\frac{a_j}{d_j},\,\,\, \text{if}\,\, d_j \neq 0, \end{array} \right.
$$
where, as usual, $p_j=(a_j,b_j)$ and $q_j=\left(\begin{array}{r} c_j \\ d_j \end{array} \right)$.
This map is well-defined since, if $c_j,d_j \neq 0$, one has by \eqref{complex1} that
$$
\frac{b_j}{c_j} =  -\frac{a_j}{d_j}.
$$
Moreover, it is $K$-equivariant with respect to $\overline{\rho}_j$ since
$$
\overline{t}_j((p,q)\cdot[A,e_1,\ldots,e_n])=e_j^{-2} \, \overline{t}_j(p,q),
$$
and so it induces a section $t_j$ of $L_j$ vanishing on 
$$
W_j:=\{[p,q]\in X(\alpha) \mid\, p_j=0\}.
$$ 
Hence we obtain the following proposition.
\begin{proposition}\label{prop:wj}
The line bundle
$L_j \vert_{X(\alpha)\setminus W_j} \stackrel{\pi_j}{\to} X(\alpha)\setminus W_j $
has a section.
\end{proposition}
We conclude that $c_j$ is represented in Borel-Moore homology by both $D_{i,j}(\alpha)$ ($i\neq j$) and by $-W_j$.
\subsection{Restriction to a core component}
We will restrict the circle bundles defined in the previous sections to a core component $U_S$ and determine the Poincar\'{e} Dual of the Chern classes of these restrictions. For that, recall that $D_{i,j}(\alpha)$ has two connected components $D_{i,j}^\pm(\alpha)$. Then, if $i \neq j$ and $i,j \notin S$, the intersection $D_{i,j}^\pm(\alpha) \cap U_S(\alpha)$ is diffeomorphic to a core component $U_S(\alpha^{\pm})$ for a lower dimensional hyperpolygon space $X(\alpha^\pm)$.
\begin{proposition}\label{prop:map}
Assuming $S=\{1, \ldots, \vert S\vert \}$ and $\alpha_i > \alpha_j$ with $i,j \notin S$ there exist diffeomorphisms
$$
s_{\pm} :D_{i,j}^\pm(\alpha) \cap U_S(\alpha) \to U_S (\alpha^\pm), 
$$
with
\begin{equation}\label{eq:map}
s_{\pm} \left( [p,q] \right) = \left[p_1,\ldots,p_{\vert S\vert}, 0,\ldots, 0, q_1,\ldots,\hat{q}_i,\ldots, \hat{q}_j, \ldots, q_n, \sqrt{\frac{\alpha_i\pm\alpha_j}{\alpha_i}}\, q_i \right],
\end{equation}
where
$$
\alpha_{i,j}^\pm :=(\alpha_1, \ldots,\hat{\alpha}_i, \ldots, \hat{\alpha}_j, \ldots, \alpha_n, \alpha_i \pm \alpha_j).
$$
\end{proposition} 
\begin{Remark}
By permutation it is not restrictive to assume $S=\{1, \ldots, \vert S \vert\}$. Moreover, note that both $\alpha_{i,j}^\pm$ are generic provided that $\alpha$ is. 
\end{Remark}
\begin{proof}
From \cite{hp} we know that $U_S(\alpha)$ is homeomorphic to the moduli space of $n+1$ vectors
$$
\big\{ u_l, v_k, w \in \mathbb{R}^3\mid \,\, l \in S,\,\, k \in S^c\big\}
$$
satisfying conditions $1)$ to $5)$ in Theorem~\ref{hp}, taken up to rotation. Moreover, in $D_{i,j}^\pm(\alpha)\cap U_S(\alpha)$ one has $v_j=\lambda v_i$ for some $\lambda \in \mathbb{R}^\pm$ and so one can trivially identify this intersection with the moduli space of $n$ vectors
$$
\big\{ u_l, v_{\vert S\vert +1}, \ldots, \hat{v}_i,\ldots, \hat{v}_j, \ldots, v_{n}, v_i\pm v_j, w \mid l \in S \big\}
$$ 
satisfying
\begin{align*}
1)& \quad w  + v_{\vert S\vert +1}+\cdots + v_{i-1}+v_{i+1}+\cdots + v_{j-1} + v_{j+1} +\cdots +v_{n} + (v_i \pm v_j) = 0 \\
2)& \quad \sum_{l \in S} u_l =0\\
3)& \quad u_l \cdot w = 0, \quad \text{for all} \,\, l\in S\\
4)& \quad \vert \vert v_k\vert \vert= \alpha_k,\,\, k \neq i,j\quad  k\in S^c,  \quad \vert \vert v_i \pm v_j\vert \vert =\alpha_i \pm \alpha_j\\
5)& \quad \vert \vert w \vert \vert = \sum_{l \in S} \sqrt{\alpha_l^2 + \vert \vert u_l\vert \vert^2}, 
\end{align*}
which, in turn,  is homeomorphic to $U_{S}(\alpha^\pm)$ (cf. Figure~\ref{figB}). The composition of these homeomorphisms defines the map
$$
s_{\pm}: D_{i,j}^\pm(\alpha) \cap U_S(\alpha) \to U_S(\alpha^\pm)
$$ 
of \eqref{eq:map}.  Note that the map $s_{\pm}$ is clearly a diffeomorphism between the two manifolds.
\end{proof}
\begin{figure}[htbp]
\begin{center}
\psfrag{w}{\small{$w$}}
\psfrag{v1}{\small{$v_{\lvert S \rvert +1}$}}
\psfrag{vi}{\small{$v_i$}}
\psfrag{vj}{\small{$v_j$}}
\psfrag{vn}{\small{$v_n$}}
\psfrag{u1}{\small{$u_1$}}
\psfrag{uS}{\small{$u_{\vert S \vert}$}}
\psfrag{I}{$(I)$}
\psfrag{II}{$(II)$}
\includegraphics[width=12cm]{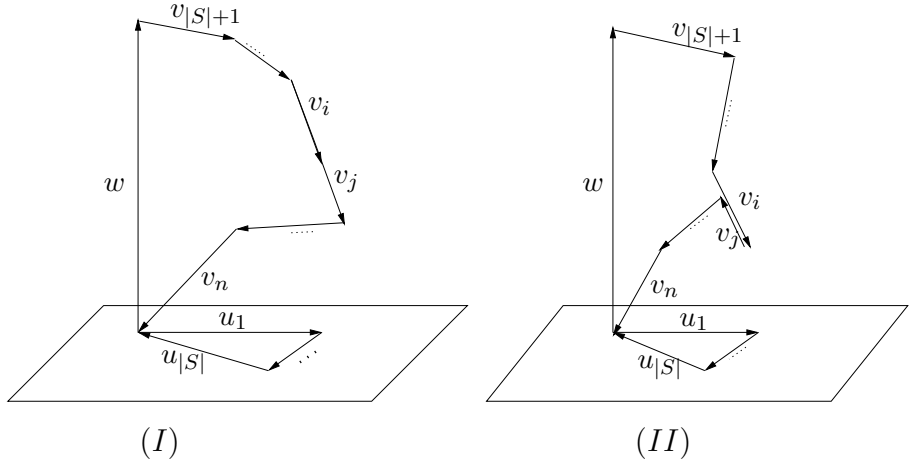}
\caption{$(I)$ A hyperpolygon in $D_{i,j}^+ (\alpha) \cap U_S(\alpha)$;
$(II)$ A hyperpolygon in $D_{i,j}^- (\alpha) \cap U_S(\alpha)$.}
\label{figB}
\end{center}
\end{figure}
We conclude that the manifolds $ D_{i,j}^\pm(\alpha) \cap U_S(\alpha)$ are connected and symplectic and so we can orient them using the symplectic form by requiring
$$
\int_{D_{i,j}^\pm (\alpha) \cap U_S(\alpha)} (i_S^\pm \circ s_\pm)^* (\omega_\mathbb{R}^\pm )^{n-4} > 0,
$$
where 
$$
i_S^\pm:U_S(\alpha^\pm) \to X(\alpha^\pm)
$$
is the natural inclusion map.
One obtains in this way two generators of 
$$H_{2(n-4)}(D_{i,j}^\pm(\alpha) \cap U_S(\alpha)),$$
namely $[D_{i,j}^+(\alpha) \cap U_S(\alpha)]$ and $[D_{i,j}^-(\alpha) \cap U_S(\alpha)]$. Hence, to determine the Poincar\'{e} dual of the class $i_S^* \, c_j$, where $i_S:U_S(\alpha) \to X(\alpha)$ is the inclusion map, one just has to determine constants $a_{i,j},b_{i,j}$ as follows.
\begin{proposition}\label{prop:dij}
Let $i:D_{i,j}(\alpha) \cap U_S(\alpha) \to X(\alpha)$ be the inclusion map. If $\alpha_i \neq \alpha_j$ and $i,j \notin S$ then the Poincar\'{e} dual of $i_S^*\, c_j$ is in $i_* H_{2(n-4)} (D_{i,j}(\alpha)\cap U_S(\alpha))$ and can be written as
$$
a_{i,j} [D_{i,j}^+ (\alpha) \cap U_S(\alpha)] + b_{i,j} [D_{i,j}^- (\alpha)  \cap U_S(\alpha)],
$$
where
$$
a_{i,j}=1 \quad \text{and}\quad b_{i,j}=\sgn (\alpha_i-\alpha_j).
$$
\end{proposition}
\begin{proof}
For simplicity, consider $i=n-1$, $j=n$ and $S=\{1, \ldots, \vert S\vert\}$. Then take  a fixed element in $U_S(\alpha)$ with $p_i=0$ for all $i\geq 3$. Let $(p^0,q^0)$ be a fixed representative of this class.  Consider the subvariety $N$ of $U_S(\alpha)$ defined by the elements $[(p,q)]$ of $U_S(\alpha)$ with $p_i=p_i^0$ for all $i$ and $q_i=q_i^0$ for $i=1,\ldots, n-3$. This subvariety $N$ is thus obtained by fixing $p_i$ for all $i$, and $q_i$ for all $i \leq n-3$, allowing only to vary the last three values $q_{n-2}$, $q_{n-1}$ and $q_n$ (noting that the corresponding coordinates of $p$ are $p_{n-2}=p_{n-1}=p_n=0$). It is then symplectomorphic to the moduli space of polygons in $\mathbb{R}^3$
$$
M(l,\alpha_{n-2},\alpha_{n-1}, \alpha_n),
$$
with 
$$
l =\Big\lvert \Big\lvert \sum_{k=1}^{n-3} \big(q_k^0(q_k^0)^*\big)_0 - \big((p_1^0)^*p_1^0\big)_0 - \big((p_2^0)^*p_2^0\big)_0   \Big\rvert \Big\rvert, 
$$
which we know is a sphere.  Note that $N$ is homeomorphic to the moduli space of vectors
$u_1,u_2,v_k,w \in \R^3$, $k \in S^c$, such that
\begin{align*}
 u_1 & = -u_2=q_1^0p_1^0+(p_1^0)^*(q_1^0)^*,\\ 
v_k & = \big(q_k^0(q_k^0)^*\big)_0, \quad \forall k=\vert S \vert +1, \ldots, n-3,\\
 w & = \sum_{i \in S}   \big(q_k^0(q_k^0)^*\big)_0 - \big((p_1^0)^*p_1^0\big)_0 - \big((p_2^0)^*p_2^0\big)_0.
\end{align*}
On the other hand, $N$ equipped with the bending action along the first diagonal is a toric manifold with moment polytope given by  the interval
$$
\Delta=\left[\max \{\vert l - \alpha_{n-2}\vert, \vert\alpha_{n-1}-\alpha_{n} \vert \}, \min \{ l + \alpha_{n-2}, \alpha_{n-1}+\alpha_{n}  \}   \right],
$$ 
(cf. \cite{hk, km1} for details) and so we can use the following well-known fact about toric manifolds.

Consider a family of symplectic forms $\Omega_t$ on a toric manifold and the corresponding family of moment polytopes $\Delta_t$ with $m$ facets given, as usual, by
$$
F_{t,k}:=\big\{ x\in \frak{t}^*\mid \langle x, \nu_k\rangle = \lambda_k(t) \big\} \quad \text{for} \quad k=1, \ldots,m,
$$
with $\nu_k$ the inward unit normal vector to the facet $F_{t,k}$ and $\lambda_k(t) \in \R$. Suppose that the polytopes $\Delta_t$ stay combinatorially the same as $t$ changes but the value of $\lambda_i(t)$ for some $i \in \{1, \ldots,m\}$ depends linearly on $t$ and, as $t$ increases, the facet $F_{t,i}$ moves outwards while the others stay fixed. Then, $\frac{d \Omega_t}{dt}$ is the Poincaré dual of the homology class $[\mu^{-1}(F_{t,i})]$ where the orientation is given by requiring that 
$$\int_{\mu^{-1}(F_{t,i})} \Omega_t^{\frac{1}{2}(\dim \mu^{-1}(F_{t,i}))} >0$$
(cf. Section 2.2 of \cite{G} for details). 

Applying this result to the submanifold $N$ we see that, as $\alpha_n$ changes, the cohomology of the symplectic form on $N$ 
$$
[(i_S \circ i_N)^*\omega_{\R}]
$$
changes by the Poincaré dual of the homology class 
\begin{align*}
& [\mu^{-1}(\alpha_{n-1}+\alpha_n)\cap U_S(\alpha)] + \sgn (\alpha_{n-1}-\alpha_{n}) [\mu^{-1}(\lvert\alpha_{n-1}-\alpha_n\rvert)\cap U_S(\alpha)] \\ & = 
 [D_{n-1,n}^+(\alpha)\cap U_S(\alpha)\cap N] + \sgn (\alpha_{n-1}-\alpha_{n}) [D_{n-1,n}^-(\alpha) \cap U_S(\alpha) \cap N].
\end{align*}  
The result then follows from the fact that
$$
(i_S^* \, c_n)\vert_N = i_N^* \, i_S^* c_n = (i_S \circ i_N)^* \frac{\partial}{\partial \alpha_n} [\omega_\mathbb{R}] = \frac{\partial}{\partial \alpha_n} [(i_S \circ i_N)^* \omega_\mathbb{R}]. 
$$
\end{proof}
\subsection{Recursion formula}
To prove our recursion formula we have to first study the behavior of the classes $c_j$ when restricted to 
$$[D_{n-1,n}^\pm \cap U_S(\alpha)].$$
\begin{proposition}
\label{prop:pullbackcj} Suppose $\alpha_n \neq \alpha_{n-1}$ and
let $c_n^+$ and $c_n^-$ be the cohomology classes
$c_1\big(\widetilde{V}_n(\alpha^+)\big)$ and $c_1\big(\widetilde{V}_n(\alpha^-)\big)$, where 
$$
\alpha^+:=(\alpha_1,\ldots,\alpha_{n-2},\alpha_{n-1} +  \alpha_n) \quad \text{and} \quad \alpha^-:=(\alpha_1,\ldots,\alpha_{n-2},\lvert \alpha_{n-1} - \alpha_n\rvert ).
$$
Then, considering  the inclusion maps $i_\pm: D_{n-1,n}^\pm(\alpha) \cap U_S(\alpha) \hookrightarrow U_S(\alpha)$ and the diffeomorphisms  
$s_\pm:  D_{n-1,n}^\pm(\alpha)\cap U_S(\alpha) \longrightarrow  U_S(\alpha^\pm)$ from Proposition~\ref{prop:map}, we have
\begin{align*}
& (i_\pm \circ s_\pm^{-1})^* (i_S^*\, c_i)   =   (i^\pm_S)^* \, c_i^\pm \quad \text{for}\quad 1 \leq i\leq n-2; \\
&(i_+ \circ s_+^{-1})^* (i_S^*\, c_{n-1})    =  (i^+_S)^* \, c_{n-1}^+; \\
&(i_- \circ s_-^{-1})^* (i_S^*\, c_{n-1})    =  \sgn (\alpha_{n-1}-\alpha_n) \, (i^-_S)^* \, c_{n-1}^-; \\
&(i_+ \circ s_+^{-1})^* (i_S^* \, c_{n})     = (i^+_S)^* \, c_{n-1}^+ ;   \\
& (i_- \circ s_-^{-1})^* (i_S^*\, c_{n})     =  - \sgn (\alpha_{n-1}-\alpha_n)\,(i^-_S)^* \,  c_{n-1}^-.  
\end{align*}
\end{proposition}
\begin{proof}
Recall the identification of $U_S(\alpha)$ with the moduli space $\mathcal{Z}$ of $(n+1)$-tuples of vectors 
$$
\big\{ u_l, v_k, w \in \mathbb{R}^3,\, l\in S,\, k \in S^c\big\}
$$ 
taken up to rotation, satisfying $1)-5)$ in Theorem~\ref{hp}. Recall also that $D_{n-1,n}^\pm (\alpha) \cap U_S(\alpha)$ can  be identified via this homeomorphism to the subspace $\mathcal{D}^\pm$ of $\mathcal{Z}$ where $v_{n-1}=\lambda v_n$ with $\lambda \in \mathbb{R}^\pm$, and that this space is, in turn, clearly homeomorphic to $U_S(\alpha^\pm)$. We then have homeomorphisms 
\begin{eqnarray*}
\phi_\alpha^\pm & :&  \mathcal{D}^\pm \to \, D_{n-1,n}^\pm(\alpha) \cap U_S(\alpha)\\
\phi_{\alpha^\pm}& :&\mathcal{D}^\pm \to \, U_S(\alpha^\pm)
\end{eqnarray*} 
and the corresponding pull-back bundles
\begin{eqnarray*}
(\phi_\alpha^\pm)^* \big(\widetilde{V}_j(\alpha)\big) & \to & \mathcal{D}^\pm \\
(\phi_{\alpha^\pm})^* \big(\widetilde{V}_j(\alpha^\pm)\big) & \to & \mathcal{D}^\pm
\end{eqnarray*} 
are topological circle bundles over $\mathcal{D}^\pm$ obtained by rotation of the pairs of polygons  formed by the vectors $u_l,v_k,w$ around the axis defined by the vector $v_j$. 

{\tiny{
\begin{equation}
\xymatrixrowsep{5pc} \xymatrixcolsep{.5pc} \xymatrix{
\widetilde{V}_j (\alpha^\pm)\vert_{U_S(\alpha^\pm)}  \ar[d] & \ar[l] (\phi_{\alpha^\pm})^*\big(\widetilde{V}_j(\alpha^\pm)\vert_{U_S(\alpha^\pm)}\big)  \ar[dr] &  &   (\phi_{\alpha}^\pm)^*\big(\widetilde{V}_j(\alpha)\vert_{D_{n-1,n}^\pm(\alpha) \cap U_S(\alpha)}\big) \ar[r] \ar[dl] & \widetilde{V}_j(\alpha)\vert_{D_{n-1,n}^\pm (\alpha)\cap U_S(\alpha)} \ar[d] \\  
U_S(\alpha^\pm) &  & \ar[ll]^{\phi_{\alpha^\pm}}   \mathcal{D}^\pm   \ar[rr]_{\phi_{\alpha}^\pm} &  &  \ar @/^3pc/[llll]_{s_\pm=\phi_{\alpha^\pm} \circ (\phi_\alpha^\pm)^{-1}} D_{n-1,n}^\pm(\alpha) \cap U_S(\alpha).
}
\end{equation}}}
We would like to compare the classes $(i_\pm \circ s_\pm^{-1})^* (i_S^*\, c_j)$ and $(i_S^\pm)^*\, c_j^\pm$. For that consider the pull back of both classes to $H^2(\mathcal{D}^\pm)$ via $\phi_{\alpha^\pm}$. In particular, one obtains
$$
\phi^*_{\alpha^\pm} \big((i_\pm \circ s_\pm^{-1})^* i_S^* c_j\big)= \phi_{\alpha^\pm}^* \, (s_\pm^{-1})^* \, i_\pm^* \, (i_S^* \,c_j)=(\phi_\alpha^\pm)^* \, i_\pm^* (i_S^* c_j),
$$
which is the first Chern class of the pull-back bundle $(\phi_\alpha^\pm)^* \widetilde{V}_j(\alpha)  \to \mathcal{D}^\pm$, and
$$
\phi_{\alpha^\pm}^* \big((i_S^\pm)^* c_j^\pm\big),
$$
which is the first Chern class of the pull-back bundle $\phi_{\alpha^\pm}^* \widetilde{V}_j(\alpha^\pm) \to \mathcal{D}^\pm$.
These two bundles rotate the pairs of polygons around the axis defined by the edge $v_j(\alpha)$ and $v_j(\alpha^\pm)$ respectively, where    $v_j(\alpha)$ is the vector $v_j$ in $(\phi_{\alpha}^\pm)^* \widetilde{V}_j(\alpha)$ and $v_j(\alpha^\pm)$ is the vector $v_j$ in $\phi_{\alpha^\pm}^* \widetilde{V}_j(\alpha^\pm)$. 

Since, if $j\neq n-1,n$, one has $v_j(\alpha)=v_j(\alpha^\pm)$, one obtains  
$$(i_\pm \circ s_\pm^{-1})^* (i_S^* \, c_i)    =  (i^\pm_S)^* \, c_i^\pm \quad \text{for}\quad 1 \leq i\leq n-2.$$
As $v_{n-1}(\alpha^+)=v_{n-1}(\alpha) + v_n(\alpha)$, the vectors $v_{n-1}(\alpha^+)$, $v_{n-1}(\alpha)$ and $v_n(\alpha)$ determine the same circle action and so
$$
(i_+ \circ s_+^{-1})^* i_S^* \, c_{n-1}    =  (i^+_S)^*\,  c_{n-1}^+ \quad \text{\and} \quad (i_+ \circ s_+^{-1})^* i_S^* \, c_{n}     = (i^+_S)^* \, c_{n-1}^+.
$$
Similarly,  since $v_{n-1}(\alpha^-)=\sgn (\alpha_{n-1}-\alpha_n) (v_{n-1}(\alpha) - v_n(\alpha))$, the vectors $v_{n-1}(\alpha^-)$,  $\sgn (\alpha_{n-1}-\alpha_n) v_{n-1}(\alpha)$ and $-\sgn (\alpha_{n-1}-\alpha_n) v_{n}(\alpha)$  determine the same circle action and so 
\begin{align*}
(i_- \circ s_-^{-1})^* i_S^*\, c_{n-1}    & =  \sgn (\alpha_{n-1}-\alpha_n) \, (i^-_S)^* \, c_{n-1}^- \\
 (i_- \circ s_-^{-1})^* i_S^*\, c_{n}    & =  - \sgn (\alpha_{n-1}-\alpha_n)\,(i^-_S)^*  c_{n-1}^-.
\end{align*}
\end{proof}
Using Propositions~\ref{prop:dij} and \ref{prop:pullbackcj} one obtains the following recursion formula.
\begin{Theorem}\label{thm:recursion}
Suppose $\alpha_{n-1} \neq \alpha_n$ and let 
$$
\alpha^+:=(\alpha_1,\ldots,\alpha_{n-2},\alpha_{n-1} +  \alpha_n) \quad \text{and} \quad \alpha^-:=(\alpha_1,\ldots,\alpha_{n-2},\lvert \alpha_{n-1} - \alpha_n\rvert ).
$$
Then, for $k_1,\ldots, k_n \in \Z_{\geq 0}$ such that $k_1+\cdots + k_n= n-3$ and $k_n\geq 1$,
\begin{equation}
\label{eq:0.1}
\begin{split}
&\dint\limits_{U_S(\alpha)}\!\! i_S^* \left( c_1^{k_1} \cdots c_n^{k_n} \right) = \!\!
\dint\limits_{U_S(\alpha^+)}\!\!
 (i_S^+)^* \left( (c^+_1)^{k_1} \cdots (c^+_{n-2})^{k_{n-2}} (c_{n-1}^+)^{k_{n-1}+k_n - 1}\right) \,+ \\
  \\
&(-1)^{k_n - 1}\! \left( \sgn( \alpha_{n-1}\! -\alpha_n ) \right)^{k_{n-1}+k_n}\!\!\!\!\!\!\!
 \dint\limits_{U_S(\alpha^-)}\!\!\!\!\! (i_S^-)^* \left( (c^-_1)^{k_1}\cdots
(c^-_{n-2})^{k_{n-2}} (c^-_{n-1})^{k_{n-1}+k_n - 1}\right). 
\end{split}
\end{equation}
\end{Theorem}
\begin{proof}
By Proposition~\ref{prop:dij} the Poincar\'{e} dual of $i_S^*\, c_n$ is 
$$
(i_S^+)_*\, [D_{n-1,n}^+ (\alpha)\cap U_S(\alpha)] + \sgn (\alpha_{n-1}- \alpha_n) (i_S^-)_*\,[D_{n-1,n}^-(\alpha) \cap U_S(\alpha)].
$$
This means that the formula
\begin{eqnarray*}
\lefteqn{\int_{U_S(\alpha)} i_S^* (a\, c_n)  = \int_{U_S(\alpha^+)} (i_+\circ s_+^{-1})^*(i_S^*\, a) \quad +} \\ & &  + \quad \sgn (\alpha_{n-1}- \alpha_n) \int_{U_S(\alpha^-)} (i_-\circ s_+^{-1})^*(i_S^*\,a)
\end{eqnarray*}
holds true for all $a\in H^{n-4}(U_S(\alpha))$. The result  then follows from  Proposition~\ref{prop:pullbackcj}.
\end{proof}
\subsection{Explicit formulas}
Using Theorem~\ref{thm:recursion} one can obtain explicit expressions for the computation of intersection numbers. For that we first note the following facts concerning the Chern classes $c_j$.
\begin{claim}\label{cl:1}
If $1\in S$ then $i_S^* c_1 = i_S^* c_j$ for every $j \in S$.
\end{claim}
\begin{proof} By Proposition~\ref{prop:dij} the class $i_S^*(c_1-c_j)$ is represented by
$$
2 \sgn(\alpha_j-\alpha_1) [D_{1,j}^-(\alpha)\cap U_S(\alpha)].
$$
However, in $D_{1,j}^-(\alpha)$, the vectors $(q_1q_1^*-p_1^*p_1)_0$ and $(q_jq_j^*-p_j^*p_j)_0$ in $\mathbb{R}^3$ point in opposite directions and that is impossible in $U_S(\alpha)$ since, by hypothesis, both $j$ and $1$ are in $S$. Indeed, the vectors $q_i$ for $i \in S$ are all proportional, implying that the vectors $(q_iq_i^*)_0$ are positive scalar multiples of each other and, moreover, the moment map condition \eqref{complex1} implies that $(p_1^*p_1)_0$ is a non-positive scalar multiple of $(q_iq_i^*)_0$. Hence, for all $i \in S$, the vectors $(q_1q_1^*-p_1^*p_1)_0$ all point in the same direction and so $i_S^*(c_1-c_j)=0$. 
\end{proof}
\begin{claim}\label{cl:2}
If $1 \in S$ then $i_S^* \, c_j^2= i_S^* \, c_1^2$ for all $j \in S^c$.
\end{claim}
\begin{proof}
Since $i_S^* (c_j^2-c_1^2) = i_S^* ((c_j-c_1)(c_j+c_1))$ and $i_S^* (c_j-c_1)$ is represented by
$$
2 \sgn(\alpha_1-\alpha_j)[D_{1,j}^- (\alpha)\cap U_S(\alpha)],
$$
while $i_S^*(c_j+c_1)$ is represented by
$$
2 [D_{1,j}^+(\alpha) \cap U_S(\alpha)],
$$
the result follows. Here note that in $D_{1,j}^+$ the vectors $(q_1q_1^*)_0$ and $(q_jq_j^*)_0$ (and consequently $(q_1q_1^*-p_1^*p_1)_0$ and $(q_jq_j^*-p_j^*p_j)_0$) point in the same direction while in $D_{1,j}^-$ they point in opposite directions.
\end{proof}
\begin{claim}\label{cl:3}
If $1\in S$ and $\vert S\vert =n-1$ then $i_S^* \, c_j = - i_S^* \, c_1$ for the unique $j \notin S$.
\end{claim}
\begin{proof}
Note that 
$$
i_S^*(c_j + c_1)= 2 PD\left([D_{1,j}^+(\alpha) \cap U_S(\alpha)]\right) = 0,
$$
since it is impossible for $(q_jq_j^*)_0$ to point in the same direction as $(q_1q_1^*)_0$ (the corresponding spatial polygons in $U_S(\alpha)$ would not close).
\end{proof}
Using the first two claims, and reordering $\alpha$ if necessary, one can reduce the computation of all intersection numbers to integrals of one of the two following types, where one assumes without loss of generality, that $S=\{1, \ldots, \vert S\vert\}$:
\begin{align*}
\text{(I)} & \dint_{U_S(\alpha)} i_S^*\,  c_1^{n-3},\\
\text{(II)} & \dint_{U_S(\alpha)} i_S^*\, (c_1^k \, c_{n-l}\cdots c_n), \quad \text{with}\quad n-l > \vert S \vert\quad   \text{and}\quad  k=n-l-4.
\end{align*}
To obtain explicit formulas for these integrals one needs first to consider families of triangular sets as defined in \cite{AG}.
\begin{definition}\label{def:triang}
Let $\alpha=( \alpha_1,\ldots,\alpha_m)$ be generic. A set $J \in I=\{3, \ldots, m\}$ is called \emph{triangular} if 
$$
\ell_J:= \sum_{i \in J} \alpha_i - \sum_{i \in I\setminus J} \alpha_i > 0
$$ 
and satisfies the following triangular inequalities
$$
\alpha_1 \leq \alpha_2 + \ell_J, \quad \alpha_2 \leq \alpha_1 + \ell_J \quad \text{and} \quad \ell_J \leq \alpha_1 +\alpha_2. 
$$
Moreover, define the family of triangular sets in $I$ as
$$
\mathcal{T}(\alpha)=\{J \in I\mid J \,\, \text{is triangular}\}.
$$
\end{definition}

For integrals of type (I) one has the following result.
\begin{Theorem}
\label{thm:1} 
Let $S$ be the short set $\{1, \ldots, \vert S\vert\}$.

If $\vert S \vert \leq n-3$ then
\begin{equation}\label{eq:I1}
\dint_{U_S(\alpha)} i_S^* \, c_1^{n-3} =   
\sum_{J\in \mathcal{T}(\widetilde{\alpha})} (-1)^{\big(n-\lvert S \rvert\big)\big\lvert J \cap \{n-\vert S\vert +1\}\big\rvert + \vert J \vert + \vert S\vert},
\end{equation}
where $\widetilde{\alpha}:=\left(\alpha_n,\alpha_{\vert S \vert + 1}, \ldots, \alpha_{n-1}, \sum_{i \in S} \alpha_i\right)$ and 
$\mathcal{T}(\widetilde{\alpha})$ is the corresponding family of triangular sets.

If $\vert S \vert = n-2$ then 
\begin{equation}\label{eq:I2}
\dint_{U_S(\alpha)} i_S^* \, c_1^{n-3} =  \left\{ \begin{array}{ll} (-1)^{n-1}, & \text{if $S$ is a maximal short set for $\alpha$} \\ \\ 0, &  \text{otherwise}.\end{array} \right.
\end{equation}

If $\vert S \vert =n-1$ then 
\begin{equation}\label{eq:I3}
\dint_{U_S(\alpha)} i_S^* \, c_1^{n-3} = (-1)^{n-1}.
\end{equation}
\end{Theorem}
\begin{proof}

$\bullet$  If $\vert S\vert =n-1$ and assuming that $S=\{1, \ldots, n-1\}$ then, by Claim~\ref{cl:1} and Proposition~\ref{prop:wj},
$$
i_S^* \, c_1^{n-3}=i_S^* \, c_1 \cdots c_{n-3}=(-1)^{n-3} PD \big( [U_S(\alpha) \cap W_1 \cap \cdots \cap W_{n-3}]\big),
$$
where $W_i=\{[p,q]\in X(\alpha)\mid \, p_i=0\}$. Moreover, 
$$U_S(\alpha) \cap W_1 \cap \cdots \cap W_{n-3}$$
can be identified with the moduli space of vectors $u,v,w \in \mathbb{R}^3$ taken up to rotation,  satisfying 
\begin{align*}
\bullet & \quad  w=-v, \\ \bullet &  \quad u\cdot w= 0, \\ \bullet&  \quad \vert \vert w\vert \vert =\vert \vert v\vert \vert = \alpha_n, \\
\bullet &  \quad \sqrt{\alpha_{n-2}^2 + \vert \vert u \vert \vert^2} + \sqrt{\alpha_{n-1}^2 + \vert \vert u \vert \vert^2} = \alpha_n - \alpha_1-\cdots -\alpha_{n-3},
\end{align*}
(cf. Figure~\ref{fig:hip1}-(I)).
\begin{figure}[htbp]
\psfrag{w}{\footnotesize{$w$}}
\psfrag{v}{\footnotesize{$v$}}
\psfrag{u}{\footnotesize{$u_{n-2}=u$}}
\psfrag{-u}{\footnotesize{$-u=u_{n-1}$}}
\psfrag{v-1}{\footnotesize{$v_{n-1}$}}
\psfrag{vn}{\footnotesize{$v_n$}}
\includegraphics[scale=.7]{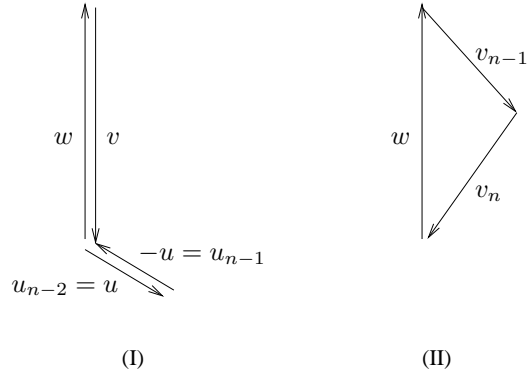}
\caption{(I) The element of $U_S(\alpha) \cap W_1 \cap \cdots \cap W_{n-3}$
represented as a pair of degenerate polygons when $\lvert S \rvert = n-1$.
(II) The element of $U_S(\alpha) \cap W_1 \cap \cdots \cap W_{n-3}$
represented by a spatial polygon, when $\lvert S \rvert = n-2$.}
\label{fig:hip1}
\end{figure}
Since, by hypothesis, $S$ is short we know that $\alpha_n > \sum_{i \in S} \alpha_i$ and so this moduli space is  a point and  
$$
\dint_{U_S(\alpha)} i_S^* \, c_1^{n-3} =  (-1)^{n-1}.
$$

$\bullet$ If $\vert S\vert = n-2$, assuming $S=\{1, \ldots, n-2\}$ and using Claim~\ref{cl:1} and Proposition~\ref{prop:wj}, one has
$$
i_S^* \, c_1^{n-3}=i_S^* \, c_1 \cdots c_{n-3}=(-1)^{n-3} PD \big( [U_S(\alpha) \cap W_1 \cap \cdots \cap W_{n-3}]\big).
$$
In this situation all the vectors $u_i$ as in Theorem~\ref{hp}  are equal to zero since  $\sum_{i \in S} u_i = 0$. Hence, 
$$U_S(\alpha) \cap W_1 \cap \cdots \cap W_{n-3}$$ is now the polygon space 
$$M_S(\alpha):= M\left(\sum_{i \in S} \alpha_i, \alpha_{n-1}, \alpha_n\right)$$
which is a point if simultaneously $\alpha_{n-1} < \sum_{i \neq n-1} \alpha_i$ and $\alpha_{n} < \sum_{i\neq n} \alpha_i$, and empty otherwise (cf. Figure~\ref{fig:hip1}-(II)). The result then follows. (Note that the fact that $S$ is short already implies that $\alpha_{n-1} + \alpha_n < \sum_{ i \in S} \alpha_i$.) 

$\bullet$ If $\vert S \vert \leq n-3$ then, assuming $S=\{1, \ldots, \vert S \vert \}$ and using Claim~\ref{cl:1} and  Proposition~\ref{prop:wj}, one has
$$
i_S^* \, c_1^{\vert S \vert -1} = (-1)^{\vert S \vert -1} PD\big( [ U_S(\alpha) \cap W_1 \cap \cdots \cap W_{\vert S\vert -1}]\big).
$$ 
Again,  in the identification of
$$
U_S(\alpha) \cap W_1 \cap \cdots \cap W_{\vert S\vert -1}
$$
as a moduli space of pairs of polygons in $\R^3$, all the vectors $u_i$  are zero, implying that
$$
U_S(\alpha) \cap W_1 \cap \cdots \cap W_{\vert S\vert -1}= M_S (\alpha)= M\left(\sum_{i \in S} \alpha_i, \alpha_{\vert S \vert +1}, \ldots, \alpha_n\right).
$$ 
Hence,
$$
\dint_{U_S(\alpha)} i_S^* \, c_1^{n-3} = (-1)^{\vert S \vert -1} \dint _{M_S(\alpha)} \tilde{c}_1^{\,n-\vert S\vert -2},
$$
where $\tilde{c}_1:=c_1\big(V_1(\alpha_S)\big)$ for $V_1$ defined in \eqref{eq:vi}, with 
$$\alpha_S=\big(\sum_{i \in S}\alpha_i, \alpha_{\vert S \vert +1}, \ldots, \alpha_n\big).$$ 
Indeed, the circle action on the principal bundle 
$$
\widetilde{V}_1\vert_{U_S(\alpha) \cap W_1 \cap \cdots \cap W_{\vert S\vert -1}} \to U_S(\alpha) \cap W_1 \cap \cdots \cap W_{\vert S\vert -1}
$$
agrees with the one on $V_1\vert_{M_S(\alpha)}$ and so these two bundles are isomorphic.
Reordering the elements in $\alpha_S$ one has
$$
\dint_{M_S(\alpha)} \tilde{c}_1^{\, n-\vert S\vert -2} = \dint_{M(\widetilde{\alpha})} c_{n-\vert S\vert +1}^{n-\vert S \vert - 2},
$$
where $\widetilde{\alpha}:=\big(\alpha_n, \alpha_{\vert S \vert +1},\ldots,\alpha_{n-1}, \sum_{i \in S} \alpha_i\big)$ and $\tilde{c}_{n-\vert S\vert +1}$ is the first Chern class of the circle bundle
$$
V_{n-\vert S\vert +1} \to M(\widetilde{\alpha}).
$$
This new integral can then be computed using Theorem 2 of \cite{AG} for polygon spaces, yielding
$$
\dint_{M_S(\alpha)} \tilde{c}_1^{\,n-\vert S \vert -2} = \sum_{J \in \mathcal{T}(\widetilde{\alpha})} (-1)^{n-\lvert S\rvert + 1 + \lvert J\rvert + \big\lvert (I\setminus J) \cap \{n-\vert S\vert +1\}\big\rvert \big(n- \lvert S \rvert\big)},
$$
where $I=\{ 3, \ldots, n-\vert S\vert +1 \}$ and the result follows.
\end{proof}
For integrals of type (II) we have:
\begin{Theorem} 
\label{thm:2}
Let $S$ be the short set $\{1, \ldots, \lvert S\rvert\}$.

If $\vert S\vert < n-l-2$ then
\begin{align} \label{eq:thm21}
& \dint_{U_S(\alpha)} i_S^* \, (c_1^k \, c_{n-l}\cdots c_n) =  \\ \nonumber  &   \sum_{{\scriptsize{J\in \mathcal{A}_{n,l}(\alpha)}}}\,\,\, \sum_{{\scriptsize{J^\prime \in \mathcal{T}_{n,l}(\alpha,J)}}} (-1)^{\big\lvert J \cap \{n-l-1\}\big\rvert +\big\lvert J^\prime \cap\{n-l-\vert S \vert \}\big\rvert \big(n-l-\vert S \vert + 1\big) + \vert J^\prime \vert + \vert S \vert +1 },
\end{align}
where $\mathcal{A}_{n,l}(\alpha)$ is the family of sets $J\subset I_{n,l}:=\{n-l-1,\ldots,n\}$  for which 
$$
\ell_J(\alpha): = \sum_{i \in J} \alpha_i \,\, - \sum_{i \in I_{n,l}\setminus J} \alpha_i > 0 
$$
and
$$
\sum_{i \in S} \alpha_i < \ell_J(\alpha) + \alpha_{\vert S\vert +1} + \cdots + \alpha_{n-l-2},
$$
and where $\mathcal{T}_{n,l}(\alpha,J):=\mathcal{T}(\widetilde{\alpha}_{n,l,J})$ is the family of triangular sets for 
$$\widetilde{\alpha}_{n,l,J}:=\left(\ell_J(\alpha), \alpha_{\vert S \vert +1}, \cdots, \alpha_{n-l-2}, \sum_{i \in S} \alpha_i\right).$$

If $\vert S\vert = n-l-2$ then
\begin{equation} \label{eq:thm22}
\dint_{U_S(\alpha)} i_S^* \, (c_1^k \, c_{n-l}\cdots c_n) =   \sum_{{\scriptsize{J\in \mathcal{A}_{n,l}(\alpha)}}}\,\,\,  (-1)^{\big\lvert J \cap \{n-l-1\}\big\rvert + \vert S \vert +1 }.
\end{equation}

If $\vert S \vert =n-l-1$ then
\begin{equation}
\label{eq:tildea}
\dint_{U_S(\alpha)} i_S^* (c_1^k\,  c_{n-l}\cdots c_n) = (-1)^{n-l}\big\vert \widetilde{\mathcal{A}}_{n,l}(\alpha) \big\vert,
\end{equation}
where
$$
\widetilde{\mathcal{A}}_{n,l}(\alpha) = \left\{J \subset \{n-l, \ldots, n\}\mid\, \ell_J(\alpha) > \sum_{i \in S} \alpha_i \right\}.
$$
\end{Theorem}
\begin{proof}
We will prove this formula by induction on $n$ starting with $n=k+4$ (implying $l=0$). Here we have to consider two cases ($\vert S \vert = n-1$ and $\vert S \vert < n-1$).

First, if $\vert S \vert = n-1 = k+3$  we have by Claim~\ref{cl:3} and Theorem~\ref{thm:1} \eqref{eq:I3}, that
$$
\dint_{U_S(\alpha)} i_S^* \, c_1^{n-4} \, c_{n} = - \dint_{U_S(\alpha)} i_S^* \, c_1^{n-3}  = (-1)^n,
$$
which is equal to the right hand side of \eqref{eq:tildea} since, in this case, 
$$
\widetilde{\mathcal{A}}_{n,0}(\alpha) = \big\{  \{ n \} \big\} .
$$

If $\vert S \vert < n-1 =k+3$ then by the recursion formula \eqref{eq:0.1} we have
\begin{equation}
\label{eq:ssmaller}
\dint_{U_S(\alpha)} i_S^*\, c_1^{n-4} \,c_{n} = \dint_{U_S(\alpha^+)} (i_S^+)^* \, c_1^{n-4} + \sgn (\alpha_{n-1}-\alpha_n) \dint_{U_S(\alpha^-)} (i_S^-)^* \, c_1^{n-4}. 
\end{equation}
with $\alpha^\pm=\big(\alpha_1, \ldots, \alpha_{n-2},\vert \alpha_{n-1} \pm \alpha_n\vert\big)$.

$\bullet$ If, in particular, $\vert S\vert =n-2=k+2$ then by Theorem~\ref{thm:1}-\eqref{eq:I3},
\begin{equation}\label{eq:integral21}
\dint_{U_S(\alpha)}\!\! i_S^* c_1^{n-4} \, c_{n} = \left\{ \begin{array}{ll} \!\!(-1)^n\big(1+\sgn (\alpha_{n-1} - \alpha_n)\big), &\!\!\text{if}\, \displaystyle\sum_{i \in S} \alpha_i < \vert \alpha_{n-1} - \alpha_n\vert \\ \\\!\!(-1)^n, & \!\! \text{otherwise}.  \end{array}\right.
\end{equation}
Note that $S$ is always short for $\alpha^+$ since, by assumption $S$ is short for $\alpha$ and that $S$ is short for $\alpha^-$ if and only if $\sum_{i \in S} \alpha_i < \vert \alpha_{n-1} - \alpha_n\vert$. On the other hand, in this case we have 
$$
\mathcal{A}_{n,0}(\alpha)=\left\{ \begin{array}{ll}\!\!\big\{ \{n-1\}, \{n-1,n\} \big\} \,\text{or} \, \big\{ \{n\}, \{n-1,n\} \big\}, &\!\! \text{if $S$ is $\alpha^-$-short}  \\ \\
\!\! \big\{ \{n-1,n\} \big\}, & \!\! \text{otherwise}. \end{array} \right.
$$
Then the right-hand-side of \eqref{eq:thm22} agrees with the result obtained in \eqref{eq:integral21}.

$\bullet$ If $\vert S\vert =n-3=k+1$ then again by Theorem~\ref{thm:1}-\eqref{eq:I2},
$$
\dint_{U_S(\alpha)} i_S^* \, c_1^{n-4} \, c_{n} 
$$
is equal to 
\begin{enumerate}
\item[(i)] $(-1)^n((1+ \sgn(\alpha_{n-1}-\alpha_n) )$, if $S$ is $\alpha^\pm$-maximal short, in which case
$$
\mathcal{A}_{n,0}(\alpha)=\big\{ \{m-1\}, \{m-1,m\} \big\} \quad \text{or} \quad \big\{ \{m\}, \{m-1,m\} \big\}
$$
and 
$$
\mathcal{T}_{n,0}\big(\alpha,\{ n-1,n\}\big)=\mathcal{T}_{n,0}\big(\alpha,\{n-1\}\big)=\mathcal{T}_{n,0}\big(\alpha,\{n\}\big)=\big\{ \{3\} \big\}.
$$
\item[(ii)] $(-1)^n$, if $S$ is $\alpha^+$-maximal short and either not $\alpha^-$-maximal short or not $\alpha^-$-short at all, in which cases
$$
\mathcal{A}_{n,0}(\alpha)=\big\{\{n-1,n\} \big\}, \,\,  \big\{ \{n-1\}, \{n-1,n\} \big\} \,\, \text{or} \,\, \big\{ \{n\}, \{n-1,n\} \big\}
$$
and 
$$
\mathcal{T}_{n,0}\big(\alpha,\{n-1,n\}\big)=\big\{ \{3\} \big\},\,\, \mathcal{T}_{n,0}\big(\alpha,\{n-1\}\big)=\mathcal{T}_{n,0}\big(\alpha,\{n\}\big)= \varnothing.
$$ 
\item[(iii)] $(-1)^n \sgn(\alpha_{n-1}-\alpha_n)$, if $S$ is not $\alpha^+$-maximal short but $\alpha^-$-maximal short, in which case
$$
\mathcal{A}_{n,0}(\alpha)=\big\{ \{n-1\}, \{n-1,n\} \big\} \,\, \text{or} \,\, \big\{ \{n\}, \{n-1,n\} \big\},
$$
and 
$$
\mathcal{T}_{n,0}\big(\alpha,\{n-1,n\}\big)=\varnothing,\,\, \mathcal{T}_{n,0}\big(\alpha,\{n-1\}\big)=\mathcal{T}_{n,0}\big(\alpha,\{n\}\big)= \big\{ \{3\} \big\}.
$$  
\item[(iv)] $0$, if $S$ if not $\alpha^+$-maximal short and either not $\alpha^-$-maximal short or not $\alpha^-$-short at all, in which cases
$$
\mathcal{A}_{n,0}(\alpha)=\big\{ \{n-1,n\} \big\}, \,\, \big\{ \{n-1\}, \{n-1,n\} \big\} \,\, \text{or} \,\, \big\{ \{n\}, \{n-1,n\} \big\}
$$
and 
$$
\mathcal{T}_{n,0}\big(\alpha,\{n-1,n\}\big)=\mathcal{T}_{n,0}\big(\alpha,\{n-1\}\big)=\mathcal{T}_{n,0}\big(\alpha,\{n\}\big)= \varnothing.
$$ 
\end{enumerate}
It is now easy to verify that the above results (i)-(iv) agree in all cases with the right hand side of \eqref{eq:thm21}.

$\bullet$ Finally, if $\vert S \vert < n-3=k+1$ then by \eqref{eq:ssmaller} and Theorem~\ref{thm:1}-\eqref{eq:I1}, considering $\mathcal{T}(\widetilde{\alpha}^\pm)$ the family of triangular sets $J \subset \{3, \ldots, n-\vert S\vert\}$ for 
$$
\widetilde{\alpha}^\pm := \left(\vert \alpha_{n-1}\pm \alpha_n\vert , \alpha_{\vert S \vert +1}, \ldots, , \alpha_{n-2}, \sum_{i \in S} \alpha_i\right),
$$
we have
\begin{eqnarray}\label{eq:integral2}
  \lefteqn{\dint_{U_S(\alpha)} i_S^* \, c_1^{n-4} \,  c_{n}=  \sum_{J^\prime \in \mathcal{T}(\widetilde{\alpha}^+)} (-1)^{\big(n-1-\vert S \vert\big)\big\vert J^\prime \cap \{n-\vert S\vert \}\big\vert +\vert J^\prime \vert + \vert S \vert} \,\,\, +} \\ \nonumber & & + \sgn(\alpha_{n-1}-\alpha_n) \sum_{J^\prime \in \mathcal{T}(\widetilde{\alpha}^-)} (-1)^{\big(n-1-\vert S \vert\big)\big\vert J^\prime\cap \{n-\vert S \vert \}\big\vert +\vert J^\prime \vert + \vert S \vert}, 
\end{eqnarray}
if $S$ is short for $\widetilde{\alpha}^-$,  and 
\begin{equation}\label{eq:integral3}
\dint_{U_S(\alpha)} i_S^* c_1^{n-4}\,  c_{n}= \sum_{J^\prime \in \mathcal{T}(\widetilde{\alpha}^+)} (-1)^{\big(n-1-\vert S\vert\big)\big\vert J^\prime\cap \{n-\vert S\vert \}\big\vert +\vert J^\prime \vert + \vert S \vert},
\end{equation}
otherwise.
In the first situation, we have 
$$\mathcal{A}_{n,0}(\alpha)=\big\{\{n-1\}, \{n-1,n\}\big\} \quad \text{or}\quad \big\{\{n\}, \{n-1,n\} \big\} $$ 
and
$$
\mathcal{T}_{n,0}\big(\alpha,\{n-1,n\}\big) =\mathcal{T}(\widetilde{\alpha}^+), \quad \mathcal{T}_{n,0}\big(\alpha,\{n-1\}\big) =\mathcal{T}_{n,0}\big(\alpha,\{n\}\big) = \mathcal{T}(\widetilde{\alpha}^-),
$$
while, in the second one, we have
$$
\mathcal{A}_{n,0}(\alpha)=\big\{ \{n-1,n\} \big\} \quad \text{and} \quad  \mathcal{T}_{n,0}\big(\alpha,\{n-1,n\}\big) =\mathcal{T}(\widetilde{\alpha}^+),
$$
and so \eqref{eq:integral2} and \eqref{eq:integral3} agree with the right hand side of \eqref{eq:thm21}.

We will now assume that  \eqref{eq:thm21}, \eqref{eq:thm22} and \eqref{eq:tildea} hold for $n$ and show that they are still true for $n+1$. Using the recursion formula \eqref{eq:0.1} we get
\begin{align*}
  \dint \limits_{U_S(\alpha)}  i_S^* \, c_{1}^{k}\, c_{n+1-l} \cdots c_{n+1}  & =  \int \limits_{U_S(\alpha^+)}  (i_S^+)^* \, (c_{1}^+)^{k} \, c_{n+1-l}^+ \cdots c_{n}^+  \,\,\, + \\ & +  \int \limits_{U_S(\alpha^-)}   (i_S^-)^*\, (c_{1}^-)^{k}\, c_{n+1-l}^- \cdots c_{n}^-. 
\end{align*}

$\bullet$ If $\vert S \vert < n-l-2$ then, if $\alpha_{n}-\alpha_{n+1} > 0$,
{\small{\begin{align} \label{eq:x}
& \mathcal{A}_{n+1,l}(\alpha) =   \Big\{ J \subset I_{n+1,l}:= \{ n-l, \ldots, n+1\} \mid \, \ell_J(\alpha) >0 \,\, \text{and}  \\ \nonumber & \quad \quad \quad \quad  \quad \quad \quad \quad \sum_{i \in S} \alpha_i < \ell_J(\alpha) + \alpha_{\vert S\vert +1} + \cdots + \alpha_{n-l-1} \Big\}  \\\nonumber  = 
&   \left\{\widetilde{J} \in \mathcal{A}_{n,l-1}(\alpha^+) \mid n \notin \widetilde{J} \right\} \,  \bigcup \, \left\{\widetilde{J} \cup \{n+1\}\mid \widetilde{J} \in \mathcal{A}_{n,l-1}(\alpha^+) \, \text{and} \, n \in \widetilde{J} \right\} \, \bigcup \\ \nonumber & \bigcup \,  \left\{\widetilde{J} \in \mathcal{A}_{n,l-1}(\alpha^-) \mid n \in \widetilde{J} \right\} \, \bigcup \,   \left\{\widetilde{J}\cup \{n+1\} \mid \widetilde{J} \in \mathcal{A}_{n,l-1}(\alpha^-) \, \text{and} \, n \notin \widetilde{J} \right\},
\end{align}}}
while, if $\alpha_{n}-\alpha_{n+1} < 0$,
{\small{\begin{align}\label{eq:y}
\mathcal{A}_{n+1,l}(\alpha) =  & \left\{\widetilde{J} \in \mathcal{A}_{n,l-1}(\alpha^+) \mid n \notin \widetilde{J} \right\} \,  \bigcup \\ \nonumber &  \bigcup \, \left\{\widetilde{J} \cup \{n+1\}\mid \widetilde{J} \in \mathcal{A}_{n,l-1}(\alpha^+) \, \text{and} \, n \in \widetilde{J} \right\} \, \bigcup \\ \nonumber &  \bigcup  \left\{\left(\widetilde{J}\setminus\{n\}\right)\cup\{n+1\}\mid \widetilde{J} \in \mathcal{A}_{n,l-1}(\alpha^-) \, \text{and} \, n \in \widetilde{J} \right\} \, \bigcup \\\nonumber  &  \bigcup \,    \,   \left\{\widetilde{J}\cup \{n\} \mid \widetilde{J} \in \mathcal{A}_{n,l-1}(\alpha^-) \, \text{and} \, n \notin \widetilde{J} \right\}.
\end{align}}}
Moreover, since  $\mathcal{T}_{n,l-1}(\alpha^\pm,J):=\mathcal{T}(\widetilde{\alpha}_{n,l-1,J}^\pm)$ is the family of triangular sets 
$$J^\prime\subset\big\{3, \ldots, n-(l-1)-\vert S\vert\big\}$$ 
for  $\widetilde{\alpha}_{n,l-1,J}^\pm:=\big(\ell_J(\alpha^\pm), \alpha_{\vert S \vert +1}, \cdots, \alpha_{n-(l+1)}, \sum_{i \in S} \alpha_i\big)$,
and $\mathcal{T}_{n+1,l}(\alpha,J):=\mathcal{T}(\widetilde{\alpha}_{n+1,l,J})$ is the family of triangular sets 
$$J^\prime\subset \big\{3, \ldots, (n+1)-l-\vert S\vert\big\}$$ 
for 
$\widetilde{\alpha}_{n+1,l,J}:=\big(\ell_J(\alpha), \alpha_{\vert S \vert +1}, \cdots, \alpha_{(n-1)-l}, \sum_{i \in S} \alpha_i\big)$,
we have that
$$
\mathcal{T}_{n,l-1}(\alpha^+,J) =\left\{ \begin{array}{l}\mathcal{T}_{n+1,l}\big(\alpha, J \cup \{n+1\}\big),\,\, \text{if}\,\, n \in J \\ \\ \mathcal{T}_{n+1,l}(\alpha, J), \,\, \text{if}\,\, n \notin J, \end{array} \right.
$$
and, if $\alpha_{n}-\alpha_{n+1} > 0$,
$$
\mathcal{T}_{n,l-1}(\alpha^-,J) =\left\{ \begin{array}{l}\mathcal{T}_{n+1,l}\big(\alpha, J \cup \{n+1\}\big),\,\, \text{if}\,\, n \notin J \\ \\ \mathcal{T}_{n+1,l}(\alpha, J), \,\, \text{if}\,\, n \in J, \end{array} \right.
$$
while, if $\alpha_{n}-\alpha_{n+1} < 0$,
$$
\mathcal{T}_{n,l-1}(\alpha^-,J) =\left\{ \begin{array}{l}\mathcal{T}_{n+1,l}\big(\alpha, J \cup \{n\}\big),\,\, \text{if}\,\, n \notin J \\ \\ \mathcal{T}_{n+1,l}\big(\alpha, (J\setminus \{n\})\cup \{n+1\} \big), \,\, \text{if}\,\, n \in J. \end{array} \right.
$$
Assuming  that \eqref{eq:thm21} holds for $n$ we have
{\scriptsize{\begin{align*}
& \dint \limits_{U_S(\alpha^+)}    (i_S^+)^* (c_{1}^+)^{k} c_{n+1-l}^+ \cdots   c_{n}^+  \\ & = \hspace{-.3cm}
\sum_{\begin{array}{c} J \,\, \text{in} \\  \mathcal{A}_{n,l-1}(\alpha^+) \end{array}}\hspace{-1.5cm} \sum_{\hspace{1.3cm}\begin{array}{l} J^\prime \,\,\text{in}\\ \mathcal{T}_{n,l-1}\big(\alpha^+,J\big)\end{array}}\hspace{-1cm} (-1)^{\big\vert J \cap \{n-(l-1)-1\}\big\vert + \big\vert J^\prime \cap \{n-(l-1)-\vert S\vert \} \big\vert \big(n-(l-1)+1-\vert S\vert\big) + \vert J^\prime \vert +\vert S \vert +1}\\ & =\hspace{-.3cm}\sum_{\begin{array}{c} \widetilde{J} \in \mathcal{A}_{n,l-1}(\alpha^+)\\ \,\text{s.t.} n\notin \widetilde{J}\end{array}}\hspace{-1.5cm} \sum_{\hspace{1.3cm}\begin{array}{l} J^\prime \,\, \text{in}\\ \mathcal{T}_{n+1,l}\big(\alpha^+,\widetilde{J}\big)\end{array}}\hspace{-1.5cm} (-1)^{\big\vert \widetilde{J} \cap \{(n+1)-l-1\}\big\vert + \big\vert J^\prime \cap \{(n+1)-l -\vert S\vert \} \big\vert \big((n+1)-l+1-\vert S\vert\big) + \vert J^\prime \vert +\vert S \vert +1} \,\, + \\ & +\hspace{-.3cm} \sum_{\begin{array}{c}\widetilde{J} \in \mathcal{A}_{n,l-1}(\alpha^+)\\ \,\text{s.t.}\, n\in \widetilde{J}\end{array}}\hspace{-2.8cm} \sum_{\hspace{2.5cm}\begin{array}{l} J^\prime \,\,\text{in}\\ \mathcal{T}_{n+1,l}\big(\alpha^+,\widetilde{J}\cup\{n+1\}\big)\end{array}}\hspace{-2.8cm}  (-1)^{\big\vert (\widetilde{J} \cup \{n+1\}) \cap \{(n+1)-l-1\}\big\vert + \big\vert J^\prime \cap \{(n+1)-l -\vert S\vert \} \big\vert \big((n+1)-l+1-\vert S\vert\big) + \vert J^\prime \vert +\vert S \vert +1}.
\end{align*}}}
On the other hand, if $\alpha_n - \alpha_{n+1}>0$,
{\scriptsize{\begin{align*}
& \dint \limits_{U_S(\alpha^-)} (i_S^-)^* (c_{1}^-)^{k} c_{n+1-l}^- \cdots   c_{n}^-  \\ & =  \hspace{-.3cm}
\sum_{\begin{array}{c} J \,\,\text{in}\\ \mathcal{A}_{n,l-1}(\alpha^-)\end{array}} \hspace{-1.5cm} \sum_{\hspace{1.3cm}\begin{array}{l} J^\prime \,\,\text{in}\\ \mathcal{T}_{n,l-1}\big(\alpha^-,J\big)\end{array}}\hspace{-1cm} (-1)^{\big\vert J \cap \{n-(l-1)-1\}\big\vert + \big\vert J^\prime \cap \{n-(l-1)-\vert S\vert \} \big\vert \big(n-(l-1)+1-\vert S\vert\big) + \vert J^\prime \vert +\vert S \vert +1}\\ & =\hspace{-.3cm} \sum_{\begin{array}{c}\widetilde{J} \in \mathcal{A}_{n,l-1}(\alpha^-)\\ \,\text{s.t.} n\in \widetilde{J}\end{array}} \hspace{-1.5cm} \sum_{\hspace{1.3cm}\begin{array}{l} J^\prime\,\, \text{in} \\ \mathcal{T}_{n+1,l}\big(\alpha^+,\widetilde{J}\big)\end{array}}\hspace{-1.5cm} (-1)^{\big\vert \widetilde{J} \cap \{(n+1)-l-1\}\big\vert + \big\vert J^\prime \cap \{(n+1)-l -\vert S\vert \} \big\vert \big((n+1)-l+1-\vert S\vert\big) + \vert J^\prime \vert +\vert S \vert +1} \,\, + \\ & +\hspace{-.3cm} \sum_{\begin{array}{c}\widetilde{J} \in \mathcal{A}_{n,l-1}(\alpha^+)\\ \,\text{s.t.} n\notin \widetilde{J}\end{array}}\hspace{-2.8cm} \sum_{\hspace{2.5cm}\begin{array}{l} J^\prime \,\,\text{in}\\ \mathcal{T}_{n+1,l}\big(\alpha^+,\widetilde{J}\cup\{n+1\}\big)\end{array}}\hspace{-2.8cm}  (-1)^{\big\vert (\widetilde{J} \cup \{n+1\}) \cap \{(n+1)-l-1\}\big\vert + \big\vert J^\prime \cap \{(n+1)-l -\vert S\vert \} \big\vert \big((n+1)-l+1-\vert S\vert\big) + \vert J^\prime \vert +\vert S \vert +1}
\end{align*}}}
and similarly for $\alpha_n - \alpha_{n+1}<0$. The result now follows from \eqref{eq:x} and \eqref{eq:y}.

$\bullet$ If $\vert S \vert = n-l-2$ the family of sets $\mathcal{A}_{n+1,l}(\alpha)$ is the same as in \eqref{eq:x} and \eqref{eq:y}. Moreover, assuming  \eqref{eq:thm22} holds for $n$,we have
\begin{align*}
& \dint \limits_{U_S(\alpha^+)}    (i_S^+)^* (c_{1}^+)^{k} c_{n+1-l}^+ \cdots   c_{n}^+  = 
\sum_{J \in \mathcal{A}_{n,l-1}(\alpha^+)} (-1)^{\big\vert J \cap \{n-(l-1)-1\}\big\vert  +\vert S \vert +1}\\ & = \hspace{-1.8cm}\sum_{\hspace{1.5cm}\begin{array}{l} \widetilde{J} \in \mathcal{A}_{n,l-1}(\alpha^+)\\ \,\text{s.t.} n\notin \widetilde{J}\end{array}} \hspace{-2cm} (-1)^{\big\vert \widetilde{J} \cap \{(n+1)-l-1\}\big\vert +\vert S \vert +1} + \hspace{-1.8cm} \sum_{\hspace{1.5cm}\begin{array}{l}\widetilde{J} \in \mathcal{A}_{n,l-1}(\alpha^+)\\ \,\text{s.t.}\, n\in \widetilde{J}\end{array}} \hspace{-2cm} (-1)^{\big\vert (\widetilde{J} \cup \{n+1\}) \cap \{(n+1)-l-1\}\big\vert +\vert S \vert +1}.
\end{align*}
On the other hand, if $\alpha_n - \alpha_{n+1}>0$,
\begin{align*}
& \dint \limits_{U_S(\alpha^-)} (i_S^-)^* (c_{1}^-)^{k} c_{n+1-l}^- \cdots   c_{n}^-  = 
\sum_{J \in \mathcal{A}_{n,l-1}(\alpha^-)}  (-1)^{\big\vert J \cap \{n-(l-1)-1\}\big\vert +\vert S \vert +1}\\ & =\hspace{-1.8cm}\sum_{\hspace{1.5cm}\begin{array}{l}\widetilde{J} \in \mathcal{A}_{n,l-1}(\alpha^-)\\ \,\text{s.t.} n\in \widetilde{J}\end{array}}\hspace{-2cm} (-1)^{\big\vert \widetilde{J} \cap \{(n+1)-l-1\}\big\vert + \vert S \vert +1}  +\hspace{-1.8cm} \sum_{\hspace{1.5cm}\begin{array}{l}\widetilde{J} \in \mathcal{A}_{n,l-1}(\alpha^+)\\ \,\text{s.t.} n\notin \widetilde{J}\end{array}}  \hspace{-2cm} (-1)^{\big\vert (\widetilde{J} \cup \{n+1\}) \cap \{(n+1)-l-1\}\big\vert +\vert S \vert +1}
\end{align*}
and similarly for $\alpha_n - \alpha_{n+1}<0$. The result then follows from \eqref{eq:x} and \eqref{eq:y}.

$\bullet$ If $\vert S \vert = n-l-1$ then, writing
$$
\widetilde{\mathcal{A}}_{n,l-1}(\alpha^\pm) =  \left\{ J \subset I_{n,l-1}:= \{ n-(l-1), \ldots, n \}\mid \, \ell_J(\alpha^\pm) > \sum_{i \in S} \alpha_i  \right\},
$$
we have for $\alpha_{n}-\alpha_{n+1} > 0$ that
\begin{align*}
 \widetilde{\mathcal{A}}_{n+1,l}(\alpha)  = & \left\{\widetilde{J} \in  \widetilde{\mathcal{A}}_{n,l-1}(\alpha^+) \mid n \notin \widetilde{J} \right\} \,  \bigcup  \\ & \bigcup \, \left\{\widetilde{J} \cup \{n+1\}\mid \widetilde{J} \in  \widetilde{\mathcal{A}}_{n,l-1}(\alpha^+) \, \text{and} \, n \in \widetilde{J} \right\} \, \bigcup \\ & \bigcup \,  \left\{\widetilde{J} \in  \widetilde{\mathcal{A}}_{n,l-1}(\alpha^-) \mid n \in \widetilde{J} \right\} \, \bigcup \, \\ & \bigcup   \left\{\widetilde{J}\cup \{n+1\} \mid \widetilde{J} \in  \widetilde{\mathcal{A}}_{n,l-1}(\alpha^-) \, \text{and} \, n \notin \widetilde{J} \right\},
\end{align*}
while, for  $\alpha_{n}-\alpha_{n+1} < 0$ we have
\begin{align*}
\widetilde{\mathcal{A}}_{n+1,l}(\alpha)  = &  \left\{\widetilde{J} \in  \widetilde{\mathcal{A}}_{n,l-1}(\alpha^+) \mid n \notin \widetilde{J} \right\} \,  \bigcup \\ & \bigcup \, \left\{\widetilde{J} \cup \{n+1\}\mid \widetilde{J} \in  \widetilde{\mathcal{A}}_{n,l-1}(\alpha^+) \, \text{and} \, n \in \widetilde{J} \right\} \, \bigcup \\ & \bigcup \,  \left\{\widetilde{J}\cup \{n\} \mid \widetilde{J} \in  \widetilde{\mathcal{A}}_{n,l-1}(\alpha^-) \, \text{and}\, n \notin \widetilde{J} \right\} \, \bigcup \\ & \bigcup \,   \left\{(\widetilde{J}\setminus\{n\}) \cup \{n+1\} \mid \widetilde{J} \in  \widetilde{\mathcal{A}}_{n,l-1}(\alpha^-) \, \text{and} \, n \in \widetilde{J} \right\}.
\end{align*}
Then
\begin{align*}
\dint \limits_{U_S(\alpha^+)}   &  (i_S^+)^* (c_{1}^+)^{k} c_{n+1-l}^+ \cdots   c_{n}^+  + \dint \limits_{U_S(\alpha^-)} (i_S^-)^* (c_{1}^-)^{k} c_{n+1-l}^- \cdots   c_{n}^- \\ = & (-1)^{n-(l-1)}\Big\{\big\lvert \widetilde{\mathcal{A}}_{n,l-1}(\alpha^+)  \big\rvert + \big\lvert \widetilde{\mathcal{A}}_{n,l-1}(\alpha^-)  \big\rvert \Big\} = (-1)^{(n+1)-l}\big\lvert \widetilde{\mathcal{A}}_{n+1,l}(\alpha) \big\rvert
\end{align*}
and the result follows.

In the above proof one has to assume that each time that the recursion formula is used one has $\alpha_n\neq \alpha_{n+1}$. However, this result is still valid even if this is not the case, as long as $\alpha$ is generic.  In fact, for a generic  $\alpha$ with $\alpha_{n} = \alpha_{n+1}$ we may  take a small value of $\varepsilon > 0$ for which
$U_S(\alpha)$ is diffeomorphic to $U_S(\alpha_\varepsilon)$ with $\alpha_\varepsilon:=(\alpha_1,\ldots,\alpha_{n-1}, \alpha_{n} + \varepsilon)$. For $\varepsilon$ small enough,  $\mathcal{A}_{n,l-1}(\alpha^\pm)=\mathcal{A}_{n,l-1}(\alpha_\varepsilon^\pm)$ and $\mathcal{T}_{n,l-1}(\alpha^\pm,J)=\mathcal{T}_{n,l-1}(\alpha_\varepsilon^\pm,J)$ (since $\alpha$ generic implies that $\alpha_\varepsilon$, $\alpha^+_\varepsilon$ and $\alpha^-_\varepsilon$ are also generic) and so the induction step still holds. 
\end{proof}

\subsection{Examples}
\begin{example}
Let $\alpha=(1,1,3,3,3)$ and consider the space $X(\alpha)$ and the short set $S=\{1,2\}$. The fixed point set of the core component $U_S(\alpha)$ consists of the minimum component $M_S(\alpha) \cong \C \P^1$ and four isolated fixed points. From Claims~\ref{cl:1} and \ref{cl:2} one has
\begin{align*}
\dint\limits_{U_S(\alpha)}  i_S^* \,c_1^2  = \hspace{-.3cm} \dint\limits_{U_S(\alpha)} i_S^* \,c_2^2 =\hspace{-.3cm} \dint\limits_{U_S(\alpha)} i_S^* \,c_3^2=\hspace{-.3cm} \dint\limits_{U_S(\alpha)} i_S^*\, c_4^2=\hspace{-.3cm}\dint\limits_{U_S(\alpha)} i_S^*\, c_5^2=\hspace{-.3cm}\dint\limits_{U_S(\alpha)} i_S^*\, (c_1c_2).
\end{align*}
Using the fact that
$$
i_S^* c_1= -PD\big(U_S(\alpha) \cap W_1\big)=PD\big(M_S(\alpha)\big)=PD\big(M(2,3,3,3)\big)
$$
with $W_1=:\{[p,q]\in X(\alpha)\mid p_1=0\}$ (cf. Proposition~\ref{prop:wj}) and the
recursion formula for polygon spaces in \cite{AG}, one can compute
\begin{align*}
\dint\limits_{U_S(\alpha)}  i_S^* \,c_1^2 = & - \dint\limits_{M_S(\alpha)} \tilde{c}_1 = - \dint\limits_{M(2,3,3,3)} \tilde{c}_1 = - \dint\limits_{M(3,3,3,2)} \tilde{c}_4  \\ = &  - \dint\limits_{M(3,3,5)} 1 - \dint\limits_{M(3,3,1)} 1 =-2,
\end{align*}
where, as usual, given a polygon space $M(\lambda)$ one defines $\tilde{c}_j:=c_1(V_j(\lambda))$. Note that the polygon spaces $M(3,3,5)$ and $M(3,3,1)$ consist of only one point as in Figure~\ref{fig:hip1}-(II). 

If one uses Theorem~\ref{thm:1} to compute these integrals one obtains
$$
\dint\limits_{U_S(\alpha)}  i_S^* \,c_1^2 = \sum_{J \in \mathcal{T}(\widetilde{\alpha})} (-1)^{3\big \lvert J \cap \{4\} \big\rvert + \lvert J\rvert + 2}=-2,
$$
where $\widetilde{\alpha}=(3,3,2)$, since
$$
\mathcal{T}(\widetilde{\alpha})=\Big\{J \subset\{3,4\}\mid \sum_{j \in J} \widetilde{\alpha}_j - \sum_{j \in \{3,4\}\setminus J} \widetilde{\alpha}_j >0 \Big\} = \big\{ \{3\}, \{3,4\} \big\}.
$$

Similarly,
\begin{align*}
\dint\limits_{U_S(\alpha)}  i_S^* \,(c_1 c_5)  & = \dint\limits_{U_S(\alpha)} i_S^* \,(c_1 c_3) =\dint\limits_{U_S(\alpha)} i_S^* \,(c_1 c_4)=\dint\limits_{U_S(\alpha)} i_S^*\, (c_2 c_3) \\ & =\dint\limits_{U_S(\alpha)} i_S^*\, (c_2 c_4)=\dint\limits_{U_S(\alpha)} i_S^*\, (c_2 c_5).
\end{align*}
These integrals can be computed using the recursion formula \eqref{eq:0.1} as follows:
\begin{align*}
\dint\limits_{U_S(\alpha)}  i_S^* \,(c_1 c_5)  & = \dint\limits_{U_S(1,1,3,3,3+\varepsilon)}  i_S^* \,(c_1 c_5) = \dint\limits_{U_S(1,1,3,6+\varepsilon)} i_S^* \,c_1 - \dint\limits_{U_S(1,1,3,\varepsilon)} i_S^* \,c_1\\ &  = -  \dint\limits_{M(2,3,6+\varepsilon)} 1 +  \dint\limits_{M(2,3,\varepsilon)} 1 = 0
\end{align*}
since $i_S^* c_1=-PD\big(U_S(\alpha^\pm_\varepsilon) \cap W_1\big)$.
Note that 
$$M(2,3,6+\varepsilon)=M(2,3,\varepsilon)=\varnothing$$ 
as the polygons in these spaces would not close.

If one uses Theorem~\ref{thm:2}-\eqref{eq:thm21} to compute these integrals one obtains
$$
\dint\limits_{U_S(\alpha_\varepsilon)}  i_S^* \,c_1c_5 = \sum_{J\in \mathcal{A}_{5,0} (\alpha_{\varepsilon})} \sum_{J^\prime \in \mathcal{T}_{5,0}(\alpha_\varepsilon,J)} (-1)^{\big\lvert J \cap \{4\} \big\rvert + 4 \big\lvert J^\prime \cap\{3\} \big\rvert + \lvert J^\prime \rvert + 3}=0
$$
since
$$
\mathcal{A}_{5,0} (\alpha_{\varepsilon})=\big\{J \subset\{4,5\}\mid \ell_J(\alpha_\varepsilon)>0 \,\, \text{and}\,\, 2 < \ell_J(\alpha_\varepsilon)+3 \big\} = \big\{ \{5\}, \{4,5\}\big\},
$$
$$
\mathcal{T}_{5,0}\big(\alpha_\varepsilon,\{5\}\big)=\mathcal{T}_{5,0}(\varepsilon,3,2)=\varnothing
$$
and 
$$
\mathcal{T}_{5,0}\big(\alpha_\varepsilon,\{4,5\}\big)=\mathcal{T}_{5,0}(6+\varepsilon,3,2)=\varnothing.
$$

Finally, 
\begin{align*}
\dint\limits_{U_S(\alpha)}  i_S^* \,(c_4 c_5)  = \dint\limits_{U_S(\alpha)} i_S^* (c_3 c_5) =\dint\limits_{U_S(\alpha)} i_S^* (c_3 c_4) 
\end{align*}
and, by the recursion formula \eqref{eq:0.1}, one has
\begin{align*}
\dint\limits_{U_S(\alpha)}  i_S^* \,(c_4 c_5)  & = \dint\limits_{U_S(1,1,3,6+\varepsilon)} c_4 +  \dint\limits_{U_S(1,1,3,\varepsilon)} c_4 \\ & = \dint\limits_{U_S(1,1,9+\varepsilon)}\hspace{-.3cm} 1 - \dint\limits_{U_S(1,1,3+\varepsilon)}\hspace{-.3cm} 1 + \dint\limits_{U_S(1,1,3+\varepsilon)} \hspace{-.3cm} 1 + \dint\limits_{U_S(1,1,3-\varepsilon)}\hspace{-.3cm} 1 = 2.
\end{align*}
Note that the core components $U_S(1,1,9+\varepsilon)$, $U_S(1,1,3+\varepsilon)$ and $U_S(1,1,3-\varepsilon)$ consist of a single point as $S=\{1,2\}$ is short in all cases.

If one uses Theorem~\ref{thm:2}-\eqref{eq:thm22} one obtains
$$
\dint\limits_{U_S(\alpha_\varepsilon)}  i_S^* \,c_4c_5 = \sum_{J\in \mathcal{A}_{5,1} (\alpha_{\varepsilon})} (-1)^{\big\lvert J \cap \{3\} \big\rvert  + 1}= 2,
$$
since
\begin{align*}
\mathcal{A}_{5,1} (\alpha_{\varepsilon})=& \big\{J \subset\{3,4,5\}\mid \ell_J(\alpha_\varepsilon)>0 \,\, \text{and}\,\, 2 < \ell_J(\alpha_\varepsilon) \big\} \\ = &  \big\{ \{3,4\},\{3,5\}, \{4,5\}, \{3,4,5\} \big\}.
\end{align*}
These computations agree with the results in Example 4.7 of \cite{hp}. In fact, $U_S$ is homeomorphic to the blow-up of $\C \P^2$ at $3$ points and the intersection form on $H^2(U_S)$ with respect to the basis
$$
\left\{\frac{c_1+c_3+c_4+c_5}{2}, -\frac{c_1+c_3}{2},-\frac{c_1+c_4}{2},-\frac{c_1+c_5}{2} \right\}
$$
can be obtained from our results and is represented by the diagonal matrix $\text{Diag}(1,-1,-1,-1)$. Indeed, for example,
\begin{align*}
\bullet &  \dint\limits_{U_S(\alpha)}\hspace{-.2cm} i_S^* \left(\frac{c_1+c_3+c_4+c_5}{2}\right)^2 = \dint\limits_{U_S(\alpha)}\hspace{-.2cm}i_S^* c_1^2 + \frac{3}{2} \dint\limits_{U_S(\alpha)}\hspace{-.2cm}i_S^* c_4 \, c_5 = -2 +3 = 1,\\
\bullet & \dint\limits_{U_S(\alpha)} \hspace{-.2cm}i_S^* \left(\frac{c_1+c_3}{2}\right)^2 = \frac{1}{2}  \dint\limits_{U_S(\alpha)}\hspace{-.2cm}i_S^* \, c_1^2 + \frac{1}{2} \dint\limits_{U_S(\alpha)}\hspace{-.2cm}i_S^* \, c_1\,  c_3  = -1 + 0 = -1, \\
\bullet & \dint\limits_{U_S(\alpha)} \hspace{-.2cm}i_S^* \left(\frac{c_1+c_3+c_4+c_5}{2}\right) \left(-\frac{c_1+c_3}{2}\right) \\ & = - \dint\limits_{U_S(\alpha)}\hspace{-.2cm} i_S^* \left(\frac{c_1+ c_3}{2}\right)^2 -\frac{1}{2}\dint\limits_{U_S(\alpha)}\hspace{-.2cm} i_S^* \, c_1 \, c_5   -\frac{1}{2} \dint\limits_{U_S(\alpha)} \hspace{-.2cm}i_S^*\, c_3\, c_5=1-0-1=0 \\
\bullet &  \dint\limits_{U_S(\alpha)}\hspace{-.2cm}i_S^* \left(\frac{c_1+c_3}{2}\right)\left(\frac{c_1+c_4}{2}\right) \\ & = \frac{1}{4} \dint\limits_{U_S(\alpha)}\hspace{-.2cm} i_S^* \, c_1^2 +\frac{1}{2} \dint\limits_{U_S(\alpha)}\hspace{-.2cm}i_S^* \, c_1 \, c_4   +\frac{1}{4} \dint\limits_{U_S(\alpha)}\hspace{-.2cm}i_S^* \, c_3\, c_4= - \frac{1}{2}+ 0 + \frac{1}{2}= 0.
\end{align*}
\end{example}

\begin{example}
Let us consider the same hyperpolygon space $X(\alpha)$ as in the preceding example and compute the intersection numbers of the core component $U_S(\alpha)$ with $S=\{1,2,3\}$. By Claims~\ref{cl:1} and \ref{cl:2} it is enough to consider the following three integrals.
$$
\dint\limits_{U_S(\alpha)}i_S^* \, c_1^2, \quad  \dint\limits_{U_S(\alpha)}i_S^* \, c_1\, c_5 \quad \text{and} \quad  \dint\limits_{U_S(\alpha)}i_S^* \, c_4\, c_5.
$$

The value of the first one is 
$$
\dint\limits_{U_S(\alpha)}i_S^*\, c_1^2 = 1
$$
since 
$$i_S^* \, c_1^2 = i_S \, c_1\, c_2 = PD\big(U_S(\alpha) \cap W_1 \cap W_2\big) = PD\big(M(5,3,3)\big) = PD\big(\{\text{pt}\}\big).$$ 

This agrees with the value given by Theorem~\ref{thm:1}-\eqref{eq:I2} since $S$ is maximal short for $\alpha$.

For the second one we get
$$
\dint\limits_{U_S(\alpha_\varepsilon)}i_S^* \, c_1\, c_5 = \dint\limits_{U_S(1,1,3,6+\varepsilon)}i_S^* \, c_1 - \dint\limits_{U_S(1,1,3,\varepsilon)}i_S^* \, c_1 =-1 -0=-1
$$
since $S$ is not short for $(1,1,3,\varepsilon)$ and, in $U_S(1,1,3,6+\varepsilon)$, 
$$
i_S^* c_1= -PD\big(U_S(1,1,3,6+\varepsilon) \cap W_1\big) = -PD\big(\{\text{pt}\}\big).
$$
On the other hand, by Theorem~\ref{thm:2}-\eqref{eq:thm22} one has
$$
\dint\limits_{U_S(\alpha_\varepsilon)}i_S^* \, c_1\, c_5 = \sum_{J\in \mathcal{A}_{5,0}(\alpha_\varepsilon)} (-1)^{\big\lvert J \cap \{4\}\big\rvert +3+1}=-1,
$$
since
$$
\mathcal{A}_{5,0}(\alpha_\varepsilon)=\big\{ J \subset \{4,5\}\mid \ell_J(\alpha_\varepsilon) >0 \, \text{and} \, 5<\ell_J(\alpha_\varepsilon)\big\}=\big\{ \{4,5\}\big\}.
$$

Finally,
$$
\dint\limits_{U_S(\alpha_\varepsilon)}i_S^* \, c_4\, c_5 = \dint\limits_{U_S(1,1,3,6+\varepsilon)} \hspace{-.3cm} i_S^* \, c_4 + \dint\limits_{U_S(1,1,3,\varepsilon)}\hspace{-.3cm} i_S^* \, c_4 = - \dint\limits_{U_S(1,1,3,6+\varepsilon)}\hspace{-.3cm} i_S^* \, c_1 + 0 =1,
$$ 
where we used Claim~\ref{cl:3}, the fact that $S$ is not short for $(1,1,3,\varepsilon)$ and the fact that, in $U_S(1,1,3,6+\varepsilon)$, one has
$$
i_S^* c_1 = -PD\big(U_S(1,1,3,6+\varepsilon)\cap W_1\big) = -PD\big( \{\text{pt}\}\big). 
$$
By Theorem~\ref{thm:2}-\eqref{eq:tildea},
$$
\dint\limits_{U_S(\alpha_\varepsilon)}i_S^*\, c_4\, c_5=(-1)^4 \big\lvert \widetilde{\mathcal{A}}_{5,1}(\alpha_\varepsilon) \big\rvert = 1,
$$
since
$$
 \widetilde{\mathcal{A}}_{5,1}(\alpha_\varepsilon)=\big\{ J \subset\{4,5\}\mid \ell_J(\alpha_\varepsilon)>5 \big\} = \big\{\{4,5\}\big\}.
$$
These values agree with the fact that, since $S$ is a maximal short set for $\alpha$, the core component $U_S$ is $\C \P^2$ (cf. Proposition~\ref{Smaximal}). Indeed one can choose $c_1$ to be the generator of $H^2(U_S(\alpha))$.  
\end{example}

\section{Intersection numbers for PHBs}\label{sec:intersecPHB}

In this section we will use the isomorphism $\mathcal{F}: \mathcal{H}(\beta) \to X(\alpha)$ defined in \eqref{eq:mathcalf} to obtain explicit formulas for the intersection numbers of the nilpotent cone components of $\mathcal H(\beta)$. 
Consider the pull backs  $\mathcal{F}^* \widetilde{V}_i$ of $ \widetilde{V}_i$ as in the following diagram 
\begin{equation*}
\xymatrix{
\mathcal{F}^* \widetilde{V}_i \ar[r] \ar[d] & \widetilde{V}_i \ar[d]^{\pi}\\
\mathcal H(\beta) \ar[r]^{\mathcal F} & X(\alpha).
 }
\end{equation*}
In particular, 
$$
\mathcal F^* \widetilde{V}_i := \left\{ \big( [E, \Phi], (p,q)\big) \in \mathcal H(\beta) \times \widetilde{V}_i \mid \mathcal F \big( [E, \Phi] \big) = \pi \big( (p,q) \big)\right\}.
$$
Note that  the PHBs $ [E, \Phi] \in \mathcal H(\beta)$ for which there exists $(p,q)\in \widetilde{V}_i$ such that $\mathcal F \big( [E, \Phi] \big) = \pi \big( (p,q) \big)$ have parabolic structure at $x_i$ given by
$$
\mathbb{C}^2 = E_{x,1} \supset E_{x,2} = \langle (1,0)^t \rangle \supset 0
$$
and Higgs field with residue of the form 
\begin{equation}\label{eq:Resxi}
\text{Res}_x \Phi = \left(\begin{array}{cc} 0 & * \\ 0 & 0\end{array} \right).
\end{equation}
Indeed, since any $(p,q) \in \widetilde{V}_i$ satisfies 
$$(q_iq_i^* -p_i^* p_i)_0 = \left(\begin{array}{cc} t & 0 \\ 0 & -t\end{array} \right), \quad t>0,$$
writing, as usual, $p_i=(a_i, b_i)$ and $q_i= (c_i, d_i)^t$, one has
\begin{displaymath}
c_i \bar{d}_i - a_i \bar{b}_i =0 \quad \text{and} \quad \vert c_i\vert^2 - \vert d_i\vert^2 - \vert a_i\vert^2 + \vert b_i\vert^2 >0
\end{displaymath}
which, together with \eqref{complex1} gives $a_i = d_i =0$. Then, \eqref{eq:Resxi} follows from \eqref{eq:Ni}.

Consider the first Chern classes of these pull back bundles $c_1(\mathcal F^*  \widetilde{V}_i ) = \mathcal F^* c_i$ which we will also denote by $c_i$. Then it is clear that these classes generate $H^* (\mathcal H(\beta), \Q)$ as in the case of hyperpolygon spaces (cf. \cite{konno}, \cite{hauselp} and \cite{hp}). In particular, following Corollary 4.5 in  \cite{hauselp} we can explicitly describe the ring structure of the cohomology of $\mathcal{H}(\beta)$.
\begin{Theorem}\label{thm:H}
The cohomology ring $H\big(\mathcal H(\beta), \Q\big)$ is independent of $\beta$ and is isomorphic to 
$$
\Q[c_1, \ldots, c_n] / \big(\langle c_i^2-c_j^2\mid i,j \leq n \rangle + \langle \text{all monomials of degree $n-2$}\rangle \big).
$$
\end{Theorem}

Moreover one can reduce the computation of the intersection numbers of any nilpotent cone component $\mathcal U_{(0,S)} = \mathcal I(U_S (\beta))$ of $\mathcal H(\beta)$ to one of the following two cases.
\begin{align*}
\text{(I)} & \dint_{\mathcal U_{(0,S)}}  \iota_S^*\,  c_1^{n-3} = \dint_{U_S(\alpha)} i_S^*\,  c_1^{n-3},\\
\text{(II)} & \dint_{\mathcal U_{(0,S)}} \iota_S^*\, (c_1^k \, c_{n-l}\cdots c_n) = \dint_{U_S(\alpha)} i_S^*\, (c_1^k \, c_{n-l}\cdots c_n), \\ &  \text{with} \quad n-l > \vert S \vert \,\, \text{and}\,\, k=n-l-4,
\end{align*}
where $\iota_S:\mathcal{U}_{(0,S)} \to \mathcal{H}(\beta)$ is the inclusion map, and we used the fact that $\mathcal F \circ \iota_S \circ \mathcal{I} = i_S$. These integrals can then be computed using the formulas in Theorem~\ref{thm:1} and Theorem~\ref{thm:2}.

The ring structure of $H^*(\mathcal U_{(0,S)}, \Q)$ can also be obtained from the ring structure of $H^*(U_{S}, \Q)$ (presented in \cite{hp}), through the isomorphism of Theorem~\ref{isomorphism}. 
Explicitly, one obtains the following result. 
\begin{Theorem}\label{thm:US}
Consider the classes $b_i = -  \iota_S^* \left( \frac{c_1 + c_i}{2}\right)$ for $1=1, \ldots, n$. Then $H^*(\mathcal U_{(0,S)}, \Q)$ is isomorphic to $\Q[b_1, \ldots, b_n] / I_S $ where
$I_S$ is generated by the following four families of relations:
\begin{align*}
1)& \, b_1 - b_i \text{ for all }i \in S,\\
2)& \, b_j(b_1 - b_j) \text{ for all }j \in S^c,\\
3)& \, \prod_{j \in R} b_j \text{ for all }R \subseteq S^c \text{ such that }R \cup S \text{ is long},\\
4)& \, b_1^{\vert S \vert -2} \prod_{j \in L} (b_j - b_1) \text{ for all long subsets }L \subseteq S^c.\\
\end{align*}
\end{Theorem}
Note that relations $1)$ and $2)$ in Theorem~\ref{thm:US} are trivial consequences of Claims~\ref{cl:1} and \ref{cl:2} respectively.

\begin{example}
Let $S$ be a maximal $\alpha$-short set. Then 
$$\mathcal U_{(0,S)} \cong U_S(\alpha) \cong \C\P^{n-3}$$ 
(cf. Proposition~\ref{Smaximal}). 
This can be confirmed using Theorem~\ref{thm:US}. In fact, $R \cup S$ is long for any $R \subseteq S^c$, so $3)$ implies that $b_j=0$ for all $j \in S^c$, and then $2)$ is  trivially verified. Since by $1)$ we have $b_1 = b_i$ for all $i \in S$, we can chose $b_1$ to be the generator of $H^*(\mathcal U_{(0,S)}, \Q)$. Moreover, since $S$ is maximal, the only long subset of $S^c$ is $S^c$ itself, and so $I_S$ is generated by the unique condition $b_i^{n-2}=0$.
The cohomology ring of the nilpotent cone component $\mathcal U_{(0,S)} \cong \C\P^{n-3}$  is then 
$$ 
H^*(\mathcal U_{(0,S)}, \Q) \cong \Q[b_1] / \langle b_1^{n-2} \rangle \cong H^*(\C\P^{n-3}, \Q)
$$
as expected.
\end{example}


\begin{thebibliography}{99}

\bibitem{AG} J.~Agapito and L.~Godinho, \emph{Intersection numbers of polygon spaces},  Trans. Amer. Math. Soc.  {\bf 361}  (2009),  4969--4997.

\bibitem{AW} S.~Agnihotri and C.~Woodward, {\em Eigenvalues of products of unitary matrices and quantum Schubert calculus},  Math. Res. Lett.  {\bf 5}  (1998), 817--836.

\bibitem{A} M.~Atiyah, {\em Complex analytic connections in fiber bundles}, Trans. Amer. Math. Soc. {\bf 85} (1957), 181--207.

\bibitem{B} S.~Bauer, {\em Parabolic bundles, elliptic surfaces and $SU(2)$-representation spaces of genus zero Fuchsian groups}, Math. Ann. {\bf 290} (1991), 509--526. 

\bibitem{Bi} O.~Biquard, {\em Fibr\'{e}s paraboliques stables et connexions singuli\`{e}res plates}, Bill. Soc. Math. France {\bf 119} (1991), 231--257. 

\bibitem{BH} H.~U.~Boden and Y.~Hu, {\em Variations of moduli of parabolic bundles},  Math. Ann. {\bf 301} (1995), 539--559. 

\bibitem{BY} H.~U.~Boden and K.~Yokogawa, {\em Moduli spaces of parabolic Higgs bundles and parabolic K(D) pairs over smooth curves:I}, Internat. J. Math. {\bf 7} (1996), 573--598. 

\bibitem{bp} M.~Brion and C.~Procesi, \emph{Action d'un tore dans une varieté projective}, Operator algebras, unitary representations, enveloping algebras, and invariant theory (Paris, 1989), 509--539, Progr. Math., {\bf 92}, Birkh\"auser Boston, Boston, MA, 1990. 

\bibitem{del} T.~Delzant, \emph{Hamiltoniens p\'eriodiques et images convexes de l'application moment}, Bull. Soc. Math. France  {\bf 116}  (1988),  315--339. 

\bibitem{Rsurfaces} H.~Farkas and I.~Kra, \emph{Riemann surfaces}. Second edition. Graduate Texts in Mathematics, {\bf 71}, Springer-Verlag, New York, 1992.

\bibitem{F} T.~Frankel, {\em Fixed points and torsion on K\"{a}hler manifolds},  Ann. of Math. {\bf 70} (1959), 1--8. 

\bibitem{Fri} R.~Friedman, {\em Algebraic surfaces and holomorphic vector bundles}, Universitext. Springer-Verlag, New York, 1998.

\bibitem{GGM} O.~Garc\'{i}a-Prada, P.~B.~Gothen, and V.~Mu\~{n}oz, {\em Betti numbers of the moduli space of rank 3 parabolic Higgs bundles}, Mem. Amer. Math. Soc.  {\bf 187}  (2007). 

\bibitem{G} V.~Guillemin, \emph{Moment maps and combinatorial invariants of Hamiltonian $T^n$-spaces}, Progress in Mathematics, {\bf 122}, Birkh\"auser, Boston, 1994. 

\bibitem{gs} V.~Guillemin and S.~Sternberg, \emph{Birational equivalence in the symplectic category},  Invent. Math.  {\bf 97}  (1989),  485--522.

\bibitem{hp} M.~Harada and N.~Proudfoot, \emph{Hyperpolygon spaces and their cores}, Trans. Amer. Math. Soc.  {\bf 357}  (2005), 1445--1467. 

\bibitem{Ha} T.~Hausel  \emph{Compactification of moduli of Higgs bundles}, J. Reigne Angew. Math.  {\bf 503} (1998), 169--192.

\bibitem{hauselp} T.~Hausel  and N.~Proudfoot, \emph{Abelianization for hyperk\"{a}hler quotients}, Topology {\bf 44} (2005), 231--248.

\bibitem{hk} J.-C.~Hausmann and A.~Knutson, \emph{Polygon spaces and Grassmannians}, Enseign. Math. (2)  {\bf 43}  (1997), 173--198. 

\bibitem{H}  N.J.~Hitchin, {\em The self-duality equations on a Riemann surface}, Proc. London Math. Soc.  {\bf 55} (1987), 59--126.

\bibitem{H2} N.J.~Hitchin, {\em Stable bundles and integrable systems}, Duke Math. J. {\bf 54} (1987), 91--114.

\bibitem{hitchin} N.J.~Hitchin, \emph{Hyper-K\"ahler manifolds}, S\'eminaire Bourbaki,  Astérisque  {\bf 206}  (1992), 137--166.

\bibitem{Huy} D.~Huybrechts, {\em Fourier-Mukai transforms in algebraic geometry}, Oxford Mathematical Monographs. The Clarendon Press, Oxford University Press, Oxford, 2006.

\bibitem{km1} M.~Kapovich and J.~Millson, \emph{The symplectic geometry of polygons in Euclidean space},  J. Differential Geom.   {\bf 44}  (1996),  479--513. 

\bibitem{konno} H.~Konno, \emph{On the cohomology ring of the hyperK\"ahler analogue of the polygon spaces},  Integrable systems, topology, and physics (Tokyo, 2000),  129--149, Contemp. Math., {\bf 309}, Amer. Math. Soc., Providence, RI, 2002.

\bibitem{K} H.~Konno, {\em Construction of the moduli space of stable parabolic Higgs bundles on a Riemann surface}, J. Math. Soc. Japan {\bf 45} (1993), 253--276.

\bibitem{Ko} M.~Kontsevich, {\em Intersection theory on the moduli space of curves and the matrix Airy function}, Comm. Math. Phys. {\bf 147} (1992), 1--23.

\bibitem{L} H.~Lange, {\em Universal families of extensions}, J. Algebra {\bf 83} (1983), 101--112.   

\bibitem{m} A.~Mandini, \emph{The Duistermaat-Heckman formula and the cohomology of moduli spaces of polygons}, arXiv:0811.4062.

\bibitem{Mu} S.~Mukai, {\em Symplectic structure of the moduli space of sheaves on an abelian or K3 surface}, Invent. Math., {\bf 77}  (1984),  101--116. 

\bibitem{n2} H.~Nakajima, \emph{Varieties associated with quivers}, Representation theory of algebras and related topics (Mexico City, 1994),  139--157, CMS Conf. Proc., 19, Amer. Math. Soc., Providence, RI, 1996. 

\bibitem{n3} H.~Nakajima, \emph{Hyper-K\"{a}ler structures on moduli spaces of parabolic Higgs bundles on Riemann surfaces}, Moduli of vector bundles (Sanda, 1994; Kyoto, 1994), Lect. Notes Pure Appl. Math.  {\bf  179}, Dekker  (1996),  199--208.

\bibitem{n1} H.~Nakajima, \emph{Quiver varieties and Kac-Moody algebras}, Duke Math. J.  {\bf 91}  (1998),  515--560. 

\bibitem{S} C.~T.~Simpson, {\em Harmonic bundles on noncompact curves}, J. Amer.  Math. Soc. {\bf  3}  (1990),  713--770.

\bibitem{T} M.~Thaddeus, {\em Variation of moduli of parabolic Higgs bundles}, J. Reine Angew.  Math. {\bf  547}  (2002),  1--14.

\bibitem{W} K.~Walker, {\em Configuration spaces of linkages}, Undergraduate Thesis, Princeton (1985).

\bibitem{Wi} J.~Weitsman, {\em Geometry of the intersection ring of the moduli space of flat connections and the conjectures of Newstead and Witten}, Topology {\bf  37} (1998), 115--132.

\bibitem{W1} E.~Witten, {\em Two-dimensional gravity and intersection theory on moduli spaces}, Surveys in differential geometry (Cambridge, MA, 1990), 243--310, Lehigh Univ., Bethlehem, PA, 1991.

\bibitem{W2}E.~Witten, {\em On the Kontsevich model and other models of two-dimensional gravity}, Proceedings of the XXth International Conference on Deifferential Geometric Methods in Theoretical Physics, Vol. 1,2 (New York, 1991), 176--216, World Sci. Publ., River Edge, NJ, 1992.

\bibitem{Y2} K.~Yokogawa, {\em Compactification of moduli of parabolic sheaves and moduli of parabolic Higgs sheaves}, J. Math. Kyoto Univ. {\bf  33}  (1993),  451--504.

\bibitem{Y} K.~Yokogawa, {\em Infinitesimal deformation of parabolic Higgs sheaves}, Internat. J. Math. {\bf  6}  (1995),  125--148. 


\end{thebibliography}
\end{document}